\documentclass[11pt, a4paper]{article}

\usepackage{amsmath}
\usepackage{amsfonts}
\usepackage{amssymb}
\usepackage[english]{babel}
\usepackage{xcolor}
\usepackage{euscript}
\usepackage{eufrak}
\usepackage{hyperref}
\usepackage{graphicx}

\usepackage{dsfont}

\def\english{\selectlanguage{english}}

\providecommand\mathbb{\bf}
\newcommand\R{{\mathbb R}}

\newcommand\Z{{\mathbb Z}}
\newcommand\T{{\mathbb T}}

\newtheorem{thm}{Theorem}[section]
\newtheorem{lemma}{Lemma}[section]
\newtheorem{pro}{Proposition}[section]
\newtheorem{defi}{Definition}[section]

\newtheorem{remark}{Remark}[section]

\newcounter{Remark}

\renewcommand\theRemark{\arabic{Remark}}

\newcounter{steps}
\newenvironment{proof}[1][]{%
\par\medbreak\setcounter{steps}{0}
{\noindent\bfseries Proof#1. }} {\hfill\fbox{\ }\medbreak}

\newcounter{substeps}[steps]

\newcommand{\rot}[0]{
\mathrm{rot}}

\newcommand{\marmd}[0]{\mathrm{d}}

\newcommand{\calE}[0]{
\mathcal{E}}

\newcommand{\calR}[0]{
\mathcal{R}}

\newcommand{\calV}[0]{
\mathcal{V}}

\newcommand{\calF}[0]{
\mathcal{F}}

\newcommand{\intxt}[1]{
\int _{\R ^3} \!#1 \;\mathrm{d}x}

\newcommand{\intvt}[1]{
\int _{\R ^3} \!#1 \;\mathrm{d}v}

\newcommand{\intvxt}[1]{
\int _{\R ^3} \int _{\R ^3} \!#1 \;\mathrm{d}v \mathrm{d}x}

\newcommand{\Divx}[0]{
\mathrm{div}_x}

\newcommand{\Divv}[0]{
\mathrm{div}_v}

\newcommand{\calM}[0]{
\mathcal{M}}

\newcommand{\eps}[0]{
\varepsilon}

\newcommand{\calO}[0]{
\mathcal{O}}

\begin{document}
\english

\title{Long time behavior for collisional strongly magnetized plasma in three space dimensions}

\author{
Miha\"i BOSTAN \thanks{Aix Marseille Universit\'e, CNRS, Centrale Marseille, Institut de Math\'ematiques de Marseille, UMR 7373, Ch\^ateau Gombert 39 rue F. Joliot Curie, 13453 Marseille FRANCE. E-mail : {\tt mihai.bostan@univ-amu.fr}},
Anh-Tuan VU \thanks{Faculty of Mathematics and Informations, Hanoi University of Science and Technology, 1 Dai Co Viet, Hai Ba Trung, Hanoi, Vietnam. E-mail : {\tt tuan.vuanh1@hust.edu.vn}}
}

\date{(\today)}

\maketitle

\begin{abstract}
We consider the long time evolution of a population of charged particles, under strong magnetic fields and collision mechanisms. We derive a fluid model and justify the asymptotic behavior toward regular solutions of this regime. In three space dimensions, a constraint occurs along the parallel direction. For eliminating the corresponding Lagrange multiplier, we average along the magnetic lines.

\end{abstract}

\paragraph{Keywords:} Long time behavior, Strongly magnetized plasmas, Relative entropy.

\paragraph{AMS classification:} 35Q75, 78A35, 82D10.
\\
\\

\section{Introduction}
\label{Intro}
A plasma is a gas that is significantly ionized (through heating or photoionization), consisting of electrons and ions. One of the main applications in plasma physics concerns the energy production through thermo-nuclear fusion. The controlled fusion is achieved by magnetic confinement, where the plasma is confined into toroidal devices called Tokamaks, under the action of strong magnetic fields $\textbf{B}$. This confinement occurs because the radius of the circular motion of charged particles around the magnetic filed lines, known as the Larmor radius $r_L$, is proportional to the inverse of the strength $|\textbf{B}|$ of the magnetic field, i.e., $r_L = mv/|q\textbf{B}|$. Here, $m$ is the particle mass, $q$ is the particle charge and $v$ is the velocity in the plane perpendicular to the magnetic field lines. In Tokamaks, the plasma is heated to extremely high temperatures. As the collision frequency decreases with increasing temperature, fusion plasmas enter a weakly collisional regime. The mathematical modelling useful of such plasma confinement is based on a kinetic framework, which are mesoscopic descriptions for the electrons and ions, coupled with Poisson’s equation or Maxwell’s equation for the computation of the electrostatic or electromagntic fields, respectively.

In this paper, we study a class of kinetic-type models for fusion plasmas, typically represented by Vlasov-type equations that include self-consistent electric forces in strongly confined tokamak plasma, as well as in weakly collisional plasma. Direct simulations of the kinetic system of Vlasov equations are fairly expensive from a numerical point of view due to the high dimensionality of the phase space. Additionally, the presence of strong magnetic fields introduces a small time scale, namely the rotation period of particles around the magnetic lines (known as the cyclotronic period $T_c = 2\pi m/|q\textbf{B}|$). This requires  very small time steps for the numerical resolutions. In most industrial applications, it is necessary to reduce the dimensionality of the model, eliminating the unnecessary fast dynamics while retaining the complete low-frequency physics. Therefore, deriving an approximate model that is numerically less expensive is crucial. Large magnetic fields usually lead to the so-called gyro-kinetic or drift-kinetic limits (see \cite{LeeGyro1983, LittHam1981} for physics references and \cite{BosAsyAna, Bre2000} for mathematical results), which consists in the asymptotic behavior of the charged particles dynamics under slowly varying magnetic fields on the typical gyroradius length.

Here, we derive a new asymptotic model under the assumptions of a large magnetic field and a large-time asymptotic limit for the three dimensional Vlasov-Poisson-Fokker-Planck system. This problem is relevant from the physical point of view, particularly for modeling tokamak plasmas, and has been studied by many mathematicians. Let us summarize the previously known results on this topic. In the case of uniform constant magnetic fields, F. Golse and L. Saint-Raymond have  carefully studied this problem for the Vlasov-Poisson system in two dimensions \cite{GolSaintMag1999, Saint2002} (obtained when one restricted to the perpendicular dynamics). For non-constant magnetic fields (i.e., magnetic fields with constant direction but varying amplitude), P. Degond and F. Filbet, in \cite{DegFil16} have formally derived the asymptotic limit of the three dimensional Vlasov-Poisson system. Recently, in \cite{BosTuan}, M. Bostan and A-T. Vu have been provided a rigorous mathematical justification of the obtained asymptotic model for the Vlasov-Poisson-Fokker-Planck system in two dimensions.

In all the above references, the magnetic field is always assumed to have a constant direction (i.e., the curvature direction of the magnetic field lines is neglected), and the question remains fully open in the general case of the magnetic field. Our goal in this work is to provide a rigorous derivation for certain classes of non-trivial magnetic fields, using a method developed in \cite{BosTuan}. The main distinction of this paper, compared with \cite{BosTuan}, is that in the three dimensional setting with curved magnetic fields, the analysis is much more difficult due to the combination of parallel and perpendicular dynamics, as well as the curvature effects. This introduces additional constraints that must be addressed along the parallel direction of the magnetic field lines.

\section{Setting of the problem and main results}
\subsection{The physical model and Scaling}

We consider a one-species plasma interacting both through the action of the self-consistent electrostatic field and through collisions, in the presence of an external magnetic field. Let $f=f(t,x,v)$ be the density distribution function of charged particles of mass $m$, charge $q$, depending on time $t\in\R_+$, position $x\in\R^3$, and velocity $v\in\R^3$. The evolution of the plasma is governed by the Vlasov-Poisson-Fokker-Planck (referred to as
VPFP for simplicity of presentation) system, with external magnetic field:
\begin{equation}
\label{Non_scale_VPFP}
\partial_t f + v\cdot\nabla_x f + \dfrac{q}{m}\left( E + v\wedge \textbf{B}_{\text{ext}} \right)\cdot\nabla_v f = \mathcal{Q}(f),\,\, (t,x,v)\in\R_+\times\R^3\times\R^3.
\end{equation}
The self-consistent electric field $E(t,x) = -\nabla_x \Phi(t,x)$ and the electric potential $\Phi(t,x)$ will then satisfy the Poisson equation
\[
-\epsilon_0 \Delta_x \Phi(t,x)  =  q n(t,x)=q \int_{\R^3}{f(t,x,v)}\mathrm{d}v,
\]
where the constant $\epsilon_0$ represents the electric permittivity of the vacuum and $n(t,x)$ denotes the particle concentration. The external magnetic field $\textbf{B}_{\text{ext}}$ is of the form
\[
\textbf{B}_{\text{ext}}(x) = B(x) e(x),\quad \Divx(Be) =0, \quad |e(x) | =1, \quad x\in\R^3,
\]
for some scalar positive function $B(x)$ and some field of unitar vector $e(x)$. We assume that $B$, $e$ are smooth.
In the equation \eqref{Non_scale_VPFP}, the operator $\mathcal{Q}(f)$ is the linear Fokker-Planck operator acting on the velocities, which accounts for friction and diffusion effects, i.e.,
\[
\mathcal{Q}(f) = \dfrac{1}{\tau}\Divv{\{\sigma \nabla_v f + vf\}},
\]
where $\tau$ is the relaxation time and $\sigma$ is the velocity diffusion, see \cite{Cha1949} for the introduction of this operator, based on the principle of Brownian motion. The problem is supplemented with the following initial data:
\[
f(0,x,v) = f_{\mathrm{in}}(x,v),\,\,(x,v)\in\R^3\times\R^3.
\]

We are now interested in finding the asymptotic limit of the VPFP system \eqref{Non_scale_VPFP} describing the long-time dynamics of charged particles when they are submitted to a large magnetic field, in order to observe the drift phenomenon in the directions that run orthogonal to the magnetic field. Indeed, it is well known that the velocities of electric cross-field drift, the magnetic gradient drift, and the magnetic curvature drift are proportional to $\frac{1}{B}$ and, consequently, that
it is necessary to observe the drift movements over a large period of time that is proportional to $B$; see \cite{GolSaintMag1999, BosTuan, BosGuiCen3D, Bos2020}. Namely, we consider the following
\[
B^\eps (x)  = \dfrac{B(x)}{\eps},\quad f(t,x,v) = f^\eps(\tilde t,x,v),\quad t = \eps \tilde{t}.
\]
Here, $\eps >0$ is a small parameter, related to the ratio between the cyclotronic period  and the observation time scale. Hence, $\partial_t f = \eps \partial_{\tilde t} f^\eps$. Then, in \eqref{Non_scale_VPFP}, the
term $\partial_t$ is to be replaced by $\eps \partial_{\tilde t}$ or by $\eps \partial_t$ to simplify our notation; the VPFP system given by \eqref{Non_scale_VPFP} becomes:
\begin{equation}
\label{equ:VPFP-Scale}
\eps \partial_t f^\eps + v\cdot \nabla_x f^\eps + \dfrac{q}{m}E[f^\eps]\cdot \nabla_v f^\eps + \dfrac{q}{m}(v\wedge B^\eps(x)e(x))\cdot\nabla_v f^\eps = \dfrac{1}{\tau}\Divv (\sigma \nabla_v f^\eps + v f^\eps),
\end{equation}
\begin{equation}
\label{equ:PoissonEpsi}
E[f^\eps] = - \nabla_x \Phi[f^\eps],\quad -\epsilon_0\Delta_x\Phi[f^\eps] =q \, n[f^\eps]= q \int_{\R^3}{f^\eps}\mathrm{d}v.
\end{equation}
We complete the above system by applying the following initial condition,
\begin{equation}
\label{equ:Initial}
f^\eps(0,x,v) = f^\eps_{\mathrm{in}}(x,v),\,\, (x,v)\in \R^3\times\R^3.
\end{equation}

There are many works dealing with the existence and uniqueness of solutions to the VPFP system, in the three dimensional setting. For the existence of weak solutions for the VPFP problem \eqref{equ:VPFP-Scale}, \eqref{equ:PoissonEpsi}, and \eqref{equ:Initial} we refer to \cite{CarSol1995, Vic1991}. Existence and uniqueness results for strong solutions of the VPFP problem can be found in \cite{Bou1993, Bou1995, Deg1986, ODwVic1990, ReinWeckler1990}.

\subsection{Main results}
In this subsection, we focus on the main result concerning the hydrodynamic convergence. Specifically, we present a rigorous mathematical study of the formal analysis in the three-dimensional setting with general magnetic fields, assuming the existence of regular solutions to the limit model. For the case of non-trivial magnetic fields relevant to tokamaks, we enhance the result by demonstrating the well-posedness of the limit system.

We briefly formalize the study of the asymptotic behavior of the solutions, denoted by $(f^\eps)_{\eps>0}$, to the problem given by \eqref{equ:VPFP-Scale}, \eqref{equ:PoissonEpsi}, and \eqref{equ:Initial}, as $\eps$ tends to  zero. By analysing the balance of the free energy functional associated with the VPFP system, we formally conclude that the limit distribution function $f$ of the family $(f^\eps)_{\eps >0}$, as $\eps \searrow 0$, is an equilibrium of the form of a local Maxwellian distribution in velocity, parametrized by macroscopic quantities (particle concentration), for any $(t,x)\in \R_+ \times \R^3$, i.e.,
\[
f(t,x,v) = n(t,x)M(v)= n(t,x)\dfrac{e^{-|v|^2/2\sigma}}{(2\pi\sigma)^{3/2}},\,\,(t,x,v)\in \R_+ \times \R^3 \times \R^3.
\]
The concentration $n(t,x)$ satisfies the following transport equation with a constraint
\begin{equation}
\label{equ:gyro-kinetic}
\partial_t n + \Divx\left( \dfrac{n e}{\omega_c}\wedge \nabla_x k[n]\right) + Be\cdot\nabla_x p =0,\,\, (t,x)\in \R_+ \times \R^3,
\end{equation}
\begin{equation}
\label{equ:constraint}
B e\cdot \nabla_x k[n] = 0,\quad k[n] = \sigma (1+\ln n) + \dfrac{q}{m}\Phi[n],
\end{equation}
coupled to the Poisson equation
\begin{equation}
\label{equ:PoissonLimit}
E[n] = -\nabla_x \Phi[n],\quad -\epsilon_0 \Delta_x \Phi[n] = q n,
\end{equation}
with initial condition 
\[
n(0,x) = n_{\mathrm{in}} (x) =  \int_{\R^3}{f(0,x,v)}\mathrm{d}v.
\]
Here, $ p$ is thought as a Lagrange multiplier associated to the constraint \eqref{equ:constraint}, and $\omega_c(x) = \frac{qB(x)}{m} $ represents the cyclotron frequency. The formal derivation of the limit system is presented in more detail in Section \ref{ForDerLimMod}. Furthermore, Proposition \ref{Equiv_form} show that the concentration $n$ is advected along the electric cross-field drift $\frac{E\wedge e}{B}$, magnetic gradient drift $-\frac{\sigma \nabla \omega_c \wedge e}{\omega_c ^2}$, and magnetic curvature drift $-\frac{\sigma\partial_x e e \wedge e}{\omega_c}$. The limit model obtained in the three dimensional framework is significantly much more complex than the two-dimensional one (see \cite{BosTuan}), as it requires handling additional constraints in the parallel direction of the magnetic field lines. To eliminate the constraint and determine $p$, we derive an equivalent formulation of the limiting model via averaging along the characteristic flow generated by the magnetic field. Such methods have been employed in numerous papers previously, e.g. references \cite{BogMit61, BosTraEquSin, BosAsyAna, BosGuiCen3D, BosFinHauCRAS, BosFin16, BosSIAM09}. The result is a reduced partial differential equation (PDE) for the average of the particle concentration $n(t,x)$, as explained in Section \ref{RefLiMod}. Moreover, for a special case of magnetic field relevant to tokamaks, such as the cylindrical magnetic field (see \cite{Negu}), the reduced limit system is further simplified in Section \ref{AngVectFields}.

In the present work, the convergence of the VPFP system \eqref{equ:VPFP-Scale}, \eqref{equ:PoissonEpsi}, and \eqref{equ:Initial} towards the limit model \eqref{equ:gyro-kinetic}, \eqref{equ:constraint}, and \eqref{equ:PoissonLimit}  will be investigated by appealing to the relative entropy or modulated energy method, as introduced in \cite{Yau1991}. To the best of our knowledge, to date, there has been no result on the asymptotic
regime in the three dimensional setting with the non-trivial magnetic fields. This technique yields strong converges, provided that the solution of the limit system is regular and that  the initial data converge appropriately. Many asymptotic regimes were obtained using this technique, see \cite{Bre2000, BreMauPue2003, GolSaintQuas2003, PueSaint2004} for quasineutral regimes in collisionless plasma physics, \cite{Saint2003, BerVas2005} for hydrodynamic limits in gaz dynamics, \cite{GouJabVas2004} for fluid-particle interaction, \cite{BosGou08, Bos2007} for high electric or magnetic field limits in plasma physics. 

Before writing our main result, we define the modulated energy $\calE[n^\eps(t)|n(t)]$ by
\[
\calE[n^\eps(t)|n(t)] = \sigma \int_{\R^3}{n(t) h\left( \dfrac{n^\eps(t)}{n(t)}\right)}\mathrm{d}x + \dfrac{\epsilon_0}{2m}\int_{\R^3}{|\nabla_x \Phi[n^\eps] - \nabla_x \Phi[n]|^2}\mathrm{d}x,
\]
where $h: \R_+ \to \R_+$ is the convex function defined by $h(s) = s\ln s -s +1$, $s\in \R_+$. This quantity splits into the standard $L^2$ norm of the electric field plus the relative entropy between the particle density $n^\eps$ of \eqref{equ:VPFP-Scale}, \eqref{equ:PoissonEpsi}, and \eqref{equ:Initial} and the particle concentration $n$ of the limit model \eqref{equ:gyro-kinetic}, \eqref{equ:constraint}, and \eqref{equ:PoissonLimit}. For any nonnegative integer $k$ and $p \in [1, \infty]$, $W^{ k,p} = W^{ k,p} (\R^d )$ stands for the $k$-th order $L^p$ Sobolev space. $C_b^k$ stands for $k$ times continuously differentiable functions, whose partial derivatives, up to order $k$, are all bounded and $C^k ([0, T ]; E)$ is the set of $k$-times continuously differentiable functions from an interval $[0, T] \subset \R$ into a Banach space $E$. $L^p(0, T ; E)$ is the set of measurable functions from an interval $(0, T )$ to a Banach space $E$, whose $p$-th power of the $E$-norm is Lebesgue measurable. The main result of this paper is the following
\begin{thm}
\label{MainThm}
Let $B$ be a smooth magnetic field, such that $\inf_{x\in\R^3}B(x)=B_0 >0$. 
Assume that the initial particle densities $(f^\eps_{\mathrm{in}})_{\eps>0}$ satisfy $f^\eps_{\mathrm{in}}\geq 0$, $M_{\mathrm{in}}:=\sup_{\eps>0}M^\eps_{\mathrm{in}}<+\infty$, $U_{\mathrm{in}}:=\sup_{\eps>0}U^\eps_{\mathrm{in}}<+\infty$ where
\[
M^\eps _{\mathrm{in}} := \int_{\R^3}\int_{\R^3}{f^\eps _{\mathrm{in}} (x,v)}\mathrm{d}v\mathrm{d}x,\quad U^\eps _{\mathrm{in}} := \int_{\R^3}\int_{\R^3}{\dfrac{|v|^2}{2}f^\eps _{\mathrm{in}} (x,v)}\mathrm{d}v\mathrm{d}x + \dfrac{\epsilon_0}{2m}\int_{\R^3}{|\nabla_x \Phi[f^\eps _{\mathrm{in}}]|^2}\mathrm{d}x.
\]
Let $T>0$. We denote by $f^\eps$ the weak solutions of \eqref{equ:VPFP-Scale}, \eqref{equ:PoissonEpsi}, and \eqref{equ:Initial} with the initial data $f^\eps_\mathrm{in}$  on $[0,T]$ provided by Theorem \ref{Weaksol3D}. We assume that $n$ is a non-negative regular solution of \eqref{equ:gyro-kinetic}, \eqref{equ:constraint}, and \eqref{equ:PoissonLimit} on $[0,T]$ such that $W[n] = \frac{e}{\omega_c}\wedge \nabla_x k[n] + \frac{pBe}{n}$ belongs to $W^{1,\infty}((0,T)\times\R^3)$ with the initial concentration $n_\mathrm{in}$ verifying $n_{\mathrm{in}}\geq 0$, $n_{\mathrm{in}} \in L^{1}(\R^3)\cap W^{1,\infty}(\R^3)$, and $k[n_{\mathrm{in}}]\in \mathrm{ker}(Be\cdot\nabla_x)$. Then, we have the following inequality for $0<\eps<1$ and $t\leq T$:
\begin{equation*}
\begin{split}
\calE[n^\eps(t)|n(t)] +&  \sigma\int_{\R^3}\int_{\R^3}{n^\eps(t) M h\left(\dfrac{f^\eps(t)}{n^\eps(t) M} \right)}\mathrm{d}v\mathrm{d}x + \dfrac{1}{2\eps\tau}\int_0^t\int_{\R^3}\int_{\R^3}{\dfrac{|\sigma \nabla_v f^\eps + v f^\eps|^2}{f^\eps}}\mathrm{d}v\mathrm{d}x\mathrm{d}s\\
&\leq  \left[\calE[n^\eps_\mathrm{in}|n_\mathrm{in}] +  \sigma\int_{\R^3}\int_{\R^3}{n^\eps_\mathrm{in} M(v) h\left(\dfrac{f^\eps_\mathrm{in}}{n^\eps_\mathrm{in} M(v)} \right)}\mathrm{d}v\mathrm{d}x + C \sqrt{\eps}\right]e^{C t}.
\end{split}
\end{equation*}
where $n^\eps_{\mathrm{in}} = \intvt{f^\eps_{\mathrm{in}}}$, and $C>0$ is a constant that is independent of $\eps$.
In particular, if
\[
\lim_{\eps\searrow 0}\sigma \int_{\R^3}\int_{\R^3}{n^\eps_{\mathrm{in}}M(v) h\left(\dfrac{f^\eps_{\mathrm{in}}}{n^\eps_{\mathrm{in}}M(v)} \right)}\mathrm{d}v\mathrm{d}x =0,\quad \lim_{\eps\searrow 0}\calE[n^\eps_{\mathrm{in}}|n_{\mathrm{in}}] =0 ,
\]
then we obtain 
\[
\lim_{\eps\searrow 0} \sup_{0 \leq t \leq T} \sigma \int_{\R^3}\int_{\R^3}{n^\eps(t)M(v)h\left(\dfrac{f^\eps}{n^\eps M} \right)}\mathrm{d}v\mathrm{d}x = 0,\quad \lim_{\eps\searrow 0} \sup_{0 \leq t \leq T}\calE[n^\eps(t)|n(t)] =0,
\]
\[
\lim_{\eps\searrow 0}\dfrac{1}{\eps \tau}\int_0^T\int_{\R^3}\int_{\R^3}{\dfrac{|\sigma \nabla_v f^\eps + v f^\eps |^2}{f^\eps}} \mathrm{d}v\mathrm{d}x\mathrm{d}t =0.
\]
Moreover, we have the convergences $\lim_{\eps\searrow 0} f^\eps = nM$ in $L^\infty (0,T;L^1(\R^3\times \R^3))$, and $\lim_{\eps\searrow 0} \nabla_x \Phi[f^\eps] = \nabla_x\Phi[n]$ in $L^\infty(0,T;L^2(\R^3))$.
\end{thm}

The main result of this paper, Theorem \ref{MainThm}, relies on the existence of regular solutions to the limit system. However, obtaining such solutions is challenging, as the limit equation does not fall within the scope of standard PDEs theory. In this work, we focus exclusively on the cylindrical case of the magnetic field. In Section \ref{Example}, we establish the existence and uniqueness of local-in-time regular solutions to the simplified limit system derived in Section \ref{AngVectFields}.  The analysis of the limit system in the other relevant cases, such as the toroidal case (see \cite{Lutz2013}), is an interesting problem and is left for future work.

\begin{remark}
The assumption on the external magnetic field $B$ in Theorem \ref{MainThm} is crucial in the establishment of the existence of a regular solution to the limit model, as well as the regularity properties of certain quantities associated with the solution. For an example, we refer the reader to Section \ref{Example}.
\end{remark}

The remainder of the paper is organized as follows. In Section \ref{Prelim}, we establish some a priori estimates for the three-dimensional VPFP system. In Section \ref{ForDerLimMod}, we formally derive the asymptotic model using the Hilbert expansion. In Section \ref{Conve}, we rigorously prove the convergence towards the asymptotic model, assuming the regularity of the solution to the limit problem. Section \ref{RefLiMod} is dedicated to finding an equivalent model by eliminating the Lagrange multiplier. In Section \ref{AngVectFields},  we further simplify the limit model for the cylindrical case of magnetic fields. Finally, in Section \ref{Example}, we investigate the well-posedness of the simplified limit model derived in Section \ref{AngVectFields}.

\section{Preliminaires}
\label{Prelim}
We begin by introducing the concept of a weak solution to the VPFP system \eqref{equ:VPFP-Scale}, \eqref{equ:PoissonEpsi}, and \eqref{equ:Initial} for any fixed $\eps>0$. 
\begin{defi}
\label{Weaksol3D}
Let $T>0$. Given $f^\eps_{\mathrm{in}} \geq 0$ and $f^\eps_{\mathrm{in}}\in L^1(\R^3\times\R^3)$, we will say that $f^\eps$ is a weak solution of \eqref{equ:VPFP-Scale}-\eqref{equ:Initial} on the time interval $[0,T]$ if and only if the following conditions are satisfied:\\
(i) $f^\eps \geq 0,\,\, f^\eps\in L^\infty(0,T; L^1\cap L^\infty(\R^3\times\R^3))$,\\
(ii) for any $\psi\in C^\infty_c([0,T[\times\R^3\times\R^3)$, we have
\begin{align*}
\int_{0}^{T}\int_{\R^3}\int_{\R^3} f^\eps\left[ \eps\dfrac{\partial \psi}{\partial t} + v\cdot\nabla_x \psi + \dfrac{q}{m}\left( E[f^\eps] + v \wedge \dfrac{Be}{\eps} \right)\cdot\nabla_v \psi\right]\mathrm{d}v\mathrm{d}x\mathrm{d}t \\
+ \int_{0}^{T}\int_{\R^3}\int_{\R^3}\dfrac{1}{\tau} f^\eps(\sigma \Delta_v \psi - v\cdot\nabla_v \psi) \mathrm{d}v\mathrm{d}x\mathrm{d}t + \int_{\R^3}\int_{\R^3}\eps f^\eps_{\mathrm{in}}(x,v)\psi(0,x,v)\mathrm{d}v\mathrm{d}x =0.
\end{align*}
\end{defi}

The global-in-time existence of a weak solution for the nonlinear VPFP system  \eqref{equ:VPFP-Scale}-\eqref{equ:Initial} follows almost the same argument as presented in \cite{CarChoiJung2021, ChoiJeong2023}. We state the existence theorem for the solution and do not provide further details here.
\begin{thm}
\label{Thm:Weaksol}
Let $B\in L^\infty(\R^3)$. Suppose that the initial data $f^\eps_\mathrm{in}$ satisfies
\[
f^\eps _\mathrm{in} \geq 0,\,\, f^\eps _\mathrm{in} \in L^1\cap L^\infty (\R^3\times\R^3),\,\, (|x|^2 + |v|^2 + \Phi[f^\eps _\mathrm{in}])f^\eps _\mathrm{in} \in L^1(\R^3\times\R^3).
\]
Then, for any $T>0$, there exists a global weak solution of the system \eqref{equ:VPFP-Scale}-\eqref{equ:Initial} in the sense of Definition \ref{Weaksol3D} satisfying:
\[
f^\eps \in L^\infty(0,T;L^1\cap L^\infty (\R^3\times\R^3))\,\,\text{and}\,\, (|x|^2 + |v|^2 + \Phi[f^\eps])f^\eps \in L^\infty(0,T;L^1(\R^3\times\R^3)).
\]
\end{thm}

The asymptotic behavior of the VPFP equation \eqref{equ:VPFP-Scale} when $\eps$ becomes small comes from the balance of the free energy functional
\[
\calE[f^\eps] = \int_{\R^3}\int_{\R^3}{\left( \sigma f^\eps \ln f^\eps + f^\eps \dfrac{|v|^2}{2} \right)}\mathrm{d}v\mathrm{d}x + \dfrac{\epsilon_0}{2m}\int_{\R^3}{|E[f^\eps]|^2}\mathrm{d}x.
\]
Multiplying the left hand side of \eqref{equ:VPFP-Scale} by $\sigma(1+ \ln f^\eps) + \frac{|v|^2}{2}$ and integrating with respect to $(x,v)\in \R^3\times\R^3$ yield
\begin{align}
\label{equ:BalanceEner}
&\int_{\R^3}\int_{\R^3}{\left[ \eps \partial_t f^\eps + v\cdot \nabla_x f^\eps + \dfrac{q}{m}\left( E[f^\eps]+ v\wedge \dfrac{Be}{\eps} \right)\cdot \nabla_v f^\eps \right]\left[\sigma(1+ \ln f^\eps) + \dfrac{|v|^2}{2} \right]}\mathrm{d}v\mathrm{d}x\nonumber\\
&= \eps \dfrac{\mathrm{d}}{\mathrm{d}t} \int_{\R^3}\int_{\R^3}{\left( \sigma f^\eps \ln f^\eps + f^\eps \dfrac{|v|^2}{2} \right)}\mathrm{d}v\mathrm{d}x + \dfrac{q}{m} \int_{\R^3}{\nabla_x\Phi[f^\eps]\cdot \left( \intvt{v f^\eps}\right)}\mathrm{d}x.
\end{align} 
Thanks to the continuty equation
$
\eps \partial_t n[f^\eps] + \Divx \int_{\R^3}{v f^\eps}\mathrm{d}v =0,
$
we write
\begin{align}
\label{equ:EvoluElect}
\dfrac{q}{m} \int_{\R^3}{\nabla_x\Phi[f^\eps]\cdot \left( \intvt{v f^\eps}\right)}\mathrm{d}x =\eps \dfrac{q}{m}\int_{\R^3}{\Phi[f^\eps]\partial_t n[f^\eps]}\mathrm{d}x.
\end{align}
Multiplying the right hand side of \eqref{equ:VPFP-Scale} by $\sigma(1+ \ln f^\eps) + \frac{|v|^2}{2}$ and then integrating with respect to $(x,v)\in \R^3\times\R^3$ imply
\begin{align}
\label{equ:Dissipation}
\int_{\R^3}\int_{\R^3}{\mathcal{Q}(f^\eps)\left[ \sigma(1+ \ln f^\eps) + \dfrac{|v|^2}{2} \right]}\mathrm{d}v\mathrm{d}x = -\dfrac{1}{\tau} \int_{\R^3}\int_{\R^3}{\dfrac{|\sigma M\nabla_v (f^\eps/M) |^2}{f^\eps}}\mathrm{d}v\mathrm{d}x,
\end{align}
where $M$ stands for the Maxwellian equilibrium $M(v) = (2\pi\sigma)^{-3/2}\exp\left(-\frac{|v|^2}{2\sigma}\right)$, $v\in\R^3$. Combining \eqref{equ:BalanceEner}, \eqref{equ:EvoluElect}, and \eqref{equ:Dissipation} leads to the balance
\begin{align}
\label{equ:EquFreeEne}
&\eps \dfrac{\mathrm{d}}{\mathrm{d}t} \left[ \int_{\R^3}\int_{\R^3}{\left( \sigma f^\eps \ln f^\eps + f^\eps \dfrac{|v|^2}{2} \right)}\mathrm{d}v\mathrm{d}x + \dfrac{\epsilon_0}{2m}\int_{\R^3}{|\nabla_x \Phi[f^\eps]|^2}\mathrm{d}x\right]\\
&+ \dfrac{1}{\tau} \int_{\R^3}\int_{\R^3}{\dfrac{|\sigma M\nabla_v (f^\eps/M) |^2}{f^\eps}}\mathrm{d}v\mathrm{d}x =0, \nonumber
\end{align}
or equivalently
\[
\eps \calE[f^\eps(t)] + \dfrac{1}{\tau} \int_0^t\int_{\R^3}\int_{\R^3}{\dfrac{|\sigma M\nabla_v (f^\eps/M) |^2}{f^\eps}}\mathrm{d}v\mathrm{d}x\mathrm{d}s = \eps \calE[f^\eps(0)].
\]
Note that weak solutions may only satisfy an inequality in the above relation, which is sufficient for our purposes. At least formally, we deduce that $f^\eps = f + \calO(\eps)$, as $\eps \searrow 0$, where the leading order density $f$ satisfies
\[
\dfrac{1}{\tau}\int_{\R^3}\int_{\R^3}{\dfrac{|\sigma M\nabla_v (f/M) |^2}{f}}\mathrm{d}v\mathrm{d}x =0,\,\, t\in \R_+ .
\]
Therefore, we have $f(t,x,v) = n(t,x)M(v), (t,x,v)\in \R_+ \times \R^3\times\R^3$ and it remains to determine the time evolution of the concentration $n = \intvt{f}$.

We establish uniform bounds for the kinetic energy.
\begin{lemma}
\label{KinEne}
Let $T>0$. Assume that the initial particle densities $(f^\eps_{\mathrm{in}})$ satisfy $f^\eps _{\mathrm{in}} \geq 0$, $M_{\mathrm{in}}:= \sup_{\eps >0} M^\eps _{\mathrm{in}} < +\infty$, $U_{\mathrm{in}} := \sup_{\eps >0} U^\eps _{\mathrm{in}} < +\infty$, where for any $\eps >0$
\[
M^\eps _{\mathrm{in}} := \int_{\R^3}\int_{\R^3}{f^\eps _{\mathrm{in}} (x,v)}\mathrm{d}v\mathrm{d}x,\quad U^\eps _{\mathrm{in}} := \int_{\R^3}\int_{\R^3}{\dfrac{|v|^2}{2}f^\eps _{\mathrm{in}} (x,v)}\mathrm{d}v\mathrm{d}x + \dfrac{\epsilon_0}{2m}\int_{\R^3}{|\nabla_x \Phi[f^\eps _{\mathrm{in}}]|^2}\mathrm{d}x.
\]
We assume that $(f^\eps)_{\eps>0}$ are weak solutions of \eqref{equ:VPFP-Scale}, \eqref{equ:PoissonEpsi} and \eqref{equ:Initial}. Then we have 
\[
\eps \sup_{0\leq t\leq T} \left\{ \int_{\R^3}\int_{\R^3}{\dfrac{|v|^2}{2}f^\eps(t,x,v)}\mathrm{d}v\mathrm{d}x + \dfrac{\epsilon_0}{2m}\int_{\R^3}{|\nabla_x \Phi[f^\eps]|^2} \right\}\mathrm{d}x \leq \eps U_{\mathrm{in}} + \dfrac{3\sigma}{\tau}TM_{\mathrm{in}}
\]
and
\[
\dfrac{1}{\tau} \int_0^T\int_{\R^3}\int_{\R^3}{|v|^2 f^\eps(t,x,v)}\mathrm{d}v\mathrm{d}x\mathrm{d}t \leq \eps U_{\mathrm{in}} + \dfrac{3\sigma}{\tau}TM_{\mathrm{in}}.
\]
\end{lemma}
\begin{proof}
We will establish the results  for smooth solutions, and we observe that the same conclusions hold true in the framework of weak solutions by combining the formal arguments to be exposed here with the choice of an appropriate sequence of test functions in Definition \ref{Weaksol3D} for every studied property (cf. \cite{BoniCarSoler1997, BouDol1995}).
Multiplying \eqref{equ:VPFP-Scale} by $\frac{|v|^2}{2}$ and integrating with respect to $(x,v)\in\R^3\times\R^3$ yield
\[
\eps \dfrac{\mathrm{d}}{\mathrm{d}t} \left\{ \int_{\R^3}\int_{\R^3}{\dfrac{|v|^2}{2}f^\eps(t,x,v)}\mathrm{d}v\mathrm{d}x + \dfrac{\epsilon_0}{2m}\int_{\R^3}{|\nabla_x \Phi[f^\eps]|^2}\mathrm{d}x \right\} = \dfrac{3\sigma}{\tau}M^\eps _{\mathrm{in}} - \dfrac{1}{\tau} \int_{\R^3}\int_{\R^3}{|v|^2 f^\eps}\mathrm{d}v\mathrm{d}x,
\]
and therefore we obtain
\begin{align*}
&\eps  \left\{ \int_{\R^3}\int_{\R^3}{\dfrac{|v|^2}{2}f^\eps(t,x,v)}\mathrm{d}v\mathrm{d}x + \dfrac{\epsilon_0}{2m}\int_{\R^3}{|\nabla_x \Phi[f^\eps]|^2}\mathrm{d}x \right\} + \dfrac{1}{\tau} \int_0^t\int_{\R^3}\int_{\R^3}{|v|^2 f^\eps}\mathrm{d}v\mathrm{d}x\mathrm{d}s \\
&= \eps U^\eps _{\mathrm{in}}+  \dfrac{3\sigma}{\tau} t M^\eps _{\mathrm{in}},
\end{align*}
which yields the results.
\end{proof}
\section{Formal derivation of the limit model}
\label{ForDerLimMod}
This section is devoted to deriving the limit model for \eqref{equ:VPFP-Scale}, \eqref{equ:PoissonEpsi}, and \eqref{equ:Initial} when $\eps$ becomes
very small, using the properties of the averaged dominant transport operator. At a formal level, we initiate our analysis with a Hilbert expansion 
\[f^\eps = f + \eps f_1 + \eps^2 f_2 + ....\]
 Plugging the above ansatz into the kinetic equation \eqref{equ:VPFP-Scale} and identifying the contributions for each power of $\eps$ leads to
\begin{equation}
\label{equ:Order0}
\dfrac{q}{m}(v\wedge Be)\cdot\nabla_v f =0.
\end{equation}
\begin{equation}
\label{equ:Order1}
v\cdot \nabla_x f + \dfrac{q}{m}E[f]\cdot \nabla_v f + \dfrac{q}{m}(v\wedge Be)\cdot \nabla_v f_1 = \mathcal{Q}(f).
\end{equation}
\begin{equation}
\label{equ:Order2}
\partial_t f + v\cdot \nabla_x f_1 + \dfrac{q}{m}E[f_1]\cdot\nabla_v f + \dfrac{q}{m}E[f]\cdot\nabla_v f_1 + \dfrac{q}{m}(v\wedge Be)\cdot \nabla_v f_2 = \mathcal{Q}(f_1).
\end{equation}
Multiplying \eqref{equ:Order1} by $\sigma(1+ \ln f) + \frac{|v|^2}{2}$ and integrating with respect to $(x,v)\in \R^3\times\R^3$ yields
\begin{align}
\label{equ:EquBalancef1}
&\int_{\R^3}\int_{\R^3}{\left( v\cdot\nabla_x + \dfrac{q}{m}E[f]\cdot\nabla_v \right)\left(\sigma f \ln f + f\dfrac{|v|^2}{2} \right)}\mathrm{d}v\mathrm{d}x + \dfrac{1}{\tau}\int_{\R^3}\int_{\R^3}{\dfrac{|\sigma M \nabla_v(f/M)|^2}{f}}\mathrm{d}v\mathrm{d}x\nonumber\\
&= \intvxt{\dfrac{q}{m}E[f]\cdot vf} + \intvxt{f_1 \dfrac{q}{m}(v\wedge Be)\cdot\dfrac{\sigma \nabla_v f}{f}}.
\end{align}
Integrating \eqref{equ:Order1} with respect to $v\in\R^3$, we deduce that $\Divx \intvt{vf} =0$ and therefore we have
\[
\int_{\R^3}\int_{\R^3}{\dfrac{q}{m}E[f]\cdot vf}\mathrm{d}v\mathrm{d}x = - \dfrac{q}{m}\int_{\R^3}{\nabla_x \Phi[f]\cdot \left( \int_{\R^3}{v f}\mathrm{d}v\right)}\mathrm{d}x = 0.
\]
Using also \eqref{equ:Order0}, the last contribution in the right hand side of 
\eqref{equ:EquBalancef1} cancels, and therefore we obtain
\[
\dfrac{1}{\tau} \int_{\R^3}\int_{\R^3}{\dfrac{|\sigma M\nabla_v (f/M)|^2}{f}}\mathrm{d}v\mathrm{d}x =0,\,\, t\in \R_+ ,
\]
saying that $f =nM$, for some function $n=n(t,x)$ to be determined. In that case, the constraint \eqref{equ:Order0} is satisfied and \eqref{equ:Order1} becomes
\begin{equation*}
v\cdot\nabla_x f + \dfrac{q}{m}E[f]\cdot\nabla_v f \in \mathrm{Range}((v\wedge e(x))\cdot\nabla_v),\,\, x\in \R^3.
\end{equation*}
For any $e \in \mathbb{S}^2 $, we denote by $\mathcal{R} (\theta,e)$ the rotation of angle $\theta$ around the axis $e$
\begin{equation*}
\mathcal{R}(\theta, e )v = \cos\theta (I_3 -e\otimes e)v - \sin\theta (v\wedge e) + (v\cdot e)e,\,\, v\in\R^3.
\end{equation*}
The characteristic flow of the field $(v\wedge e)\cdot\nabla_v$
\[
\dfrac{\mathrm{d}\calV}{\mathrm{d}\theta} = \calV(\theta;v)\wedge e, \quad \calV(0;v)= v,
\]
is given by
\[
\calV(\theta;v) = \mathcal{R}(-\theta,e)v = \cos\theta (I_3 -e\otimes e)v + \sin\theta (v\wedge e) + (v\cdot e)e,\,\, (\theta,v)\in \R\times\R^3.
\]
For any function $g(v) = (v\wedge e)\cdot\nabla_v h$ in the range of the operator $(v\wedge e)\cdot\nabla_v$, we have
\[
g(\calV(\theta;v)) = \dfrac{\mathrm{d}}{\mathrm{d}\theta}h(\calV(\theta;v)),\,\, (\theta,v)\in \R\times\R^3,
\]
and by the periodicity of the flow we obtain
\[
\dfrac{1}{2\pi}\int_{0}^{2\pi} g(\calV(\theta;v)) \mathrm{d}\theta =0, \,\, v\in\R^3.
\]
Therefore, for any $x\in\R^3$, the average along the characteristic flow with respect to $(v\wedge e(x))\cdot\nabla_v$ of the function $v\cdot\nabla_x f + \frac{q}{m}E[f]\cdot\nabla_v f$ vanishes. But
\[
v\cdot\nabla_x f + \frac{q}{m}E[f]\cdot\nabla_v f = (v \cdot \nabla_x n)M - \dfrac{q}{m}(E[f]\cdot v)n\dfrac{M}{\sigma} = \dfrac{n}{\sigma}Mv\cdot\nabla_x(\sigma\ln n + \dfrac{q}{m}\Phi[f]),
\]
and since
\[
\dfrac{1}{2\pi} \int_{0}^{2\pi} M(\calV(\theta;v))\calV(\theta;v)\mathrm{d}\theta = M(v)(v\cdot e)e,
\]
finally we obtain the constraint
\[
e\cdot \nabla_x k[n] = 0,\quad k[n] = \sigma(1+\ln n) + \dfrac{q}{m}\Phi[n],\,\, x\in\R^3.
\]
Here the potential $\Phi = \Phi[n]$ writes
\[
\Phi[n(t)](x) = \dfrac{q}{4\pi\epsilon_0}\int_{\R^3}\dfrac{n(t,x')}{|x-x'|}\mathrm{d}x',\,\, (t,x)\in \R_+ \times\R^3.
\]
The time evolution for the concentration $n$ comes by integrating \eqref{equ:Order2} with respect to $v\in\R^3$
\begin{equation}
\label{equ:EquContinum}
\partial_t n + \Divx\int_{\R^3}{v f_1}\mathrm{d}v =0.
\end{equation}
Multiplying \eqref{equ:Order1} by $v$ and integrating with respect to $v\in\R^3$ we obtain
\[
\Divx \int_{\R^3}{v\otimes v f}\mathrm{d}v - n\dfrac{q}{m}E[f] - \dfrac{qB}{m}\int_{\R^3}{v f_1\wedge e}\mathrm{d}v =0.
\]
Since $f$ is a Maxwellian equilibrium, we have $\intvt{v\otimes v f} = \sigma n I_3$ and the previous equality becomes
\[
\omega_c \int_{\R^3}{v f_1}\mathrm{d}v \wedge e = \sigma \nabla_x n - n\dfrac{q}{m}E[f],
\]
or equivalently
\begin{align*}
\omega_c (I_3 - e\otimes e)\int_{\R^3}{v f_1}\mathrm{d}v = ne\wedge \nabla_x k[n].
\end{align*}
The divergence with respect to $x$ of $\int_{\R^3}{vf_1}\mathrm{d}v$ writes
\begin{align*}
\Divx \int_{\R^3}{v f_1}\mathrm{d}v &= \Divx \left[ (I_3 - e\otimes e)\int_{\R^3}{v f_1}\mathrm{d}v \right] + \Divx \left[ e\otimes e\intvt{v f_1} \right]\\
&= \Divx\left( \dfrac{n e}{\omega_c}\wedge \nabla_x k[n]\right) +  Be\cdot\nabla_x \int_{\R^3}{\dfrac{(v\cdot e)f_1}{B}}\mathrm{d}v.
\end{align*}
Coming back in \eqref{equ:EquContinum}, we obtain that the limiting concentration $n$ satisfies equation \eqref{equ:gyro-kinetic}, 
for some function $p$ such that the constraint \eqref{equ:constraint} holds true.
The limit model involves a Lagrange multiplier $p$, associated to the constraint \eqref{equ:constraint}. One of the main difficulty is that the unknown is the concentration $n$, whereas the constraint relies on $k[n]$. 
\begin{remark}
 In the absence of magnetic fields, the constraint \eqref{equ:constraint} is removed, and equation \eqref{equ:Order1} simplifies to
\[
v\cdot\nabla_x f - \dfrac{q}{m}\nabla_x \Phi[f]\cdot\nabla_v f =0.
\]
Substituting $f(t,x,v) = n(t,x)M(v)$ into the previous equation, direct computations yield 
$
\nabla_x k[n] = 0,
$
which implies that the concentration $n(t,x)$ takes the form
\begin{equation}
\label{equ:Boltz-Gibb}
n(t,x) = Z(t) e^{-\frac{q}{m\sigma}\Phi[n(t)](x)},
\end{equation}
which is the so-called Boltzmann-Gibbs relation, relating the electron density to the electric potential, cf. \cite{BarGolToanSen16}. When the magnetic field is uniform, i.e., $Be = (0,0,1)^t$, the constraint \eqref{equ:constraint} becomes
$
\partial_{x_3} k[n] =0, 
$
which implies that the concentration $n(t,x)$ can be written as
\begin{equation}
\label{equ:ReduBoltz-Gibb}
n(t,x) = N(t,x_{\perp})\dfrac{e^{-\frac{q}{m\sigma}\Phi[n(t)](x)}}{\int_{\R}e^{-\frac{q}{m\sigma}\Phi[n(t)](x_\perp, x_3)}\mathrm{d}x_3},
\end{equation}
where $ x= (x_\perp, x_3)\in \R^2\times \R$, cf. \cite{HerRod2019, Negu}. It is worth noting that our limit model \eqref{equ:gyro-kinetic} is consistent with the limit model of the electron distribution function obtained in \cite{HerRod2019}. Indeed, in the case of  uniform magnetic fields, the limit equation \eqref{equ:gyro-kinetic} becomes
\[
\partial_t n + \Divx \left( n E \wedge e \right) + \partial_{x_3}p =0.
\]
Integrating in $x_3$ to eliminate the Lagrange multiplier $p$ and using \eqref{equ:ReduBoltz-Gibb}, we obtain 
\[
\partial_t N(t,x_{\perp}) + \mathrm{div}_{x_\perp} \left( N(t,x_\perp){^\perp} \nabla_{x_\perp}\tilde{\Phi}  \right)  =0,
\]
where $\tilde{\Phi}: \R_+ \times \R^2 \to \R$ is an $x_3$ averaged of $\Phi[n]$
\[
\tilde{\Phi}(t,x_\perp) = \dfrac{m\sigma}{q} \ln \left( \int_{\R} e^{-\frac{q}{m\sigma}\Phi[n(t)](x_\perp, x_3)}\mathrm{d}x_3 \right),
\]
which is exactly the limit model introduced in \cite{HerRod2019}.
\end{remark}

We now provide  some fundamental properties satisfied by the asymptotic model \eqref{equ:gyro-kinetic}, \eqref{equ:constraint}, and \eqref{equ:PoissonLimit}. Formally, we have the balances
\begin{pro}
\label{BalLiMod}
Any non-negative regular solution of the limit model \eqref{equ:gyro-kinetic}, \eqref{equ:constraint}, and \eqref{equ:PoissonLimit} verifies the mass and free energy conservations
\[
\dfrac{\mathrm{d}}{\mathrm{d}t}\int_{\R^3}{n(t,x)}\mathrm{d}x =0,\quad \dfrac{\mathrm{d}}{\mathrm{d}t}\int_{\R^3}{\left\{ \sigma n\ln n + \dfrac{\epsilon_0}{2m}|\nabla_x\Phi[n]|^2\right\}}\mathrm{d}x =0.
\]
\end{pro}
\begin{proof}
Clearly we have the total mass conservation. For the energy conservation, we multiply \eqref{equ:gyro-kinetic} by $k[n]$ and integrate with respect to $x\in\R^3$, observing that
\[
\int_{\R^3}{\partial_t n k[n]}\mathrm{d}x = \dfrac{\mathrm{d}}{\mathrm{d}t}\int_{\R^3}{\left\{ \sigma n\ln n + \dfrac{\epsilon_0}{2m}|\nabla_x\Phi[n]|^2\right\}}\mathrm{d}x,
\]
\[
\int_{\R^3}{\Divx\left(\dfrac{ne}{\omega_c}\wedge \nabla_x k[n] \right)k[n]}\mathrm{d}x = - \int_{\R^3}{\left(\dfrac{ne}{\omega_c}\wedge \nabla_x k[n] \right)\cdot\nabla_x k[n]}\mathrm{d}x =0,
\]
\[
\int_{\R^3}{Be\cdot\nabla_xp k[n]}\mathrm{d}x = - \int_{\R^3}{pBe\cdot \nabla_x k[n]}\mathrm{d}x =0.
\]
\end{proof}
Recall the usual drift velocities when dealing with magnetic confinement: the electric field drift, the magnetic gradient drift, and the magnetic curvature drift
\[
\dfrac{E\wedge e}{B}, \,\, - \dfrac{m|v\wedge e|^2}{2qB}\dfrac{\nabla_x B\wedge e}{B} = - \dfrac{|v\wedge e|^2}{2}\dfrac{\nabla_x \omega_c \wedge e}{\omega_c ^2},\,\, -\dfrac{m|v\wedge e|^2}{qB}\partial_x e e\wedge e = -\dfrac{(v\cdot e)^2}{\omega_c}\partial_x e e\wedge e.
\]
When working at the fluid level, the averages with respect to $v\in\R^3$ of the above drift velocities become
\[
v_{\wedge D} = \int_{\R^3}{\dfrac{E\wedge e}{B}M(v)}\mathrm{d}v = \dfrac{E\wedge e}{B},
\]
\[
v_{GD} = - \int_{\R^3}{\dfrac{|v\wedge e|^2}{2}\dfrac{\nabla_x \omega_c \wedge e}{\omega_c ^2}M(v)}\mathrm{d}v = -\sigma \dfrac{\nabla_x \omega_c \wedge e}{\omega_c ^2},
\]
\[
v_{CD} = - \int_{\R^3}{\dfrac{(v\cdot e)^2}{\omega_c}\partial_x e e\wedge e M(v)}\mathrm{d}v = - \sigma \dfrac{\partial_x e e\wedge e}{\omega_c}.
\]
The flux in the limit model \eqref{equ:gyro-kinetic} also writes $n\calV[n]$, where $\calV[n] = v_{\wedge D} + v_{GD} + v_{CD}$.
\begin{pro}
\label{Equiv_form}
Any non-negative regular function $n$ satisfying
\[
\partial_t n + \Divx\left( \dfrac{ne}{\omega_c}\wedge \nabla_x k[n] \right) + Be\cdot\nabla_x p =0,\quad k[n] = \sigma(1+\ln n) + \dfrac{q}{m}\Phi[n],
\]
also verifies
\[
\partial_t n + \Divx(n\calV[n]) + Be\cdot\nabla_x\tilde{p} =0,\quad \calV[n] = \dfrac{E\wedge e}{B} -\sigma \dfrac{\nabla_x \omega_c \wedge e}{\omega_c ^2}- \sigma \dfrac{\partial_x e e\wedge e}{\omega_c},
\]
and $\tilde{p} = p+ \frac{\sigma n}{B\omega_c}(e\cdot \rot_x e)$.
\end{pro}
\begin{proof}
Recall the formula $\Divx{(\xi \wedge \eta)}= \eta\cdot \rot_x \xi - \xi\cdot \rot_x \eta$, for any smooth vector fields $\xi$ and $\eta$. Therefore we can write
\begin{align*}
\Divx\left( \dfrac{ne}{\omega_c}\wedge \nabla_x k[n] \right) &= \Divx\left[\dfrac{ne}{\omega_c}\wedge\left( \sigma\dfrac{\nabla_x n}{n} -\dfrac{q}{m}E \right) \right]\\
&=\Divx\left( n\dfrac{E\wedge e}{B}\right) + \sigma\, \rot_x\left(\dfrac{e}{\omega_c} \right)\cdot\nabla_x n\\
&= \Divx\left( n\dfrac{E\wedge e}{B}\right) + \sigma\, \Divx\left(n \,\rot_x \left(\dfrac{e}{\omega_c} \right) \right)\\
&= \Divx\left( n\dfrac{E\wedge e}{B}\right) - \sigma\, \Divx\left( n\dfrac{\nabla_x \omega_c \wedge e}{\omega_c ^2} - \dfrac{n}{\omega_c}\rot_x e \right)\\
&= \Divx\left( n v_{\wedge D} + n v_{GD} + \dfrac{\sigma n}{\omega_c}(I_3 -e\otimes e)\rot_x e\right) + \Divx\left(\dfrac{\sigma n}{\omega_c}(e\cdot\rot_x e)e \right).
\end{align*}
Notice that we can write 
\[
(I_3 -e\otimes e)\rot_x e = e\wedge (\rot_x e \wedge e) = e\wedge[(\partial_x e - {^t}\partial_x e)e] = e \wedge \partial_x e e,
\]
implying that
\[
\sigma \dfrac{n}{\omega_c} (I_3 - e\otimes e)\rot_x e = - \sigma \dfrac{n}{\omega_c} \partial_x e e \wedge e = n v_{CD}.
\]
Finally, we obtain
\[
\Divx\left( \dfrac{ne}{\omega_c}\wedge \nabla_x k[n] \right) = \Divx(n\calV[n]) + Be\cdot \nabla_x\left[ \dfrac{\sigma n}{B\omega_c}(e\cdot \rot_x e) \right],
\]
and our conclusion follows.
\end{proof}
\section{Convergence result}
\label{Conve}
We concentrate now on the asymptotic behavior as $\eps \searrow 0$ of the family of weak solutions $(f^\eps, E[f^\eps])_{\eps>0}$ of the VPFP system \eqref{equ:VPFP-Scale}, \eqref{equ:PoissonEpsi}, and \eqref{equ:Initial} and we establish rigorously the connection to the fluid model \eqref{equ:gyro-kinetic}, \eqref{equ:constraint}, and \eqref{equ:PoissonLimit}. \\
We are looking a model for the concentration $n^\eps = n[f^\eps] =\intvt{f^\eps}$, similar to the equation \eqref{equ:gyro-kinetic} of the limit concentration $n$ and we perform the balance of the relative entropy between $n^\eps$ and $n$. As usual, these computations require the smoothness of the solution for the limit model. We justify the asymptotic behavior of  $(f^\eps, E[f^\eps])_{\eps>0}$ when $\eps \searrow 0$, provided that there is a regular solution $(n, E[n] = -\nabla_x\Phi[n])$ for the fluid model \eqref{equ:gyro-kinetic}, \eqref{equ:constraint}, and \eqref{equ:PoissonLimit}. We do not concentrate on the well posedness of this fluid model, nevertheless we refer to Section 8 for some examples of regular solutions. We are working with weak solutions $(f^\eps,E[f^\eps])_{\eps>0}$.

The balance for the number of particles writes
\begin{equation}
\label{equ:ParticleDens}
\partial_t n^\eps + \dfrac{1}{\eps}\Divx j^\eps =0,\quad j^\eps = j[f^\eps]=\int_{\R^3}{f^\eps v}\mathrm{d}v.
\end{equation}
We are using the balance momentum as well
\begin{equation}
\label{equ:EquMomentum}
\eps \partial_t j^\eps + \Divx \int_{\R^3}{f^\eps v\otimes v}\mathrm{d}v - \dfrac{q}{m}n^\eps E[f^\eps] - \dfrac{\omega_c}{\eps}j^\eps \wedge e = - \dfrac{j^\eps}{\tau},
\end{equation}
which allows us to express the orthogonal component of $j^\eps$
\begin{align*}
\dfrac{j^\eps -(j^\eps\cdot e)e}{\eps} &= \dfrac{n^\eps e}{\omega_c}\wedge \left(\sigma \dfrac{\nabla_x n^\eps}{n^\eps}+ \dfrac{q}{m}\nabla_x \Phi[f^\eps] \right)\\
&\quad + \dfrac{e}{\omega_c}\wedge \left[ \Divx\int_{\R^3}{(\sigma \nabla_v f^\eps + v f^\eps)\otimes v}\mathrm{d}v + \eps\partial_t j^\eps + \dfrac{j^\eps}{\tau}  \right]\\
&= \dfrac{n^\eps e}{\omega_c}\wedge \nabla_x k[n^\eps] + \dfrac{e}{\omega_c}\wedge F^\eps ,
\end{align*}
where we denote
\[
F^\eps =  \Divx\int_{\R^3}{(\sigma \nabla_v f^\eps + v f^\eps)\otimes v}\mathrm{d}v + \eps\partial_t j^\eps + \dfrac{j^\eps}{\tau},
\]
and in the above computation, we have used that $\Divx\intvt{\sigma \nabla_v f^\eps \otimes v} = - \sigma \nabla_x n^\eps$.\\
Observe that 
\begin{align*}
\dfrac{1}{\eps}\Divx j^\eps &= \Divx \dfrac{j^\eps - (j^\eps\cdot e)e}{\eps} + \Divx\left[ \dfrac{(j^\eps\cdot e)Be}{B\eps}\right]\\
&= \Divx\left( \dfrac{n^\eps e}{\omega_c}\wedge \nabla_x k[n^\eps] \right) + \Divx\left( \dfrac{e}{\omega_c}\wedge F^\eps\right)+ Be\cdot \nabla_x \left[ \dfrac{(j^\eps\cdot e)}{B\eps}\right],
\end{align*}
and finally, thanks to \eqref{equ:ParticleDens}, we obtain a similar model for $n^\eps$, as in \eqref{equ:gyro-kinetic}
\begin{equation}
\label{equ:EquDensityEps}
\partial_t n^\eps + \Divx\left( \dfrac{n^\eps e}{\omega_c}\wedge \nabla_x k[n^\eps] \right) + \Divx\left( \dfrac{e}{\omega_c}\wedge F^\eps\right)+ Be\cdot \nabla_x p^\eps =0,\quad p^\eps = \dfrac{j^\eps\cdot e}{B\eps}.
\end{equation}
We are also looking for a equation, analogous to \eqref{equ:constraint}, in order to complete the evolution equation \eqref{equ:EquDensityEps}, involving the Lagrange multiplier $p^\eps$. Considering the parallel component in the momentum balance \eqref{equ:EquMomentum}, we obtain
\[
\sigma e\cdot \nabla_x n^\eps + \dfrac{q}{m} n^\eps e\cdot \nabla_x \Phi[n^\eps] + e\cdot F^\eps =0.
\]
Thanks to \eqref{equ:constraint}, the above equation also writes
\begin{equation}
\label{equ:EquParallel}
e\cdot \nabla_x \left(\sigma \dfrac{n^\eps -n}{n} + \dfrac{q}{m}(\Phi[n^\eps]-\Phi[n]) \right) + \dfrac{e}{n}\cdot F^\eps = \dfrac{q}{m}\dfrac{(n^\eps -n)(E[n^\eps]-E[n])}{n}\cdot e.
\end{equation}
We intend to estimate the modulated energy of $n^\eps$ with respect to $n$  by writing $\calE[n^\eps|n]$ as
\begin{align}
\label{equ:EntropyDens}
\calE[n^\eps|n] &= \sigma \int_{\R^3}{n h\left(\dfrac{n^\eps}{n} \right)}\mathrm{d}x + \dfrac{\epsilon_0}{2m}\int_{\R^3}{|\nabla_x \Phi[n^\eps]-\nabla_x\Phi[n]|^2}\mathrm{d}x\nonumber\\
&= \int_{\R^3}{(\sigma n^\eps \ln n^\eps +\dfrac{\epsilon_0}{2m} |\nabla_x\Phi[n^\eps]|^2)}\mathrm{d}x - \int_{\R^3}{(\sigma n \ln n +\dfrac{\epsilon_0}{2m} |\nabla_x\Phi[n]|^2)}\mathrm{d}x\nonumber\\
&\quad - \int_{\R^3}{\left\{ \sigma(1+\ln n)+ \dfrac{q}{m}\Phi[n]\right\}(n^\eps -n)}\mathrm{d}x\nonumber\\
&:= \calE[n^\eps] -\calE[n] - \int_{\R^3}{k[n](n^\eps -n)}\mathrm{d}x.
\end{align}
We introduce as well the modulated energy of $f^\eps$ with respect to $n^\eps M$, given by
\begin{align*}
\sigma&\int_{\R^3}\int_{\R^3}{n^\eps M h\left(\dfrac{f^\eps}{n^\eps M} \right)}\mathrm{d}v\mathrm{d}x + \dfrac{\epsilon_0}{2m}\int_{\R^3}{\underbrace{|\nabla_x \Phi[f^\eps]-\nabla_x\Phi[n^\eps M]|^2}_{=0}}\mathrm{d}x\\
&= \sigma \int_{\R^3}\int_{\R^3}{f^\eps \ln f^\eps - f^\eps \ln n^\eps + f^\eps \ln (2\pi\sigma)^{3/2}+ f^\eps \frac{|v|^2}{2\sigma}}\mathrm{d}v\mathrm{d}x\\
&= \int_{\R^3}\int_{\R^3}{\sigma f^\eps \ln f^\eps + f^\eps \frac{|v|^2}{2}}\mathrm{d}v\mathrm{d}x + \dfrac{\epsilon_0}{2m}\int_{\R^3}{|\nabla_x \Phi[f^\eps]|^2}\mathrm{d}x \\
&\quad- \int_{\R^3}{\sigma n^\eps \ln n^\eps}\mathrm{d}x - \dfrac{\epsilon_0}{2m}\int_{\R^3}{|\nabla_x \Phi[n^\eps]|^2}\mathrm{d}x + \sigma \ln(2\pi\sigma)^{3/2}\int_{\R^3}\int_{\R^3}{f^\eps}\mathrm{d}v\mathrm{d}x\\
&= \calE[f^\eps] - \calE[n^\eps]+ \sigma \ln(2\pi\sigma)^{3/2}\int_{\R^3}\int_{\R^3}{f^\eps}\mathrm{d}v\mathrm{d}x.
\end{align*}
Thanks to the free energy balance \eqref{equ:EquFreeEne} and mass conservation of \eqref{equ:VPFP-Scale}, we obtain
\begin{align}
\label{equ:BalanEnerDens}
\calE[n^\eps(t)] - \calE[n^\eps(0)] & + \sigma\int_{\R^3}\int_{\R^3}{n^\eps(t) M h\left(\dfrac{f^\eps(t)}{n^\eps(t) M} \right)}\mathrm{d}v\mathrm{d}x \\
& -\sigma\int_{\R^3}\int_{\R^3}{n^\eps(0) M h\left(\dfrac{f^\eps(0)}{n^\eps(0) M} \right)}\mathrm{d}v\mathrm{d}x\nonumber\\
 &=  - \dfrac{1}{\eps\tau}\int_0^t\int_{\R^3}\int_{\R^3}{\dfrac{|\sigma \nabla_v f^\eps + v f^\eps|^2}{f^\eps}}\mathrm{d}v\mathrm{d}x\mathrm{d}s\nonumber.
\end{align}
Thanks to Proposition \ref{BalLiMod} and combining \eqref{equ:EntropyDens}, \eqref{equ:BalanEnerDens} leads to
\begin{align}
\label{BalModEnerDens}
&\calE[n^\eps(t)|n(t)] +  \sigma\int_{\R^3}\int_{\R^3}{n^\eps(t) M h\left(\dfrac{f^\eps(t)}{n^\eps(t) M} \right)}\mathrm{d}v\mathrm{d}x + \dfrac{1}{\eps\tau}\int_0^t\int_{\R^3}\int_{\R^3}{\dfrac{|\sigma \nabla_v f^\eps + v f^\eps|^2}{f^\eps}}\mathrm{d}v\mathrm{d}x\mathrm{d}s\nonumber\\
&= \calE[n^\eps(0)|n(0)] +  \sigma\int_{\R^3}\int_{\R^3}{n^\eps(0) M h\left(\dfrac{f^\eps(0)}{n^\eps(0) M} \right)}\mathrm{d}v\mathrm{d}x - \int_{0}^{t}\dfrac{\mathrm{d} }{\mathrm{d} s}\int_{\R^3}{k[n](n^\eps -n)}\mathrm{d}x\mathrm{d}s.
\end{align}

The next task is to evaluate the time derivative of $\intxt{k[n](n^\eps -n)}$. Notice that for any smooth concentration $n$, we can write: 
\begin{align*}
\dfrac{n e}{\omega_c} \wedge \nabla_x k[n] 
&= \dfrac{\sigma e}{\omega_c}\wedge \nabla_x n + \dfrac{n e}{B}\wedge \nabla_x \Phi[n]\\
&= n V[n] - \sigma \rot_x \left( \dfrac{ne}{\omega_c} \right),
\end{align*}
where $V[n] = \sigma \rot_x\left(\frac{e}{\omega_c} \right)+\frac{e\wedge \nabla_x\Phi[n]}{B}$. Clearly, we have
\begin{equation}
\label{equ:EquDivVit}
\Divx\left( \dfrac{ne}{\omega_c}\wedge \nabla_x k[n]\right) = \Divx(nV[n]).
\end{equation}
\begin{pro}
\label{TimeDeri}
With the notations from \eqref{equ:gyro-kinetic}, \eqref{equ:constraint}, \eqref{equ:EquDensityEps}, and \eqref{equ:EquParallel}, we have the following equality
\begin{align*}
&\dfrac{\mathrm{d}}{\mathrm{d} t}\int_{\R^3}{k[n(t)](n^\eps(t,x)-n(t,x))}\mathrm{d}x \\&= \int_{\R^3}{\left(p\dfrac{Be}{n} + \dfrac{e}{\omega_c}\wedge \nabla_x k[n] \right)\cdot \left( \dfrac{q}{m}(n^\eps -n)(E[n^\eps]-E[n]) -F^\eps \right)}\mathrm{d}x.
\end{align*}
\end{pro}
\begin{proof}
By straightforward computations, we obtain
\begin{align}
\label{equ:FirstTimeDeri}
&\dfrac{\mathrm{d}}{\mathrm{d} t}\int_{\R^3}{k[n](n^\eps-n)}\mathrm{d}x \nonumber\\
&= \int_{\R^3}{\partial_t n \left(\sigma \dfrac{n^\eps - n}{n} + \dfrac{q}{m}(\Phi[n^\eps]-\Phi[n]) \right)}\mathrm{d}x\\
&\quad + \int_{\R^3}{k[n]\left[ \Divx\left( \dfrac{ne}{\omega_c}\wedge \nabla_x k[n]\right) - \Divx \left( \dfrac{n^\eps e}{\omega_c}\wedge \nabla_x k[n^\eps]\right) - \Divx \left( \dfrac{e}{\omega_c}\wedge F^\eps \right) \right]}\mathrm{d}x,\nonumber
\end{align}
where, in the last integral, we have used the constraint $Be\cdot\nabla_x k[n]=0$, which allows us to deduce that
\[
\int_{\R^3}{k[n](Be\cdot\nabla_x p - Be\cdot\nabla_x p^\eps)}\mathrm{d}x =0.
\]
Using the limit equation \eqref{equ:gyro-kinetic}, and then thanks to \eqref{equ:EquDivVit} and \eqref{equ:EquParallel}, we have
\begin{align}
\label{equ:FirstTimeDerBis}
&\int_{\R^3}{\partial_t n \left(\sigma \dfrac{n^\eps - n}{n} + \dfrac{q}{m}(\Phi[n^\eps]-\Phi[n]) \right)}\mathrm{d}x\\
&= -\int_{\R^3}{\Divx(nV[n])\left(\sigma \dfrac{n^\eps - n}{n} \right)}\mathrm{d}x - \int_{\R^3}{\left( \dfrac{ne}{\omega_c}\wedge \nabla_x k[n]\right)\cdot\dfrac{q}{m}(E[n^\eps]-E[n])}\mathrm{d}x\nonumber\\
&\quad+ \int_{\R^3}{\dfrac{pBe}{n}\cdot \left[\dfrac{q}{m}(n^\eps -n)(E[n^\eps] -E[n])-F^\eps \right]}\mathrm{d}x\nonumber\\
&=-\sigma \int_{\R^3}{\Divx\left( \dfrac{e\wedge \nabla_x\Phi[n]}{B}\right)(n^\eps -n)}\mathrm{d}x - \sigma \int_{\R^3}{\nabla_x \ln n \cdot V[n](n^\eps -n)}\mathrm{d}x\nonumber\\
&\quad- \int_{\R^3}{\left( \dfrac{ne}{B}\wedge \nabla_x k[n]\right)\cdot(E[n^\eps]-E[n])}\mathrm{d}x + \int_{\R^3}{\dfrac{pBe}{n}\cdot \left[\dfrac{q}{m}(n^\eps -n)(E[n^\eps] -E[n])-F^\eps \right]}\mathrm{d}x\nonumber\\
&= - \int_{\R^3}{\nabla_x k[n]\cdot V[n](n^\eps -n)}\mathrm{d}x - \int_{\R^3}{\left( \dfrac{ne}{B}\wedge \nabla_x k[n]\right)\cdot(E[n^\eps]-E[n])}\mathrm{d}x\nonumber\\
&\quad+ \int_{\R^3}{\dfrac{pBe}{n}\cdot \left[\dfrac{q}{m}(n^\eps -n)(E[n^\eps] -E[n])-F^\eps \right]}\mathrm{d}x.\nonumber
\end{align}
Thanks to \eqref{equ:EquDivVit} again, the last integral in \eqref{equ:FirstTimeDeri} writes easily
\begin{align}
\label{equ:LastInte}
&\int_{\R^3}{k[n]\left[ \Divx\left( \dfrac{ne}{\omega_c}\wedge \nabla_x k[n]\right)  - \Divx \left( \dfrac{n^\eps e}{\omega_c}\wedge \nabla_x k[n^\eps]\right) - \Divx \left( \dfrac{e}{\omega_c}\wedge F^\eps \right)\right]}\mathrm{d}x\\
&= \int_{\R^3}{\nabla_x k[n]\cdot (n^\eps V[n^\eps] - n V[n])}\mathrm{d}x - \int_{\R^3}{\left( \dfrac{e}{\omega_c}\wedge \nabla_x k[n]\right)\cdot F^\eps}\mathrm{d}x\nonumber.
\end{align}
Observe that 
\begin{align*}
n^\eps V[n^\eps] - nV[n]- (n^\eps -n)V[n] &= n^\eps  \dfrac{e\wedge \nabla_x \Phi[n^\eps]}{B} - n \dfrac{e\wedge \nabla_x \Phi[n]}{B} - (n^\eps -n)\dfrac{e\wedge \nabla_x \Phi[n]}{B}\\
&= n^\eps  \dfrac{e\wedge (\nabla_x \Phi[n^\eps]-\nabla_x\Phi[n])}{B} ,
\end{align*}
and finally \eqref{equ:FirstTimeDeri}, \eqref{equ:FirstTimeDerBis} and \eqref{equ:LastInte} yield the result.
\end{proof}

Coming back to \eqref{BalModEnerDens}, the modulated energy balance becomes
\begin{align}
\label{BalModEnerDensBis}
\calE[n^\eps(t)|n(t)] &+  \sigma\int_{\R^3}\int_{\R^3}{n^\eps(t) M h\left(\dfrac{f^\eps(t)}{n^\eps(t) M} \right)}\mathrm{d}v\mathrm{d}x + \dfrac{1}{\eps\tau}\int_0^t\int_{\R^3}\int_{\R^3}{\dfrac{|\sigma \nabla_v f^\eps + v f^\eps|^2}{f^\eps}}\mathrm{d}v\mathrm{d}x\mathrm{d}s\nonumber\\
&= \calE[n^\eps(0)|n(0)] +  \sigma\int_{\R^3}\int_{\R^3}{n^\eps(0) M h\left(\dfrac{f^\eps(0)}{n^\eps(0) M} \right)}\mathrm{d}v\mathrm{d}x \\
&\quad- \int_{0}^{t}{\int_{\R^3}{W[n]\cdot \left( \dfrac{q}{m}(n^\eps -n)(E[n^\eps]-E[n])-F^\eps\right)}\mathrm{d}x}\mathrm{d}s,\nonumber
\end{align}
where $W[n] = \frac{pBe}{n}+\frac{e}{\omega_c}\wedge \nabla_x k[n]$. In order to apply the Gronwall lemma, we estimate the terms in the last integral of \eqref{BalModEnerDensBis}. Thanks to the formula
\begin{align*}
 \dfrac{q}{m}(n^\eps -n)(E[n^\eps]-E[n]) &= \dfrac{\epsilon_0}{m}[\Divx(E[n^\eps]-E[n])](E[n^\eps]-E[n])\\
 &= \dfrac{\epsilon_0}{m} \Divx \left( (E[n^\eps]-E[n])\otimes (E[n^\eps]-E[n]) - \dfrac{|E[n^\eps]-E[n]|^2}{2}I_3 \right),
\end{align*}
we obtain
\begin{align*}
- \int_{\R^3}&{W[n]\cdot \left( \dfrac{q}{m}(n^\eps -n)(E[n^\eps]-E[n])\right)}\mathrm{d}x \\
&= \dfrac{\epsilon_0}{m}\int_{\R^3}{ \left((E[n^\eps]-E[n])\otimes (E[n^\eps]-E[n]) - \dfrac{|E[n^\eps]-E[n]|^2}{2}I_3\right): \partial_x W[n]}\mathrm{d}x\\
&\leq \|\partial_x W[n] \|_{L^\infty}\dfrac{\epsilon_0}{m}\left(1+ \dfrac{\sqrt{3}}{2} \right)\int_{\R^3}{|E[n^\eps]-E[n]|^2}\mathrm{d}x,
\end{align*}
where for any matrix $P\in \calM_{3,3}(\R)$, the notation $\| P\|$ stands for $(P:P)^{1/2}$. Similarly, we have for some value $C$ to be precised later on
\begin{align*}
&\int_{\R^3}{W[n]\cdot \Divx \int_{\R^3}{(\sigma \nabla_v f^\eps + f^\eps v)\otimes v}\mathrm{d}v}\mathrm{d}x\\
 &= - \int_{\R^3}{\partial_x W[n]: \int_{\R^3}{(\sigma \nabla_v f^\eps + f^\eps v)\otimes v}\mathrm{d}v}\mathrm{d}x\\
 &\leq \|\partial_x W[n]  \|_{L^\infty}\left[\dfrac{1}{2\eps\tau C}\int_{\R^3\int_{\R^3}}{\dfrac{|\sigma \nabla_v f^\eps + f^\eps v |^2}{f^\eps}}\mathrm{d}v\mathrm{d}x + \eps \tau C \int_{\R^3}\int_{\R^3}{f^\eps \dfrac{|v|^2}{2}}\mathrm{d}v\mathrm{d}x \right].
\end{align*}
Since $j^\eps = \intvt{(\sigma \nabla_v f^\eps + f^\eps v )}$, we have
\begin{align*}
&\int_{0}^{t}\int_{\R^3}{W[n(s)]\cdot (\eps \partial_s j^\eps + \dfrac{j^\eps}{\tau})}\mathrm{d}x\mathrm{d}s \\
&=  \eps \int_{\R^3}{W[n(t)]\cdot j^\eps(t,x)}\mathrm{d}x - \eps \int_{\R^3}{W[n(0)]\cdot j^\eps(0,x)}\mathrm{d}x\\
&\quad+ \int_0^t\int_{\R^3}\int_{\R^3}{[\sigma \nabla_v f^\eps + f^\eps(s,x,v)v]\cdot \left[ \dfrac{W[n(s)]}{\tau} - \eps \partial_s W[n(s)] \right]}\mathrm{d}v\mathrm{d}x\mathrm{d}s\\
&\leq \sqrt{\eps} \int_{\R^3}\int_{\R^3}{(f^\eps(0,x,v) + f^\eps(t,x,v))\left(\eps\dfrac{|v|^2}{2}+ \dfrac{\| W[n]\|^2 _{L^\infty}}{2} \right)}\mathrm{d}v\mathrm{d}x\\
&\quad+ \left[\eps \| \partial_s W[n] \|_{L^\infty} + \dfrac{\|W[n]\|_{L^\infty}}{\tau}  \right]\int_0^t\int_{\R^3}\int_{\R^3}{\left\{ \dfrac{1}{2\eps C}\dfrac{|\sigma \nabla_v f^\eps + f^\eps v |^2}{f^\eps}+ \dfrac{\eps C}{2}f^\eps \right\}}\mathrm{d}v\mathrm{d}x\mathrm{d}s.
\end{align*}
Plugging the above computations in \eqref{BalModEnerDensBis}, the modulated energy balance becomes for $t\in[0,T]$
\begin{align*}
&\calE[n^\eps(t)|n(t)] +  \sigma\int_{\R^3}\int_{\R^3}{n^\eps(t) M h\left(\dfrac{f^\eps(t)}{n^\eps(t) M} \right)}\mathrm{d}v\mathrm{d}x \\& \quad + \dfrac{1}{\eps\tau}\left(1-\dfrac{\|W[n]\|_{L^\infty}}{2C}-\dfrac{\eps \tau \| \partial_s W[n] \|_{L^\infty} }{2C}  -\dfrac{\|W[n]\|_{L^\infty}}{2C}  \right)\int_0^t\int_{\R^3}\int_{\R^3}{\dfrac{|\sigma \nabla_v f^\eps + v f^\eps|^2}{f^\eps}}\mathrm{d}v\mathrm{d}x\mathrm{d}s\\
&\leq \calE[n^\eps(0)|n(0)]+ \sigma\int_{\R^3}\int_{\R^3}{n^\eps(0) M h\left(\dfrac{f^\eps(0)}{n^\eps(0) M} \right)}\mathrm{d}v\mathrm{d}x \\
&\quad+ \|\partial_x W[n] \|_{L^\infty}\left(2+ \sqrt{3} \right)\dfrac{\epsilon_0}{2m}\int_{\R^3}{|E[n^\eps]-E[n]|^2}\mathrm{d}x\\
&\quad+\eps \dfrac{ \tau C}{2}\|\partial_x W[n] \|_{L^\infty} \int_0^T\int_{\R^3}\int_{\R^3}{f^\eps |v|^2}\mathrm{d}v\mathrm{d}x\mathrm{d}t + \sqrt{\eps}\sup_{0\leq t\leq T}\eps \int_{\R^3}\int_{\R^3}{f^\eps |v|^2}\mathrm{d}v\mathrm{d}x\\
&\quad+ \sqrt{\eps} \left[ \sqrt{\eps}\dfrac{CT}{2}\left( \eps \| \partial_s W[n] \|_{L^\infty} + \dfrac{\|W[n]\|_{L^\infty}}{\tau}\right)+ \|W[n]\|_{L^\infty}^2 \right]\int_{\R^3}\int_{\R^3}{f^\eps(0,x,v)}\mathrm{d}v\mathrm{d}x.
\end{align*}
Taking $0 < \eps \leq 1$ and $C$ large enough, we obtain by Lemma \ref{KinEne} and \eqref{equ:EntropyDens}, for some constant $C_T$, $0\leq t\leq T$, $0< \eps \leq 1$:
\begin{align*}
&\calE[n^\eps(t)|n(t)] +  \sigma\int_{\R^3}\int_{\R^3}{n^\eps(t) M h\left(\dfrac{f^\eps(t)}{n^\eps(t) M} \right)}\mathrm{d}v\mathrm{d}x + \dfrac{1}{2\eps\tau}\int_0^t\int_{\R^3}\int_{\R^3}{\dfrac{|\sigma \nabla_v f^\eps + v f^\eps|^2}{f^\eps}}\mathrm{d}v\mathrm{d}x\mathrm{d}s\\
&\leq  \calE[n^\eps(0)|n(0)] +  \sigma\int_{\R^3}\int_{\R^3}{n^\eps(0) M h\left(\dfrac{f^\eps(0)}{n^\eps(0) M} \right)}\mathrm{d}v\mathrm{d}x + C_T \int_{0}^{t}\calE[n^\eps(s)|n(s)]\mathrm{d} s + C_T \sqrt{\eps}.
\end{align*}
Applying Gronwall lemma, we deduce that for $0\leq t\leq T$, $0< \eps \leq 1$:
\begin{align*}
\calE[n^\eps(t)|n(t)] +&  \sigma\int_{\R^3}\int_{\R^3}{n^\eps(t) M h\left(\dfrac{f^\eps(t)}{n^\eps(t) M} \right)}\mathrm{d}v\mathrm{d}x + \dfrac{1}{2\eps\tau}\int_0^t \int_{\R^3}\int_{\R^3}{\dfrac{|\sigma \nabla_v f^\eps + v f^\eps|^2}{f^\eps}}\mathrm{d}v\mathrm{d}x\mathrm{d}s\\
&\leq  \left[\calE[n^\eps(0)|n(0)] +  \sigma\int_{\R^3}\int_{\R^3}{n^\eps(0) M h\left(\dfrac{f^\eps(0)}{n^\eps(0) M} \right)}\mathrm{d}v\mathrm{d}x  + C_T \sqrt{\eps}\right]e^{C_T t}.
\end{align*}
The above inequality says that the particle density $f^\eps$ remains close to the Maxwellian with the same concentration, i.e., $n^\eps(t)M$, and $n^\eps(t)$ stays near $n(t)$, provided that analogous behaviour occur for the initial conditions. Therefore, we are ready to prove our main theorem.
\begin{proof}(of Theorem \ref{MainThm}).
We justify the convergence of $f^\eps$ toward $nM$ in $L^\infty(0,T;L^1(\R^3\times\R^3))$, the other convergences being obvious. We use the Csis\'ar -Kullback inequality in order to control the $L^1$ norm by the relative entropy, cf. \cite{Csi1967, Kul1967}
\[
\int_{\R^n}|g-g_0|\mathrm{d} x \leq 2 \max \left\{ \left( \int_{\R^n}g_0\mathrm{d} x \right)^{1/2}, \left( \int_{\R^n}g\mathrm{d} x\right)^{1/2} \right\}\left(\int_{\R^n}g_0  h\left(\dfrac{g}{g_0} \right)\mathrm{d} x \right)^{1/2},
\]
for any non negative integrable functions $g_0,g: \R^n \to \R$. Applying two times the  Csis\'ar -Kullback inequality, we obtain
\begin{align*}
&\int_{\R^3}\int_{\R^3}{|f^\eps(t,x,v) -n(t,x)M(v)|}\mathrm{d}v\mathrm{d}x \\ &\leq \int_{\R^3}\int_{\R^3}{|f^\eps(t,x,v)- n^\eps(t,x)M(v)|}\mathrm{d}v\mathrm{d}x + \int_{\R^3}{|n^\eps(t,x)-n(t,x)|}\mathrm{d}x\\
&\leq 2 \sqrt{M_{\mathrm{in}}} \left(\int_{\R^3}\int_{\R^3}{ n^\eps(t)M(v) h\left(\dfrac{f^\eps(t)}{n^\eps(t)M} \right)}\mathrm{d}v\mathrm{d}x\right)^{1/2} \\&\quad + 2 \max\left\{ \sqrt{M_{\mathrm{in}}},\sqrt{\|n_{\mathrm{in}}\|_{L^1(\R^3)}} \right\} \left( \int_{\R^3}{n(t) h \left(\dfrac{n^\eps(t)}{n(t)}\right)}\mathrm{d}x  \right)^{1/2}  \to 0 ,\,\,\mathrm{as}\,\eps\searrow 0.
\end{align*}
\end{proof}

\section{Reformulation of the limit model}
\label{RefLiMod}
We intend to find an equivalent formulation for \eqref{equ:gyro-kinetic}, \eqref{equ:constraint} by eliminating the Lagrange multiplier $p$ which appears in \eqref{equ:gyro-kinetic}. For doing that, we will average along the characteristic flow of the magnetic field cf. \cite{BogMit61, BosTraEquSin, BosAsyAna, BosGuiCen3D, BosFinHauCRAS, BosFin16, BosSIAM09}. Let us recall briefly the definition of the average operators along a characteristic flow for functions and vector fields cf. \cite{BosSIAM16}. Consider a smooth, divergence free vector field $b=b(y):\R^m \to \R^m$
\begin{equation}
\label{equ:EquLipDiv}
b\in W^{1,\infty}_{\mathrm{loc}}(\R^m),\quad \mathrm{div}_y b =0,
\end{equation}
with at most linear growth at infinity
\begin{equation}
\label{equ:EquGrowth}
\exists C>0\,\,\text{such that}\,\, |b(y)|\leq C(1+|y|),\,\,y\in\R^m .
\end{equation}
We denote by $Y(s;y)$ the characteristic flow associated to $b$
\[
\dfrac{\mathrm{d}Y}{\mathrm{d}s} = b(Y(s;y)),\quad Y(0;y) = y,\,\, s\in\R,\,\, y\in\R^m.
\]
Under the above hypothese, this flow has the regularity $Y\in W^{1,\infty}_{\mathrm{loc}}(\R\times\R^m)$ and is measure preserving. We concentrate on periodic characteristic flows (the tokamak characteristic flows are periodic, with uniform period) that is:
\[
\exists S>0\,\,\text{such that}\,\, Y(S;y)=y,\,\,y\in\R^m.
\]
For any function $u=u(y):\R^m\to \R$, we define the average $\left< u\right>$ along the flow of $b\cdot\nabla_y$ by
\[
\left< u\right> (y) = \dfrac{1}{S} \int_0^S{u(Y(s;y))}\mathrm{d}s,\,\, y\in\R^m.
\]
When applied to $L^2(\R^m)$ functions, the above operator coincides with the orthogonal projection in $L^2(\R^m)$, over the subspace of constant functions along the flow of $b\cdot\nabla_y$, cf. \cite{BosTraEquSin}. Indeed, it is easily seen that for any $y\in\R^m$, $h\in\R$
\[
\left< u\right> (Y(h;y)) = \dfrac{1}{S} \int_{0}^{S}u(Y(s;Y(h;y)))\mathrm{d}s = \dfrac{1}{S} \int_{0}^{S}u(Y(s+h;y))\mathrm{d}s = \left< u\right>(y),
\]
and for any $\psi\in L^2(\R^m)$ which is constant along the flow $Y$, we have
\begin{align*}
\int_{\R^3}{u(y)\psi(y)}\mathrm{d}y 
&= \int_{\R^3}{u(Y(s;y))\psi(y)}\mathrm{d}y\\
&= \int_{\R^3}{\dfrac{1}{S}\int_0^S{ u(Y(s;y))}\mathrm{d}s \psi(y)}\mathrm{d}y\\
&= \int_{\R^3}{\left< u\right>(y)\psi(y)}\mathrm{d}y.
\end{align*}
For any vector field $c = c(y):\R^m \to \R^m$, we define the average $\left< c\right>$ along the flow of $b\cdot\nabla_y$ by
\[
\left< c\right> = \dfrac{1}{S}\int_0^S{\partial Y(-s;Y(s;\cdot))c(Y(s;\cdot))}\mathrm{d}s.
\]
Notice that the family of transformations $c\to \partial Y(-s;Y(s;\cdot))c(Y(s;\cdot)) $, $s\in\R$, is a one parameter group. The average operators for functions and vector fields are related by the following formulas:
\begin{equation}
\label{equ:AveConsFlowY}
\left< c\cdot \nabla \psi\right> = \left< c\right>\cdot\nabla \psi,
\end{equation}
for any function $\psi$ which is constant along the flow $Y$ and 
\begin{equation}
\label{equ:AveInvoFieldb}
\left< a\cdot\nabla\theta\right> = a\cdot \nabla \left< \theta\right>,
\end{equation}
for any vector field $a$ which is in involution with respect to $b$, that is, their Poisson bracket vanishes
\[
[a,b] : = (a\cdot\nabla_y)b - (b \cdot \nabla_y )a  =0.
\]
 Indeed, as $\psi(Y(s;\cdot)) =\psi, s\in\R$, we have $^t \partial Y(s;y)(\nabla\psi)(Y(s;y)) = \nabla\psi(y), s\in\R$ and therefore
\begin{align*}
\left< c\right>\cdot\nabla \psi &= \dfrac{1}{S} \int_0^S{ \partial Y(-s;Y(s;\cdot))c(Y(s;\cdot))}\mathrm{d}s \cdot \nabla \psi \\
&= \dfrac{1}{S} \int_0^S{ \partial Y(-s;Y(s;\cdot))c(Y(s;\cdot)) \cdot {^t} \partial Y(s;\cdot)(\nabla\psi)(Y(s;\cdot)) }\mathrm{d}s\\
&= \dfrac{1}{S} \int_0^S{ (c\cdot \nabla \psi)(Y(s;\cdot))}\mathrm{d}s \\
&= \left< c\cdot\nabla\psi\right>.
\end{align*}
In the previous computations, we utilized the equality
$
 Y(-s;Y(s;y)) = y,\,\,y\in\R^m
$
which, upon differentiation with respect to $y$, implies
$
\partial_y Y(-s;Y(s;\cdot)\partial_y Y(s;\cdot) = I_m.
$
Similarly, the condition $[a,b] =0$ expresses the commutation between the flows associated to the vector fields $a$ and $b$
\begin{equation}
\label{equ:EquComFlow}
Z(h;Y(s;y)) = Y(s;Z(h;y)),\,\, h,s \in\R, \,\, y\in\R^m,
\end{equation}
where $Z(h;y)$ denotes the characteristic flow associated to $a$
\[
\dfrac{\mathrm{d}}{\mathrm{d}h}Z(h;y) = a(Z(h;y)),\,\,(h,y)\in\R\times\R^m.
\]
Taking the derivative of \eqref{equ:EquComFlow} with respect to $h$ at $h=0$, we obtain
\[
a(Y(s;y)) = \dfrac{\mathrm{d}}{\mathrm{d}h}|_{h=0}Z(h;Y(s;y)) =  \dfrac{\mathrm{d}}{\mathrm{d}h}|_{h=0}Y(s;Z(h;y)) = \partial_y Y(s;y)a(y),\,\,(s,y)\in (\R\times\R^m).
\]
Hence we have
\begin{align*}
\left< a\cdot\nabla\theta \right> &= \dfrac{1}{S}\int_0^S{a(Y(s;\cdot))\cdot(\nabla\theta)(Y(s;\cdot))}\mathrm{d}s\\
&=  \dfrac{1}{S}\int_0^S{ a\cdot {^t}\partial_y Y(s;\cdot)(\nabla\theta)(Y(s;\cdot))}\mathrm{d}s\\
&= \dfrac{1}{S}\int_0^S{a\cdot \nabla(\theta(Y(s;\cdot)))}\mathrm{d}s\\
&= a\cdot\nabla \left<\theta\right>.
\end{align*}

We come back to the limit model \eqref{equ:gyro-kinetic}, \eqref{equ:constraint} and we consider a smooth magnetic field $Be\cdot \nabla_x$, whose characteristic flow is periodic, with a uniform period $S$. If we denote by $X=X(s;x)$ the flow of the magnetic field, we have by $S$ periodicity
\[
\left< Be\cdot \nabla_x p \right> = \dfrac{1}{S}\int_0^S{(Be\cdot \nabla_x p)(X(s;\cdot))}\mathrm{d}s= \dfrac{1}{S}\int_0^S{\dfrac{\mathrm{d}}{\mathrm{d}s}\left\{ p(X(s;\cdot)) \right\}}\mathrm{d}s =0.
\]
Therefore the  Lagrange multiplier $p$ can be eliminated by taking the average in \eqref{equ:gyro-kinetic} 
\begin{equation}
\label{equ:EquAveLim}
\partial_t \left< n \right> + \left< \Divx\left( \dfrac{ne}{\omega_c}\wedge \nabla_x k[n]\right) \right> =0.
\end{equation}
The difficulty task is how to express the average of the divergence term, with respect to $\left<n\right>$, such that we get a model for the new unknown $\left<n\right>$.
\begin{pro}
\label{ZeroAve}
For any zero average function $\alpha$, and constant along the flow $X$ function $\psi$, we have
\[
\left< \Divx \left( \dfrac{\alpha e}{B}\wedge \nabla\psi \right)\right>=0.
\]
\end{pro}
\begin{proof}
We are done if we prove that for any constant along the flow function $\theta$ we have
\begin{equation}
\label{equ:EquZeroWeak}
\int_{\R^3}{\Divx\left( \dfrac{\alpha e}{B}\wedge \nabla\psi \right)\theta(x)}\mathrm{d}x =0.
\end{equation}
As $e\cdot\nabla\psi =0$, $e\cdot \nabla\theta =0$, therefore we have $(I_3 -e\otimes e)(\nabla\theta \wedge \nabla\psi) =0$. The vector field $\nabla\theta \wedge \nabla\psi$ is divergence free
\[
\Divx(\nabla\theta \wedge \nabla\psi) = \nabla \psi \cdot \rot_x(\nabla\theta)-\nabla\theta \cdot \rot_x(\nabla\psi) =0,
\]
and therefore there is a constant function $\lambda$ along the flow $X$ such that $\nabla\theta \wedge \nabla\psi = \lambda Be$. We deduce that
\begin{align*}
\int_{\R^3}{\Divx\left( \dfrac{\alpha e}{B}\wedge \nabla\psi \right)\theta(x)} \mathrm{d}x
&= \int_{\R^3}{(\nabla\theta\wedge\nabla\psi)\cdot\dfrac{\alpha e}{B}}\mathrm{d}x\\
&= \int_{\R^3}{\lambda \alpha}\mathrm{d}x = \int_{\R^3}{\lambda\left<\alpha\right>}\mathrm{d}x =0,
\end{align*}
and therefore \eqref{equ:EquZeroWeak} holds true.
\end{proof}

Applying Proposition \ref{ZeroAve} with the function $k[n]$, which is constant along the flow of $Be\cdot\nabla_x$, we obtain
\[
\left< \Divx\left( \dfrac{ne}{\omega_c}\wedge \nabla_x k[n]\right) \right> = \left< \Divx\left( \dfrac{\left<n\right>e}{\omega_c}\wedge \nabla_x k[n]\right) \right>.
\]
We also need to express $k[n] = \sigma(1+\ln n)+ \frac{q}{m}\Phi[n]$, with respect to $\left< n\right>$, where the concentration $n$ is such that the constraint \eqref{equ:constraint} holds true.
\begin{lemma}
\label{FirstVar}
The first variation of the free energy 
\[
\calE [n] = \int_{\R^3}{\sigma n\ln n + \dfrac{\epsilon_0}{2m}|\nabla_x \Phi[n]|^2}\mathrm{d}x
\]
 is $k[n]=\sigma(1+\ln n)+ \frac{q}{m}\Phi[n]$. For any concentration $n,n_0 \geq 0$, we have
\begin{align*}
\calE [n] - \calE [n_0] - \int_{\R^3}{k[n_0](n-n_0)}\mathrm{d}x &= \sigma \int_{\R^3}{n_0 \left(\dfrac{n}{n_0}\ln \dfrac{n}{n_0} -\dfrac{n}{n_0} +1 \right)}\mathrm{d}x\\
&\quad +\dfrac{\epsilon_0}{2m}\int_{\R^3}{|\nabla_x \Phi[n] - \nabla_x \Phi[n_0]|^2}\mathrm{d}x \geq 0,
\end{align*}
with equality iff $n=n_0$.
\end{lemma}
\begin{proof}
By direct computations, one gets
\begin{align*}
\calE [n] &-\calE [n_0] - \int_{\R^3}{k[n_0](n-n_0)}\mathrm{d}x \\
&= \sigma \int_{\R^3}{n_0 \left( \dfrac{n}{n_0}\ln \dfrac{n}{n_0} -\dfrac{n}{n_0} +1 \right)}\mathrm{d}x + \dfrac{\epsilon_0}{2m}\int_{\R^3}{|\nabla_x \Phi[n] - \nabla_x \Phi[n_0]|^2}\mathrm{d}x \geq 0,
\end{align*}
which equality iff $n=n_0$. Obviously, we have
\begin{align*}
&\lim_{h\to 0} \dfrac{\calE [n_0 + hz]-\calE[n_0]-h\int_{\R^3}{k[n_0]z}}{h}\mathrm{d}x\\
&= \lim_{h\to 0} \dfrac{\sigma}{h}\int_{\R^3}{n_0\left( \dfrac{n_0 + hz}{n_0}\ln \dfrac{n_0 +hz}{n_0} -\dfrac{n_0 +hz}{n_0} +1 \right)}\mathrm{d}x + \lim_{h\to 0} \dfrac{\epsilon_0}{2mh}\int_{\R^3}{h^2|\nabla_x\Phi[z]|^2}\mathrm{d}x=0,
\end{align*}
saying that $\lim_{h\to 0} h^{-1}(\calE [n_0 + hz]- \calE[n_0]) = \int_{\R^3}{k[n_0]z}\mathrm{d}x$.
\end{proof}

Thanks to the previous lemma, we deduce that there is at most one concentration $n$ with a given average, such that $Be\cdot \nabla_x k[n] =0$.
\begin{lemma}
\label{Uniqueness}
Let $n_1, n_2$ be two concentrations such that $\left< n_1\right> = \left< n_2 \right> $  and $Be\cdot\nabla_x k[n_1] = Be\cdot \nabla_x k[n_2]$. Therefore we have $n_1 = n_2$. In particular, for a given average, there is at most one concentration $n$ such that $Be\cdot \nabla_x k[n]=0$.
\end{lemma}
\begin{proof}
We have by Lemma \ref{FirstVar}
\begin{align*}
\calE [n_1] - \calE [n_2] - \int_{\R^3}{k[n_2](n_1 - n_2)}\mathrm{d}x &= \sigma \int_{\R^3}{n_2 \left(\dfrac{n_1}{n_2}\ln \dfrac{n_1}{n_2} -\dfrac{n_1}{n_2} +1 \right)}\mathrm{d}x\\
&\quad +\dfrac{\epsilon_0}{2m}\int_{\R^3}{|\nabla_x \Phi[n_1] - \nabla_x \Phi[n_2]|^2}\mathrm{d}x,
\end{align*}
and
\begin{align*}
\calE [n_2] - \calE [n_1] - \int_{\R^3}{k[n_1](n_2 - n_1)}\mathrm{d}x &= \sigma \int_{\R^3}{n_1 \left(\dfrac{n_2}{n_1}\ln \dfrac{n_2}{n_1} -\dfrac{n_2}{n_1} +1 \right)}\mathrm{d}x\\
&\quad +\dfrac{\epsilon_0}{2m}\int_{\R^3}{|\nabla_x \Phi[n_2] - \nabla_x \Phi[n_1]|^2}\mathrm{d}x,
\end{align*}
implying that 
\begin{align*}
\int_{\R^3}{(k[n_1]-k[n_2])(n_1 - n_2)}\mathrm{d}x = \sigma \int_{\R^3}{(n_1 -n_2)\ln \left(\dfrac{n_1}{n_2} \right)}\mathrm{d}x +\dfrac{\epsilon_0}{m}\int_{\R^3}{|\nabla_x \Phi[n_2] - \nabla_x \Phi[n_1]|^2}\mathrm{d}x.
\end{align*}
Since $Be\cdot \nabla_x (k[n_1]-k[n_2]) =0$, $\left< n_1 - n_2 \right> =0$, we deduce
\[
\int_{\R^3}{(k[n_1]-k[n_2])(n_1 - n_2)}\mathrm{d}x =0,
\]
and thus $n_1 = n_2$.
\end{proof}

If $n$ is such that $Be\cdot \nabla_x k[n] = 0$, then for any concentration $\bar{n}$ having the same average as $n$ we have
\[
\calE [\bar{n}] \geq \calE [n] + \intxt{k[n](\bar{n} - n)} = \calE [n],
\]
saying that for any given average $a$, the unique concentration $n$ such that $\left< n \right> =a$ and $Be\cdot\nabla_x k[n]=0$, satisfies
\[
\calE [n] = \min_{\left< \bar{n}\right>=a}\calE[\bar{n}].
\]
We denote by $F$ the application which maps $a\in \mathrm{ker}(Be\cdot\nabla_x)$ to $n$ such that $\left< n \right> =a$, $Be\cdot\nabla_x k[n] =0$.
\begin{lemma}
\label{FirstVarBis}
The application $a\in \mathrm{ker}(Be\cdot\nabla_x) \to \calE [n=F(a)] $ is convex and its first variation is $a\to k[n=F(a)]$.
\end{lemma}
\begin{proof}
Consider $a_1, a_2\in\mathrm{ker}(Be\cdot\nabla_x) $ and $\lambda_1, \lambda_2\in[0,1]$ such that $\lambda_1 + \lambda_2 =1$. We have
\[
\lambda_1 \calE [F(a_1)] + \lambda_2 \calE [F(a_2)] \geq \calE [\lambda_1 F(a_1) + \lambda_2 F(a_2)],
\]
since $\calE$ is convex and 
\[
\calE [F(\lambda_1 a_1 + \lambda_2 a_2)] = \min_{\left< \bar{n}\right> = \lambda_1 a_1 + \lambda_2 a_2}\calE [\bar{n}] \leq \calE [\lambda_1 F(a_1) + \lambda_2 F(a_2)],
\]
because
\[
\left< \lambda_1 F(a_1) + \lambda_2 F(a_2)  \right> = \lambda_1 \left< F(a_1) \right> + \lambda_2 \left< F(a_2)\right> = \lambda_1 a_1 + \lambda_2 a_2.
\]
Consider now $a, z\in \mathrm{ker}(Be\cdot\nabla_x)$ and $h\in\R$. The convexity of $\calE$ implies
\begin{align*}
\calE [F(a+hz)] - \calE[F(a)] &\geq \int_{\R^3}{k[F(a)][F(a+hz)-F(a)]}\mathrm{d}x\\
&= \int_{\R^3}{k[F(a)]\left< F(a+hz)-F(a) \right>}\mathrm{d}x\\
&= h \int_{\R^3}{k[F(a)]z(x)}\mathrm{d}x.
\end{align*}
Passing to the limit when $h\searrow 0$ and $h\nearrow 0$, we deduce that
\[
\lim_{h\to 0}\dfrac{\calE[F(a+hz)]-\calE[F(a)]}{h} = \int_{\R^3}{k[F(a)]z}\mathrm{d}x.
\]
\end{proof}

Combining the results in Proposition \ref{ZeroAve}, Lemma \ref{FirstVarBis}, the limit model \eqref{equ:gyro-kinetic}, \eqref{equ:constraint} becomes
\[
\partial_t a + \left<\Divx\left(\dfrac{ae}{\omega_c}\wedge \nabla_x k[F(a)]\right) \right> =0,\quad n=F(a).
\]
As $k[F(a)]\in \mathrm{ker}(Be\cdot\nabla_x)$, we obtain by \eqref{equ:AveConsFlowY}
\[
\left<\Divx\left(\dfrac{ae}{\omega_c}\wedge \nabla_x k[F(a)]\right)\right> = \left<\rot_x\left( \dfrac{ae}{\omega_c}\right)\cdot \nabla_x k[F(a)] \right> = \left<\rot_x\left( \dfrac{ae}{\omega_c}\right) \right>\cdot\nabla_x k[F(a)],
\]
and therefore the previous limit model also writes
\begin{equation}
\label{equ:AveRot}
\partial_t a + \left< \rot_x\left( \dfrac{ae}{\omega_c}\right) \right>\cdot\nabla_x k[F(a)]=0,\quad n=F(a).
\end{equation}

\section{A commutation formula for angular vector fields}
\label{AngVectFields}
The last step will concern a commutation formula between the operators $\left<\cdot\right>$ and $\rot_x$. We establish this formula for the special class of vector fields which present angle variables. In particular, this formula will apply to a special case of the magnetic field relevant to tokamaks.
\subsection{A commutation formula}
We start with a very simple example. Consider the vector field $b(y) \cdot\nabla_y = y_2\partial_{y_1} - y_1 \partial_{y_2}, y=(y_1,y_2)\in\R^2$, whose characteristic flow is $2\pi$-periodic
\begin{align*}
Y(s;y) = \calR(-s)y = \begin{pmatrix}
\cos s & \sin s\\
-\sin s &  \cos s
\end{pmatrix}y ,\,\,(s,y)\in\R\times\R^2.
\end{align*}
The gradient of any invariant function $\psi$, that is a function satisfying $\psi(Y(s;\cdot)) = \psi, s\in\R$, verifies
\begin{equation}
\label{equ:GradInvFunc}
{^t}\partial Y(s;\cdot)(\nabla\psi)(Y(s;\cdot)) = \nabla\psi,\,\, s\in\R.
\end{equation}
There are other vector fields verifying similar properties. Let us consider the angle $\theta = \theta(y)\in [0,2\pi[$ given by
\[
y_1 = |y|\cos\theta(y),\,\, y_2 = |y|\sin\theta(y),\,\, y\in\R^2 \backslash\left\{ (0,0)\right\}.
\]
The function $\theta$ is smooth in $D = \R^2 \backslash (\R_+ \times \left\{0\right\})$ and we have
\[
\nabla_y\theta = - \dfrac{(y_2,-y_1)}{|y|^2} = -\dfrac{b(y)}{|y|^2},\,\,y\in D.
\]
The function $\theta$ is discontinuous across $\R^\star _+ \times \left\{ 0 \right\}$
\[
\lim_{y_1\to z_1, y_2\searrow 0} \theta(y) = 0,\,\,\,\,\,\lim_{y_1\to z_1, y_2 \nearrow 0} \theta(y) =2\pi,\,\,z_1 >0,
\]
but its gradient, which is well defined on $D$ is the restriction of a smooth vector field on $\R^2 \backslash\left\{ (0,0)\right\}$
\[
\nu(y) = - \dfrac{(y_2,-y_1)}{|y|^2},\,\, y\in \R^2 \backslash\left\{ (0,0)\right\}.
\]
For any $y\in D$ and $|s|$ small enough, we have
\[
\dfrac{\mathrm{d}}{\mathrm{d}s} \theta(Y(s;y)) = b(Y(s;y))\cdot (\nabla\theta)(Y(s;y)) =-1,
\]
implying that $\theta(Y(s;y)) = \theta(y) -s ,y\in D$ and $|s|$ small enough. Taking the gradient with respect to $y$ we obtain
\[
{^t}\partial Y(s;y)(\nabla\theta)(Y(s;y)) = \nabla\theta(y),
\]
or 
\begin{equation}
\label{equ:GradAngVect}
{^t}\partial Y(s;y)\nu(Y(s;y)) = \nu(y),\,\, y\in D,\,\, |s| \,\mathrm{small\,enough}.
\end{equation}
Actually it is easily seen that the previous formula holds true for any $y\in\R^2\backslash\left\{ (0,0)\right\}$ and $s\in\R$. The vector field $\nu$ also satisfies
$
\mathrm{div}_y {^t}\nu (y) =0,
$
but it is not the gradient of a smooth function $\tilde{\theta}$ on $\R^2\backslash\left\{ (0,0)\right\}$, because, in that case, for any $y\in\R^2\backslash\left\{ (0,0)\right\}$, we would obtain
\[
-1 = \dfrac{1}{2\pi}\int_0^{2\pi}{(b\cdot\nu)(Y(s;y))}\mathrm{d}s = \dfrac{1}{2\pi}\int_0^{2\pi}{(b\cdot\nabla\tilde{\theta})(Y(s;y))}\mathrm{d}s =  \dfrac{1}{2\pi}\int_0^{2\pi}{\dfrac{\mathrm{d}}{\mathrm{d}s}\tilde{\theta}(Y(s;y))}\mathrm{d}s =0.
\]

Generally, given a smooth divergence free vector field $b\cdot\nabla_y$ in $\R^3$, with global characteristic flow $Y=Y(s;y), (s,y)\in\R\times\R^3$, we call angular vector field in $D\in\R^3$ any vector field $\nu\cdot\nabla_y$ satisfying
\[
b(y)\cdot \nu(y) = C,\quad {^t}\partial Y(s;y)\nu(Y(s;y)) = \nu(y),\quad \rot_y\nu =0,\,\,(s,y)\in\R\times D,
\]
for some constant $C\in \R^\star$, where $D$ is an open subset of $\R^3$, which is left invariant by the flow i.e., $Y(s;D)=D, s\in\R.$ We intend to establish the following commutation formula.
\begin{pro}
\label{AngField}
Let us consider a vector field $b\cdot\nabla_y$ in $\R^3$ satisfying \eqref{equ:EquLipDiv}, \eqref{equ:EquGrowth} with $S$-periodic characteristic flow $Y=Y(s;y)$, $(s,y)\in \R\times\R^3$. We denote by $\eta \cdot \nabla_y$ the gradient of an invariant function with respect to the flow $Y$, or an angular vector field, in some open subset $D$ of $\R^3$, which is left invariant by the flow $Y$. Therefore, for any $C^1$ function $\alpha = \alpha(y)$, we have
\begin{equation}
\label{equ:AngField}
\left< \nabla_y \alpha \wedge \eta \right> = \nabla_y \left< \alpha\right>\wedge \eta\,\, \mathrm{in}\,\, D.
\end{equation}
In particular, if $\alpha\in\mathrm{ker}(b\cdot\nabla_y)$, then $( \nabla_y \alpha \wedge \eta )\cdot \nabla_y$ is in involution with respect to $b\cdot\nabla_y$ in $D$.
\end{pro}
We will use the following lemmas.
\begin{lemma}
\label{VectProd}
We denote by $M[e]$ the matrix of the linear transformation $v \to e\wedge v, v\in\R^3$, that is $M[e]v = e \wedge v, v\in\R^3$. For any $e\in \mathbb{S}^2$, and $\xi, \eta\in \R^3$ such that $\xi\cdot e =0$, we have
\[
\xi\wedge \eta = (e\otimes M[e]\xi - M[e]\xi\otimes e)\eta.
\]
\end{lemma}
\begin{proof}
By direct computations one gets
\begin{align*}
(e\otimes M[e]\xi - M[e]\xi\otimes e)\eta &= ((e\wedge \xi)\cdot\eta)e - (e\cdot\eta)e\wedge \xi \\
&= ((\xi\wedge \eta)\cdot e)e + (\eta\cdot e)\xi\wedge e \\
&= e\otimes e(\xi\wedge (\eta-(\eta\cdot e)e)) + (\eta\cdot e)\xi\wedge e \\
&= \xi\wedge (\eta-(\eta\cdot e)e) + (\eta\cdot e)\xi\wedge e \\
&= \xi \wedge \eta,
\end{align*}
where we have used that $\xi \wedge (\eta - (\eta\cdot e)e) \in \R e$, since $\xi \cdot e =0$.
\end{proof}

For any function or vector field, the notation $F_s$ stands for $F\circ Y(s;\cdot)$.
\begin{lemma}
\label{Identity}
Let us consider a vector field $b\cdot \nabla_y$ in $\R^3$ satisfying \eqref{equ:EquLipDiv}, \eqref{equ:EquGrowth} with $S$-periodic characteristic flow $Y=Y(s;y)$, $(s,y)\in \R\times\R^3$. We denote by $M[e]$ the matrix of the linear transformation $v \to e\wedge v, v\in\R^3$, that is $M[e]v = e \wedge v, v\in\R^3$. Then, for any function u such that $u\in \mathrm{ker}(b\cdot\nabla_y)$,
we have the equality
\[
(I_3 -e_s\otimes e_s)\dfrac{\partial Y(s;\cdot)M[e]{^t}\partial Y(s;\cdot)}{|b|}(I_3 -e_s\otimes e_s)(\nabla u)_s = \dfrac{M[e_s]}{|b_s|} (\nabla u)_s,\,\, e=\dfrac{b}{|b|}.
\]
\end{lemma}
\begin{proof}
For any invariant functions $\alpha = \alpha(y), \beta = \beta(y)$ with respect to the flow $Y$ we have $\nabla_y \alpha \wedge \nabla_y \beta \in \R e$ and $\mathrm{div}_y (\nabla_y \alpha \wedge \nabla_y \beta) = 0$. Therefore there is $\lambda \in \mathrm{ker}(b\cdot \nabla_y)$ such that $\nabla_y \alpha \wedge \nabla_y \beta = \lambda b$, saying that the vector field $\nabla_y \alpha \wedge \nabla_y \beta$ is in involution with respect to $b\cdot\nabla_y$. We have
\[
\partial Y(s;\cdot) \nabla \alpha \wedge \nabla \beta = (\nabla \alpha)_s \wedge (\nabla \beta)_s.
\]
Therefore, by Lemma \ref{VectProd} we obtain
\[
\partial Y(s;\cdot)(e\otimes M[e]\nabla \alpha - M[e]\nabla \alpha\otimes e)\nabla \beta = (e_s\otimes M[e_s](\nabla \alpha)_s - M[e_s](\nabla \alpha)_s\otimes e_s)(\nabla \beta)_s ,
\]
which reduce, thanks to the equalities $e\cdot \nabla\beta =0$, $e_s\cdot (\nabla\beta)_s =0$, to
\[
\partial Y(s;\cdot)(e\otimes M[e]\nabla\alpha)\nabla\beta = (e_s\otimes M[e_s](\nabla \alpha)_s)(\nabla\beta)_s .
\]
As $\alpha$ and $\beta$ are left invariant by the flow $Y$, we have 
\[
\partial Y(s;\cdot)(e\otimes M[e]{^t}\partial Y(s;\cdot)(\nabla\alpha)_s){^t}\partial Y(s;\cdot)(\nabla\beta)_s = (e_s\otimes M[e_s](\nabla \alpha)_s)(\nabla\beta)_s .
\]
Therefore we obtain
\[
\left( e_s \otimes \dfrac{\partial Y(s;\cdot)M[e]{^t}\partial Y(s;\cdot)}{|b|}(\nabla\alpha )_s  \right)(\nabla\beta)_s = \left( e_s \otimes \dfrac{M[e_s](\nabla\alpha)_s}{|b_s|} \right)(\nabla\beta)_s ,
\]
or equivalently
\[
\left(\dfrac{\partial Y(s;\cdot)M[e]{^t}\partial Y(s;\cdot)}{|b|} -  \dfrac{M[e_s]}{|b_s|} \right) (\nabla \alpha)_s \in \R e_s .
\]
Finally, we have 
\[
(I_3 - e_s \otimes e_s) \left(\dfrac{\partial Y(s;\cdot)M[e]{^t}\partial Y(s;\cdot)}{|b|} -  \dfrac{M[e_s]}{|b_s|} \right) (\nabla \alpha)_s  =0 ,
\]
for any invariant function $\alpha$, and our conclusions follows.
\end{proof}
\begin{lemma}
\label{ZeroAveVectField}
Let us consider a vector field $b\cdot\nabla_y$ in $\R^3$ satisfying \eqref{equ:EquLipDiv}, \eqref{equ:EquGrowth} with $S$-periodic characteristic flow $Y=Y(s;y),(s,y)\in\R\times\R^3$, which possesses angular vector field $\nu$ in some invariant open subset $D \subset \R^3$. A vector field $c\cdot\nabla_y$ has zero average in $D$ iff $\left< c\cdot\nu \right> =0$ in $D$ and $\left<  c\cdot \nabla_y u\right> =0$ in $D$ for any function $u$ such that $\mathds{1}_D u \in\mathrm{ker}(b\cdot\nabla_y)$.
\end{lemma}
\begin{proof}
By formula \eqref{equ:AveConsFlowY} we know that for any function $u$ which is left invariant by $Y$ in $D$, we have $\left<c\cdot\nabla_y u\right>= \left< c\right>\cdot\nabla_y u$ in $D$. Similarly, for any $y\in D$ we write
\begin{align*}
\left< c\cdot\nu\right>(y) &= \dfrac{1}{S}\int_0^S{c(Y(s;y))\cdot \nu(Y(s;y))}\mathrm{d}s\\
&= \dfrac{1}{S} \int_0^S{\partial Y(-s;Y(s;y)) c(Y(s;y)) \cdot {^t} \partial_y Y(s;y)\nu(Y(s;y))}\mathrm{d}s \\
&= \dfrac{1}{S} \int_0^S{\partial Y(-s;Y(s;y)) c(Y(s;y))}\mathrm{d}s\cdot \nu(y) \\
&= \left< c\right>(y) \cdot \nu(y).
\end{align*}
Clearly, if $\left< c\right> =0$ in $D$, then $\mathds{1}_D \left< c\cdot \nabla_y u \right> =0$ for any function $u$ such that $\mathds{1}_D u\in \mathrm{ker}(b\cdot\nabla_y)$ and $\mathds{1}_D \left< c \cdot \nu\right> =0$. Conversely, if $\mathds{1}_D \left< c \cdot \nabla_y u\right> =0$ for any function $u$ such that $\mathds{1}_D u \in\mathrm{ker}(b\cdot\nabla_y)$ and $\mathds{1}_D \left< c \cdot \nu\right> =0$, then $\mathds{1}_D \left< c\right>\cdot\nabla_y u =0$, $\mathds{1}_D \left< c\right>\cdot \nu =0$. We deduce that there is a function $\lambda = \lambda(y)$ in $D$ such that
$
\left< c\right>(y)= \lambda(y)b(y),\,\, y\in D.
$
Taking the scalar product by $\nu(y), y\in D$, we obtain
\[
0 = \left< c\right>(y)\cdot \nu(y) = \lambda(y) b(y)\cdot \nu(y) = \lambda(y) C,\,\, y\in D.
\]
Since $C\in \R^\star$, we deduce that $\lambda$ vanishes in $D$ and $\mathds{1}_D\left< c\right> =0$.
\end{proof}

We are ready to prove the commutation formula \eqref{equ:AngField}.
\begin{proof}(of Proposition \ref{AngField})
All the computations are performed in $D$.\\
 We assume for the moment that $\alpha\in \mathrm{ker}(b\cdot\nabla_y)$ and we prove that $\nabla \alpha \wedge \eta$ is in involution with respect to $b\cdot\nabla_y$. We have by Lemma \ref{VectProd} and Lemma \ref{Identity}
\begin{align*}
\partial Y(s;\cdot)&(\nabla \alpha \wedge \eta) = \partial Y(s;\cdot)(e\otimes M[e]\nabla\alpha - M[e]\nabla\alpha\otimes e)\eta \\
&= \left[ b_s \otimes  \dfrac{\partial Y(s;\cdot)M[e]{^t}\partial Y(s;\cdot)}{|b|}(\nabla\alpha )_s - \dfrac{\partial Y(s;\cdot)M[e]{^t}\partial Y(s;\cdot)}{|b|}(\nabla\alpha )_s \otimes b_s \right]\eta_s \\
&= \left[ b_s \otimes (I_3 - e_s\otimes e_s)\dfrac{\partial Y(s;\cdot)M[e]{^t}\partial Y(s;\cdot)}{|b|}(I_3 - e_s\otimes e_s)(\nabla\alpha )_s  \right.\\
&\left.  \quad- (I_3 - e_s\otimes e_s)\dfrac{\partial Y(s;\cdot)M[e]{^t}\partial Y(s;\cdot)}{|b|}(I_3 - e_s\otimes e_s) (\nabla\alpha )_s \otimes b_s \right] \eta_s \\
&= (e_s \otimes M[e_s](\nabla \alpha)_s - M[e_s](\nabla \alpha)_s\otimes e_s)\eta_s \\
&= (\nabla\alpha)_s \wedge \eta_s.
\end{align*}
Assume now that $\left< \alpha \right> =0$ and we prove that $\left<\nabla\alpha \wedge \eta \right>=0$. 
If $\eta = \nabla \beta$ for some function $\beta$, which is left invariant by $Y$ in $D$ we have
\begin{align*}
&\partial Y (-s;Y(s;\cdot))\eta_s \wedge (\nabla\alpha)_s \\
&= \partial Y (-s;Y(s;\cdot))(e_s\otimes M[e_s](\nabla\beta)_s - M[e_s](\nabla\beta)_s\otimes e_s)(\nabla\alpha)_s\\
&= \left( b \otimes  \dfrac{\partial Y(-s;Y(s;\cdot))M[e_s]{^t}\partial Y(-s;Y(s;\cdot))}{|b_s|}\nabla\beta \right. \\
&\left. \quad - \dfrac{\partial Y(-s;Y(s;\cdot))M[e_s]{^t}\partial Y(-s;Y(s;\cdot))}{|b_s|}\nabla\beta \otimes b \right)\nabla(\alpha_s) \\
&= \left[ b \otimes (I_3 - e\otimes e)\dfrac{\partial Y(-s;Y(s;\cdot))M[e_s]{^t}\partial Y(-s;Y(s;\cdot))}{|b_s|}(I_3 - e\otimes e)\nabla\beta  \right.\\
&\left.  \quad - (I_3 - e\otimes e)\dfrac{\partial Y(-s;Y(s;\cdot))M[e_s]{^t}\partial Y(-s;Y(s;\cdot))}{|b_s|}(I_3 - e\otimes e) \nabla\beta \otimes b \right] \nabla\alpha_s  \\
&= \left( b\otimes \dfrac{M[e]}{|b|}\nabla\beta - \dfrac{M[e]}{|b|} \nabla\beta\otimes b\right)\nabla\alpha_s\\
&= (e \otimes M[e]\nabla\beta - M[e]\nabla\beta \otimes e)\nabla\alpha_s \\
&= \nabla\beta \wedge \nabla\alpha_s.
\end{align*}
We obtain
\[
\left< \nabla\beta \wedge \nabla\alpha \right> = \dfrac{1}{S}\int_0^S{\nabla\beta \wedge \nabla\alpha _s}\mathrm{d}s = \nabla\beta \wedge \nabla \left< \alpha \right> =0.
\]
If $\eta$ is an angular vector field $\nu$ in $D$, we appeal to Lemma \ref{ZeroAveVectField}. Obviously, we have $\left< (\nabla\alpha\wedge \nu)\cdot \nu\right> =0$ and for any function $u$ such that $\mathds{1}_D u\in\mathrm{ker}(b\cdot\nabla_y)$, we can write since $(\nabla_y u\wedge \nu)\cdot \nabla_y$ is in involution with $b\cdot\nabla_y$ in $D$ cf. the first part of this proof, and thanks to \eqref{equ:AveInvoFieldb}
\[
\left< (\nabla_y \alpha\wedge \nu)\cdot \nabla_y u\right> = - \left< (\nabla_y u \wedge \nu )\cdot \nabla_y \alpha\right> = - (\nabla_y u \wedge \nu )\cdot \nabla_y \left<\alpha\right> =0.
\]
Therefore, we deduce that
\[
\left< \nabla\alpha \wedge \nu \right> =0.
\]
Finally, for any function $\alpha$ we have
\begin{align*}
\left< \nabla\alpha \wedge \eta \right> &= \left< \nabla \left<\alpha\right> \wedge \eta \right> + \left< \nabla(\alpha -\left<\alpha\right>)  \wedge \eta \right> \\&= \dfrac{1}{S}\int_0^S{\partial Y(-s;Y(s;\cdot))(\nabla \left<\alpha\right>)_s \wedge \eta_s}\mathrm{d}s\\
&= \dfrac{1}{S} \int_0^S{\partial Y(-s;Y(s;\cdot))\partial Y(s;\cdot)(\nabla\left<\alpha\right> \wedge \eta)}\mathrm{d}s\\
&= \dfrac{1}{S} \int_0^S{\nabla\left<\alpha\right> \wedge \eta}\mathrm{d}s\\& = \nabla\left<\alpha\right> \wedge \eta \mathrm{d}s.
\end{align*}
\end{proof}

\subsection{Cylindrical case}
\label{Cylin}
In this subsection, we apply the previous results to a simplified framework, focusing a magnetic field whose magnetic lines wind on cylindrical surfaces.

We consider the magnetic field $ Be = B_0 \left( \frac{x_2}{R_0},-\frac{x_1}{R_0},1 \right)$, $x=(x_1,x_2,x_3)=(\bar{x},x_3)\in \R^3$, where $B_0,R_0$ are some reference values for the magnetic field and length. The characteristic flow is given by
\[
(\bar{X}(s;\bar{x}),X_3(s;x_3)) = \left( \calR\left(-s \dfrac{B_0}{R_0}\right)\bar{x},x_3 + sB_0\right),\,\,(s,\bar{x},x_3)\in \R\times\R^3 ,
\]
where 
\[
\calR(\theta) = \begin{pmatrix}
\cos \theta & -\sin \theta\\
\sin \theta &  \cos \theta
\end{pmatrix},\,\,\theta\in\R.
\]
We have two angular vector fields
\[
\nu_\theta = \dfrac{(x_2,-x_1,0)}{x_1 ^2 + x_2 ^2},\, \bar{x}\neq 0,\quad \nu_{\parallel} = (0,0,1).
\]
All the functions are supposed periodic with respect to $x_3$. Taking $S=2\pi R_0/B_0$, we define the average operator for a function $u$ by
\[
\left< u \right>(x) = \dfrac{1}{S}\int_0^S{u(\bar{X}(s;\bar{x}),X_3(s;x_3))}\mathrm{d}s = \dfrac{1}{S}\int_0^S{u\left(\calR\left(-s \dfrac{2\pi}{S}\right)\bar{x}, x_3 + s\dfrac{2\pi}{S}R_0 \right)}\mathrm{d}s,
\]
and for a vector field $c\cdot\nabla_x = \bar{c}\cdot\nabla_{\bar{x}} + c_3 \partial x_3$ by
\begin{align*}
\left< c\right>(x) &= \dfrac{1}{S}\int_0^S{\begin{pmatrix}
   \begin{array}{cr}
\calR(s\frac{2\pi}{S}) & \begin{matrix} 0\\ 0 \end{matrix} \\
  \begin{matrix} 0 && 0 \end{matrix} & 1
   \end{array}
\end{pmatrix}c\left(\calR\left(-s \dfrac{2\pi}{S}\right)\bar{x}, x_3 + s\dfrac{2\pi}{S}R_0 \right)}\mathrm{d}s\\
&= \dfrac{1}{S}\int_{0}^{S}{\begin{pmatrix} \calR(s\frac{2\pi}{S})\bar{c}\left(\calR\left(-s \frac{2\pi}{S}\right)\bar{x}, x_3 + s\frac{2\pi}{S}R_0 \right)\\
c_3\left(\calR\left(-s \frac{2\pi}{S}\right)\bar{x}, x_3 + s\frac{2\pi}{S}R_0 \right)
\end{pmatrix}}\mathrm{d}s.
\end{align*}
We use the following decomposition of $Be\cdot\nabla_x$
\[
Be = \dfrac{B_0}{R_0}|\bar{x}|^2 \nu_{\theta} + B_0 \nu_{\parallel},\,\, |\bar{x}|>0.
\]
Thanks to Proposition \ref{AngField}, we compute the term $\left< \rot_x\left( \frac{ae}{\omega_c}\right) \right>$ appearing in the limit model \eqref{equ:AveRot}. Observe that
\[
\rot_x\left( \dfrac{ae}{\omega_c}\right) = \rot_x\left[ \dfrac{a}{B\omega_c}\left(\dfrac{B_0}{R_0}|\bar{x}|^2 \nu_{\theta} + B_0 \nu_{\parallel} \right) \right] = \nabla_x \left( \dfrac{aB_0|\bar{x}|^2}{B\omega_cR_0}\right)\wedge \nu_\theta + \nabla_x \left( \dfrac{aB_0}{B\omega_c}\right)\wedge \nu_\parallel ,
\]
and therefore
\begin{align*}
\left< \rot_x\left( \dfrac{ae}{\omega_c}\right) \right> &= \nabla_x \left<\dfrac{aB_0 |\bar{x}|^2}{B\omega_c R_0} \right> \wedge \nu_\theta + \nabla_x \left<\dfrac{aB_0 }{B\omega_c } \right>\wedge \nu_\parallel \\
&= \nabla_x \left( \dfrac{aB_0 |\bar{x}|^2}{B\omega_c R_0} \right)\wedge \nu_\theta + \nabla_x \left(\dfrac{aB_0 }{B\omega_c } \right)\wedge \nu_\parallel  =  \rot_x\left( \dfrac{ae}{\omega_c}\right),
\end{align*}
since the functions $a$, $B\omega_c$ and $|\bar{x}|^2$ belong to $\mathrm{ker}(Be\cdot\nabla_x)$. We obtain
\[
\left< \rot_x\left( \dfrac{ae}{\omega_c}\right) \right> \cdot \nabla_x k[F(a)] = \rot_x\left( \dfrac{ae}{\omega_c}\right)\cdot \nabla_x k[F(a)] = \Divx\left(\dfrac{ae}{\omega_c}\wedge \nabla_x k[F(a)] \right).
\]
In that case, the vector field $\rot_x\left(\frac{ae}{\omega_c} \right)$ is in involution with $Be\cdot\nabla_x$, and \eqref{equ:AveRot} becomes
\[
\partial_t a + \Divx\left(\dfrac{ae}{\omega_c}\wedge \nabla_x k[F(a)] \right) =0,\quad n =F(a).
\]
In this case, we work in the $2\pi R_0$-periodic domain with respect to $x_3$, $\R^2 \times \mathbb{T}^1$, where $\T^1 = \R/(2\pi R_0 \Z)$. The potential $\Phi$ solves the Poisson equation
\[
-\epsilon_0 \Delta_x \Phi = qn,\,\, x\in \R^2\times\T^1,
\]
with the boundary condition
\[
\lim_{|\bar{x}|\to\infty}\Phi(\bar{x},x_3)=0,\,\, x_3\in \T^1.
\]
The Jacobian matrix of the flow $X(s,x) = (\bar{X}(s;\bar{x}),X_3(s;x_3))$ is orthogonal
\[
\partial_x X(s;x) = \begin{pmatrix}
   \begin{array}{cr}
\calR(-s\frac{B_0}{R_0}) & \begin{matrix} 0\\ 0 \end{matrix} \\
  \begin{matrix} 0 &&´ 0 \end{matrix} & 1
   \end{array}
\end{pmatrix},
\]
which implies that the Laplace operator commutes with the translations along the flow, meaning that
\[
\Delta_x u_s = (\Delta_x u)_s ,
\]
for any smooth function $u$. If $\Phi[n]$ is the potential corresponding to the $2\pi R_0$-periodic concentration $n$ with respect to $x_3$, then
\[
-\epsilon_0 \Delta_x(\Phi[n])_s = -\epsilon_0 (\Delta_x \Phi[n])_s = q n_s ,
\]
for any $x_3$ we have
\[
\lim_{|\bar{x}|\to +\infty} \Phi[n](X(s;x)) = \lim_{|\bar{x}|\to +\infty}\Phi[n](\bar{X}(s;\bar{x}),X_3(s;x_3)) =0,\,\, \mathrm{because}\, |\bar{X}(s,\bar{x})| = |\bar{x}|,
\]
and $(\Phi[n])_s$ is $2\pi R_0$-periodic with respect to $x_3$
\begin{align*}
\Phi[n](\bar{X}(s;\bar{x}),X_3(s;x_3 +2\pi R_0 )) &= \Phi[n](\bar{X}(s;\bar{x}),X_3(s;x_3) +2\pi R_0 )\\
&= \Phi[n](\bar{X}(s;\bar{x}),X_3(s;x_3 )) = (\Phi[n])_s (x).
\end{align*}
Therefore we have $(\Phi[n])_s = \Phi[n_s]$. In particular, if $n\in \mathrm{ker}(Be\cdot\nabla_x)$ then $\Phi[n]\in \mathrm{ker}(Be\cdot\nabla_x)$. By construction $n=F(a)$ is the unique concentration such that $\left< n\right> =a$, $Be\cdot\nabla_x k[n] =0$. Clearly we have $\left< a\right> =a$ and $k[a] = \sigma(1 + \ln a) + \frac{q}{m}\Phi[a]\in \mathrm{ker}(Be\cdot\nabla_x)$ and thus $n=F(a)=a$ for any $a \in \mathrm{ker}(Be\cdot\nabla_x)$. The constraint in \eqref{equ:constraint} is automatically satisfied. In that case, our limit model simply writes
\begin{equation}
\label{equ:EquCyl1}
\partial_t n + \Divx\left( \dfrac{n e}{\omega_c}\wedge \nabla_x k[n] \right) =0,\,\, (t,x)\in \R_+ \times \R^2\times \T^1.
\end{equation}

Since we know that at any time $t$, $n(t)$ belongs to $\mathrm{ker}(Be\cdot\nabla_x)$, we can reduce the above model to a two dimensional problem. We appeal to the invariants of the flow $X$
\[
\calR\left( \dfrac{X_3(s;x_3)}{R_0} \right)\bar{X}(s;\bar{x}) = \calR \left( \dfrac{x_3 +  s B_0}{R_0}\right)\calR \left( -s\dfrac{B_0}{R_0}\right)\bar{x} = \calR \left(\dfrac{x_3}{R_0} \right)\bar{x}.
\]
We introduce the new unknown function $N=N(t,\bar{y} =(y_1,y_2))$ such that
\[
n(t,x) =  N(t,\bar{y} = \calR(x_3 / R_0 )\bar{x}),
\]
and we are looking for the model satisfied by $N=N(t,\bar{y})$.
\begin{lemma}
\label{Laplace}
Let us consider a smooth function $U = U(\bar{y}),\bar{y}\in\R^2$, and $u(x) =U(\calR (x_3/ R_0)\bar{x})$, $x\in \R^2 \times \T^1$. We have
\[
\Delta_x u = \left[ \mathrm{div}_{\bar{y}}\left( I_2 + \dfrac{^\perp \bar{y} \otimes {^\perp}\bar{y}}{R_0 ^2} \right)\nabla_{\bar{y}}U \right](\bar{y}= \calR(x_3 /R_0)\bar{x}).
\]
\end{lemma}
\begin{proof}
Consider $\Psi \in C^1 _c(\R^2)$ and $\psi(x) = \Psi(\calR(x_3 /R_0)\bar{x})$, $x\in \R^2 \times \T^1$. Integrating by parts, thanks to the $x_3$-periodicity, one gets
\begin{align*}
\int_{\R^2\times\T^1}{\Delta_x u\, \psi(x)}\mathrm{d}x
&= -\int_{\R^2\times\T^1}{\dfrac{^t \partial \bar{y}}{\partial x}(\nabla_{\bar{y}}U)(\calR(x_3 /R_0)\bar{x})\cdot \dfrac{^t \partial \bar{y}}{\partial x}(\nabla_{\bar{y}}\Psi)(\calR(x_3 /R_0)\bar{x})}\mathrm{d}x\\
&= -\int_{\R^2\times\T^1}{\dfrac{ \partial \bar{y}}{\partial x}\dfrac{^t \partial \bar{y}}{\partial x}(\nabla_{\bar{y}}U)(\calR(x_3 /R_0)\bar{x})\cdot (\nabla_{\bar{y}}\Psi)(\calR(x_3 /R_0)\bar{x})}\mathrm{d}x,
\end{align*}
where $\frac{ \partial \bar{y}}{\partial x}$ is the Jacobian matrix of the apllication $x \to \calR(x_3 /R_0)\bar{x}$
\[
\dfrac{ \partial \bar{y}}{\partial x} = \left(\calR(x_3 /R_0), \calR(x_3 /R_0 + \pi/2)\dfrac{\bar{x}}{R_0}  \right)\in \calM_{2,3}(\R).
\]
The matrix product $\frac{ \partial \bar{y}}{\partial x}\frac{ ^t\partial \bar{y}}{\partial x}$ writes
\[
\frac{ \partial \bar{y}}{\partial x}\frac{ ^t\partial \bar{y}}{\partial x} = I_2 + \calR\left( \dfrac{x_3}{R_0}\right)\dfrac{^\perp \bar{x}}{R_0}\otimes  \calR\left( \dfrac{x_3}{R_0}\right)\dfrac{^\perp \bar{x}}{R_0},\,\, ^\perp\bar{x} = (x_2, -x_1),
\]
and we obtain
\begin{align*}
\int_{\R^2\times\T^1}{\Delta_x u \, \psi(x)} \mathrm{d}x
&=- \int_{\T^1}{\int_{\R^2}{\left( I_2 + \dfrac{^\perp \bar{y} \otimes {^\perp} \bar{y}}{R_0 ^2} \right)\nabla_{\bar{y}}U(\bar{y})\cdot\nabla_{\bar{y}}\Psi(\bar{y})} }\mathrm{d}\bar{y}\,\mathrm{d}x_3\\
&= 2\pi R_0 \int_{\R^2}{\left( \mathrm{div}_{\bar{y}}\left( I_2 + \dfrac{^\perp \bar{y} \otimes {^\perp} \bar{y}}{R_0 ^2} \right)\nabla_{\bar{y}}U \right) \Psi(\bar{y})}\mathrm{d}\bar{y}\\
&= \int_{\R^2\times\T^1}{\left[ \mathrm{div}_{\bar{y}}\left( I_2 + \dfrac{^\perp \bar{y} \otimes {^\perp} \bar{y}}{R_0 ^2} \right)\nabla_{\bar{y}}U \right]\left(\bar{y} = \calR\left( \dfrac{x_3}{R_0}\right)\bar{x} \right)\psi(x)}\mathrm{d}x.
\end{align*}
The previous computation shows that $\Delta_x u - \left[\mathrm{div}_{\bar{y}}\left( I_2 + \frac{^\perp \bar{y} \otimes {^\perp} \bar{y}}{R_0 ^2} \right)\nabla_{\bar{y}}U \right]\left(\bar{y} = \calR\left( \frac{x_3}{R_0}\right)\bar{x} \right) $  is orthogonal on $\mathrm{ker}(Be\cdot\nabla_x)$. But this function belongs to $\mathrm{ker}(Be\cdot\nabla_x)$, because $u$ belongs to $\mathrm{ker}(Be\cdot\nabla_x)$, together with $\Delta_x u$, since the Laplace operator commutes with the flow $X$. Finally, we obtain the desired result in the lemma.
\end{proof}
\begin{lemma}
\label{ConservLamw}
Let us consider two smooth functions $U= U(\bar{y}), W = W(\bar{y}), \bar{y}\in \R^2$ and $u(x) = U(\calR(x_3 /R_0)\bar{x})$, $w(x) = W(\calR(x_3 /R_0)\bar{x})$, $x\in \R^2 \times \T^1$. We have
\[
\Divx\left( \dfrac{u e}{\omega_c}\wedge\nabla_x w \right) =  \left[ \mathrm{div}_{\bar{y}}\left( \dfrac{U}{\omega_0}\calR(\pi/2)\nabla_{\bar{y}}W \right) \right]\left(\bar{y} = \calR\left( \frac{x_3}{R_0}\right)\bar{x} \right),\,\, \omega_0 =\dfrac{qB_0}{m}.
\]
\end{lemma}
\begin{proof}
As before, we perform the computation in a distribution sense. We already know that the vector field $\rot_x\left(\frac{ue}{\omega_c} \right)\cdot\nabla_x$ is in involution with $Be\cdot\nabla_x$, and therefore 
\[
\Divx\left( \dfrac{u e}{\omega_c}\wedge\nabla_x w \right) = \rot_x\left(\dfrac{ue}{\omega_c} \right)\cdot\nabla_x w \in \mathrm{ker}(Be\cdot\nabla_x),
\]
it is enough to consider test functions $\psi(x) = \Psi(\calR(x_3 / R_0)\bar{x}), \Psi \in C^1_c(\R^2)$
\begin{align*}
\int_{\R^2\times\T^1}&{\Divx\left( \dfrac{u e}{\omega_c}\wedge\nabla_x w \right) \psi(x)} \mathrm{d}x
=- \int_{\R^2\times\T^1}{\dfrac{u}{\omega_c}M[e]\nabla_x \omega \cdot \nabla_x \psi}\mathrm{d}x\\
&= - \int_{\R^2\times\T^1}{\dfrac{U(\calR(x_3 /R_0)\bar{x})}{\omega_c} \dfrac{\partial\bar{y}}{\partial x}M[e]\dfrac{^t\partial\bar{y}}{\partial x}(\nabla_{\bar{y}}W)(\calR(x_3 /R_0)\bar{x})\cdot (\nabla_{\bar{y}}\Psi)(\calR(x_3 /R_0)\bar{x})}\mathrm{d}x.
\end{align*}
By direct computations, we obtain
\[
\dfrac{1}{\omega_c}\dfrac{\partial\bar{y}}{\partial x}M[e]\dfrac{^t\partial\bar{y}}{\partial x} = \dfrac{1}{\omega_0}\calR\left(\dfrac{\pi}{2}\right),\,\,\omega_0 =\dfrac{qB_0}{m},
\]
and therefore the previous calculations lead to
\begin{align*}
\int_{\R^2\times\T^1}&{\Divx\left( \dfrac{u e}{\omega_c}\wedge\nabla_x w \right) \psi(x)}\mathrm{d}x= 2\pi R_0 \int_{\R^2}{\dfrac{U(\bar{y})}{\omega_0}{^\perp}\nabla_{\bar{y}}W\cdot \nabla_{\bar{y}}\Psi}\mathrm{d}\bar{y}\\
&= -2\pi R_0 \int_{\R^2}{\mathrm{div}_{\bar{y}}\left( \dfrac{U(\bar{y})}{\omega_0}{^\perp}\nabla_{\bar{y}}W\right)\Psi(\bar{y})}\mathrm{d}\bar{y}\\
&= - \int_{\R^2\times\T^1}{\left[ \mathrm{div}_{\bar{y}}\left( \dfrac{U(\bar{y})}{\omega_0}{^\perp}\nabla_{\bar{y}}W\right) \right]\left( \bar{y}= \calR\left( \dfrac{x_3}{R_0}\right)\right)\psi(x)}\mathrm{d}x. 
\end{align*}
We deduce that
\[
\Divx\left( \dfrac{u e}{\omega_c}\wedge\nabla_x w \right) =  \mathrm{div}_{\bar{y}}\left( \dfrac{U}{\omega_0}{^\perp}\nabla_{\bar{y}}W \right).
\]
\end{proof}

Combining Lemma \ref{Laplace} and Lemma \ref{ConservLamw}, we derive the limit model with respect to the new unknown $N$. The potential $\Phi = \phi[n]$ writes $\phi(t,x) = \Phi(t,\bar{y}= \calR(x_3 /R_0)\bar{x})$ where $\Phi(t,\bar{y})$ solves the elliptic equation
\[
- \epsilon_0 \mathrm{div}_{\bar{y}}\left[ \left(  I_2 + \dfrac{{^\perp} \bar{y} \otimes {^\perp} \bar{y}}{R_0 ^2}\right)\nabla_{\bar{y}}\Phi(t,\bar{y})\right] = q N(t,\bar{y}),\,\, \bar{y}\in\R^2.
\] We supplement this elliptic equation by the condition $\lim_{|\bar{y}|\to +\infty}\Phi(t,\bar{y})=0$ and we denote by $\Phi[N]$ the solution corresponding to the concentration $N$. We introduce $K[N]=\sigma(1+\ln N) + \frac{q}{m}\Phi[N]$. The time evolution for the concentration $N$ is given by
\[
\partial_t N + \mathrm{div}_{\bar{y}}\left(\dfrac{N}{\omega_0}\calR\left(\dfrac{\pi}{2} \right)\nabla_{\bar{y}}K[N]  \right) =0,\,\,(t,\bar{y})\in\R_+ \times\R^2 ,
\]
and the initial condition
\[
N(0,\bar{y}) = N_{\mathrm{in}}(\bar{y}),\,\, \bar{y}\in \R^2 ,
\]
where $n_{\mathrm{in}}(x) = N_{\mathrm{in}}(\calR(x_3 /R_0)\bar{x}),x\in\R^2\times\T^1$.

\section{Example of regular solutions for limit model}
\label{Example}
In this section, we construct a regular solution for the reduced limit model obtained in Section \ref{Cylin}. To simplify the presentation, we assume that the physical parameters $\sigma, q, m$ and $\epsilon_0$ are normalized. Thus, we focus on establishing the well-posedness of the following system: 
\begin{equation}
\label{equ:EquCylinModLim}
\partial_t n + \Divx\left[ \dfrac{ne}{B}\wedge \nabla_x \left( (1+\ln n) + \Phi[n] \right)\right] =0,\quad (t,x)\in \R_+ \times \R^2 \times \T^1 ,
\end{equation}
with the constraint for the concentration $n$ 
\begin{equation}
\label{equ:constraint_n}
Be \cdot \nabla_x n =0,
\end{equation}
where $\Phi[n]$ stands for the Poisson electric potential which solves
\begin{equation}
\label{equ:PoiCylinLim}
 - \Delta_x \Phi[n(t)](x) =  n(t,x),\,\,(t,x)\in \R_+ \times \R^2 \times \T^1,
\end{equation}
and the external magnetic field is given by  $Be = (x_2, -x_1,1)^t$. 
Denoting $E[n(t)] = -\nabla_x \Phi[n(t)]$ the electric field derives from the potential $\Phi[n(t)]$. 
We supplement our model by the initial condition
\begin{equation*}
\label{equ:IniCylinLim}
n(0,x) = n_{\mathrm{in}}(x),\,\, x\in \R^2 \times \T^1,
\end{equation*}
where $n_{\mathrm{in}}$ is a smooth function and  belongs to $ \mathrm{ker}(Be\cdot\nabla_x)$. 

We follow the same arguments as in the well-posedness proof for the Vlasov-Poisson problem with an external magnetic field, as discussed in \cite{BosSIAM2019, Bos2020}. Our goal is to obtain a priori bounds for the $L^\infty$ norm of $E[n]$ and $ \partial_x E[n]$, not in the full space $\R^3$, but in $\R^2\times\T^1$. These bounds rely on estimating the fundamental solution of Laplace's equation on $\R^2\times \T^1$. Therefore, we begin by investigating the Poisson equation for a given density in this domain and deriving a fundamental solution for this purpose. 
\subsection{Fundamental solution of Laplace's equation on $\R^2\times\T^1$}
Consider a function $\Xi: \R^2\times\T^1 \to \R$ satisfying
\begin{equation}
\label{equ:EquPoiPerZ}
-\Delta _x \Xi = \delta_0(\bar{x},x_3),\,\,x=(\bar{x}=(x_1,x_2),x_3)\in \R^2\times\T^1,
\end{equation}
in the sense of distributions, where $\delta_0(x)$ denotes the Dirac measure on $\R^2\times\T^1$ giving unit mass to the point $0$.
\begin{lemma}
\label{PerFundSol}
Let $x = (\bar{x},x_3) \in \R^2\times \T^1$. Then
$
\Xi (x) = -\frac{1}{4\pi ^2} \ln |\bar{x}| + \Gamma (x),
$
satifies \eqref{equ:EquPoiPerZ}, where
\[
\Gamma(x) = \int_{0}^{\infty}\dfrac{1}{4\pi t}e^{{-|\bar{x}|^2}/{4t}} \dfrac{1}{\pi}\left[  \sum_{n=1}^{\infty} e^{-n^2 t}\cos(n x_3)\right]\mathrm{d} t .
\]
\end{lemma}
\begin{proof}
We have
\begin{align}
\label{equ:SumPerZ}
-\Delta_x \Xi = \dfrac{1}{2\pi}\sum_{n\in\Z} \delta_0(\bar{x})e^{i n x_3},
\end{align}
where we have used the Poisson summation formula
$
\delta_0(x_3) = \frac{1}{2\pi}\sum_{n\in\Z} e^{i n x_3}.
$
Indeed,  $\delta_0(x_3)$ is periodic with period $2\pi$, it can be represented as a Fourier series
\[
\delta_0(x_3) = \sum_{n\in\Z}c_n e^{in x_3},
\]
where the Fourier coefficients are
\begin{align*}
c_n  &= \dfrac{1}{2\pi} \int_{-\pi}^{\pi}\delta_0(x_3 ) e^{-in x_3}\mathrm{d} x_3  = \dfrac{1}{2\pi}.
\end{align*}
On the other hand, as $\Xi$ is periodic in $x_3$ of period $2\pi$, we also have
\begin{equation*}
\Xi(x) = \sum_{n\in \Z} \beta_n(\bar{x}) e^{inx_3},
\end{equation*}
therefore
\begin{equation}
\label{equ:EquLap}
-\Delta_x \Xi(\bar{x},x_3) = \sum_{n\in\Z}(-\Delta_{\bar{x}}\beta_n(\bar{x}) + n^2 \beta_n(\bar{x}))e^{inx_3}.
\end{equation}
Comparing \eqref{equ:SumPerZ} and \eqref{equ:EquLap} yields the following linear elliptic equation in the whole space $\R^2$ for any $n\in\Z \backslash\left\{0\right\}$
\begin{equation}
\label{equ:EquCoeffs}
-\Delta_{\bar{x}}\beta_n(\bar{x}) + n^2 \beta_n(\bar{x}) = \dfrac{1}{2\pi}\delta_0(\bar{x}),\,\, \bar{x}\in\R^2.
\end{equation}
A solution to \eqref{equ:EquCoeffs} can be found using the Fourier transform for linear equation. It is known that the solution to this equation is given in terms of the Bessel potential $U(\bar{x})$ as $\beta_n(\bar{x}) = \frac{1}{2\pi}(U \star \delta_0)(\bar{x})$, cf. \cite{Evans}, where 
$
U(\bar{x}) = \int_{0}^{\infty}\frac{1}{4\pi t}e^{-|\bar{x}|^2/{4t}}e^{-n^2 t}\mathrm{d} t.
$
Thus, we have the solution formula
\[
\beta_n(\bar{x}) = \dfrac{1}{2\pi}\int_{0}^{\infty}\dfrac{1}{4\pi t}e^{-|\bar{x}|^2/{4t}}e^{-n^2 t}\mathrm{d} t.
\]
In the case $n=0$, equation \eqref{equ:EquCoeffs} becomes the Laplace equation on $\R^2$. It is well known that the fundamental solution is given by $-\frac{1}{4\pi ^2}\ln |\bar{x}|$.
Finally, by summing over all cases, we obtain the result stated in the lemma.
\end{proof}

Let us denote $\Gamma_{1,2}(t,\bar{x}) := \frac{1}{4\pi t}e^{-|\bar{x}|^2/{4t}}$ and $\Gamma_3 (t,x_3) := \frac{1}{2\pi}\left[1 + 2 \sum_{n=1}^{\infty}e^{-n^2 t}\cos(n x_3) \right]$. It is known that $\Gamma_{1,2}$ is a heat kernel on $\R^2$ of the heat equation 
\begin{align*}
\left\{
    \begin{array}{ll}
      \partial_t \Gamma_{1,2}(t,\bar{x}) - \Delta_{\bar{x}}\Gamma_{1,2}(t,\bar{x}) = 0,\,\, (t,\bar{x})\in \R_+\times \R^2 ,\\
\hspace*{30mm}\Gamma_{1,2}|_{t=0}(\bar{x}) = \delta_0(\bar{x}),
    \end{array}
  \right. 
\end{align*}
while $\Gamma_3$ is a heat kernel on $\T^1$ of
\begin{align*}
\left\{
    \begin{array}{ll}
\partial_t \Gamma_{3}(t,x_3) - \partial_{x_3}^2\Gamma_{3}(t,x_3) =0,\,\, (t,x_3)\in \R_+\times \T^1 , \\
\hspace*{32mm}\Gamma_{3}|_{t=0}(x_3) = \delta_0(x_3).
 \end{array}
  \right. 
\end{align*}
For a proof of this property, we refer the reader to \cite{CarPanZa20}.
Thus, we find that the function $\Gamma$ in the fundamental solution of Laplace's equation \eqref{equ:EquPoiPerZ} is related to the previous solution of the heat equation as follows: 
\begin{equation}
\label{equ:EquLapHeat}
\Gamma(x) = \int_{0}^{\infty}\Gamma_{1,2}(t,\bar{x})\left[\Gamma_3(t,x_3) -\frac{1}{2\pi} \right]\mathrm{d} t.
\end{equation}
\begin{remark}
The heat kernel $\Gamma_3$ on $\T^1$ can also be expressed using the heat kernel $k_t(x_3) = {(4\pi t)}^{-1/2}e^{-x_3 ^2/{4t}}$ on the real line $\R$ as follows:
\begin{equation}
\label{equ:HeatKerR}
\Gamma_3 (x_3) = \dfrac{1}{2\pi} g_t(x_3): = \dfrac{1}{2\pi}\left[2\pi \sum_{n\in \Z} k_t(x_3 + 2\pi n)\right],\,\, x_3 \in \T^1.
\end{equation}
Indeed, the function $g_t \in L^1(\T^1)$ since 
\[
\| g_t \|_{\T^1} = \int_{\T^1} g_t \mathrm{d} m(x_3) = \sum_{n\in\Z} \int_{\T^1} k_t(x_3 + 2\pi n) \mathrm{d} x_3 = \int_{\R} k_t(x_3) \mathrm{d} x_3 =1 ,
\]
where $\mathrm{d} m(x_3)$ is Haar measure on $\T^1$, $ \mathrm{d} m(x_3) = 1/{(2\pi)}\,\mathrm{d} x_3$. Thus, the periodic function $g_t$ can be written in the form of the Fourier serie
\[
g_t(x_3) =  \sum_{n\in \Z} \hat{g}_t(n) e^{i n x_3},
\]
where $(\hat{g}_t (n))_{n\in\Z}$ is the sequence of the Fourier coeffiecients which is given by
\begin{align*}
\hat{g}_t(n) &= \dfrac{1}{2\pi}\int_{\T^1} g_t(x_3) e^{-i n x_3} \mathrm{d} m(x_3) 
= \dfrac{1}{4\pi^2}\sum_{n\in\Z}\int_{\T^1} k_t(x_3 +2\pi n) e^{-i n (x_3 + 2\pi n)} \mathrm{d} x_3 \\
&=\dfrac{1}{4\pi^2}\int_{\R}k_t(x_3)e^{-i n x_3}\mathrm{d} x_3 = \dfrac{1}{4\pi^2}\hat{k}_t(n) =\dfrac{1}{2\pi}\left[\dfrac{1}{2\pi} e^{-n^2 t}\right],
\end{align*}
where $\hat{k}_t(n)$ is the Fourier transform of  the function $k_t(x_3)$.
\end{remark}

Since we need the bounds of the function $\Gamma$ and its derivatives, we must estimate the function $\Gamma_3 -\frac{1}{2\pi}$, as well as the first and second derivates of $\Gamma_3$ from \eqref{equ:EquLapHeat}. We will use the arguments in \cite{Maheux} to obtain the bound of $|\Gamma_3 -\frac{1}{2\pi}|$, using the following lemmas. The proofs of the lemmas are standard and are left to the reader.\\
Firstly, using the formula \eqref{equ:HeatKerR}, we can rewrite the function $\Gamma_3$ on $\T^1$ as follows:
\begin{lemma}
\label{EquiFormHeatKer}
For any $t>0$ and for any $x_3\in\T^1$, we have
\begin{equation*}
g_t(x_3) = \sqrt{\dfrac{\pi}{t}}\exp\left(\dfrac{- x_3^2}{4t} \right)\left( 1+ 2 \sum_{n\geq 1}\exp\left(\dfrac{-\pi^2 n^2}{t} \right)\cosh\left(\dfrac{\pi n x_3}{t}\right) \right).
\end{equation*}
\end{lemma}
Next, using Lemma \ref{EquiFormHeatKer}, we obtain the following estimate
\begin{lemma}
For any $t>0$ and any $x_3\in \T^1 =[-\pi,\pi]$, we have
\begin{equation}
\label{FirstEstHeatKer}
\exp\left(\dfrac{- x_3^2}{4t} \right) g_t(0) \leq g_t(x_3) \leq \left[\sqrt{\dfrac{\pi}{t}} + g_t(0) \right] \exp\left(\dfrac{ - x_3^2}{4t} \right).
\end{equation}
\end{lemma}
We need the estimate of the function $g_t(x_3)$ at $x_3 =0$.
\begin{lemma}
\label{Estginit}
For any $t>0$, we have
\[
\sqrt{\dfrac{\pi}{t}} \leq g_t(0) \leq 1+ \sqrt{\dfrac{\pi}{t}},
\]
and 
\[
2 e^{-t} \leq g_t(0) -1 \leq \dfrac{2 e^{-t}}{1- e^{-t}}.
\]
Consequently, there exist positive constants $C_1, C_2$ such that $  g_t(0) -1\leq C_1 \frac{e^{-C_2 t}}{\sqrt{t}}$.
\end{lemma}

Now, the following lemma provides estimates for $\Gamma_3 - \frac{1}{2\pi}$ and its derivatives on $\T^1$. 
\begin{lemma}
\label{BoundHeatT1}
Let $\Gamma_3(t,x_3) = \frac{1}{2\pi}\left[1 + 2 \sum_{n=1}^{\infty}e^{-n^2 t}\cos(n x_3) \right]$ represent the heat kernel on $\T^1$. Then there exist constants $C_1, C_2$, and $C_3$,  which may vary from line to line such that:
\begin{equation}
\label{FirstEstHeatKerbis}
\left| \Gamma_3 (t,x_3) - \dfrac{1}{2\pi} \right| \leq C_1\dfrac{1}{\sqrt{t}}e^{-C_2 t} e^{-C_3 x_3 ^2 /{4t}},\,\, t>0, x_3 \in\T^1,
\end{equation}
\begin{equation}
\label{SecEstHeatKer}
|\partial_{x_3}\Gamma_3 (t,x_3)| \leq C_1\dfrac{1}{t}e^{-C_2 t} e^{-C_3 x_3 ^2 /{4t}},\,\, t>0, x_3 \in\T^1,
\end{equation}
\begin{equation}
\label{ThirdEstHeatKer}
|\partial_{x_3}^2 \Gamma_3 (t,x_3)| \leq C_1\dfrac{1}{t^{3/2}}e^{-C_2 t} e^{-C_3 x_3 ^2 /{4t}},\,\, t>0, x_3\in\T^1.
\end{equation} 
\end{lemma}
\begin{proof}
Readers can find these results in \cite{CarPanZa20}, even when $\T^1$ is replaced by a more general compact manifold, cf. \cite{TianZa, Zhang}. Here, we outline the main lines of the proof.\\
The bound in \eqref{FirstEstHeatKerbis} is easily obtained as a consequence of Lemma \ref{Estginit} for $t \geq 1$. For $t \leq 1$, using \eqref{equ:HeatKerR} first yields
$
\Gamma_3(t,x_3) -\frac{1}{2\pi} = \frac{1}{2\pi}(g_t(x_3) -1).
$
Then, applying \eqref{FirstEstHeatKer}, we have
\begin{equation}
\label{LowUppHeatKer}
\dfrac{1}{2\pi}\left[\exp\left(\dfrac{- x_3^2}{4t} \right) g_t(0)-1\right] \leq \Gamma_3(t,x_3) -\dfrac{1}{2\pi} \leq \dfrac{1}{2\pi} \left[ \exp\left(\dfrac{ - x_3^2}{4t} \right)\left(\sqrt{\dfrac{\pi}{t}} + g_t(0) -1\right)\right].
\end{equation}
Using the upper bound in \eqref{LowUppHeatKer} and Lemma \ref{Estginit}, we deduce that
\begin{align*}
\Gamma_3(t,x_3) -\dfrac{1}{2\pi} \leq \dfrac{1}{2\pi}\exp\left(\dfrac{ - x_3^2}{4t} \right)\sqrt{\dfrac{\pi}{t}} + \dfrac{1}{2\pi}\exp\left(\dfrac{ - x_3^2}{4t} \right) \dfrac{2 e^{-t}}{1- e^{-t}}.
\end{align*}
If $t \in [\delta_0, 1]$, for some $\delta_0\in (0,1)$, it is  straightforward to show from the previous inequality that there exist positive constants $C_1, C_2$, and $ C_3$ such that $\Gamma_3(t,x_3) -\frac{1}{2\pi} \leq C_1\frac{1}{\sqrt{t}}e^{-C_2 t} e^{-C_3 x_3 ^2 /{4t}}$. On the other hand, for any positive test function $\varphi\in C^\infty_c(\R)$, since $\lim_{t\to 0^+}\left<\Gamma_3 -\frac{1}{2\pi}, \varphi \right> = (1-1/2\pi)\varphi(x_3)$ and $\lim_{t\to 0^+} \left< k_t,\varphi \right> = \varphi(x_3)$, where $k_t(x_3) = (4\pi t)^{-1/2}e^{-x_3 ^2 /{4t}}$ is the heat kernel on $\R$, we deduce that we can choose the positive constants as above to obtain the previous estimate of $\Gamma_3 -1/2$ as $t\to 0^+$. These arguments together give us the upper bound of \eqref{FirstEstHeatKerbis}. 
Similarly, by using the lower bound in \eqref{LowUppHeatKer}, we deduce the corresponding lower bound in \eqref{FirstEstHeatKerbis}. Consequently, we obtain the complete estimate \eqref{FirstEstHeatKerbis}. For the estimates \eqref{SecEstHeatKer} and \eqref{ThirdEstHeatKer}, we apply  Lemma $2.1$ in \cite{TianZa}, which can be extended to the parabolic case (see Lemma $2.3$ in \cite{TianZa}):
\[
|\nabla_x u(t,x_3)| \leq \dfrac{C}{r}\left( \dfrac{1}{r^4}\int_{t-r}^{r}\int_{|y- x_3|<r}|u(s,y)|^2 \mathrm{d}y\mathrm{d}s \right)^{1/2},
\]
where $u$ is asolution of the heat equation $\partial_t u -\partial^2_{x_3}u =0$ in the domain $ [t-r^2,t]\times B(x_3,r)$, with $ r =\sqrt{t}/2 $ for any fixed point $(t,x_3)\in \R\times\T^1$.
\end{proof}

Finally, we provide estimates for the function $\Gamma$ and its derivatives using the relation \eqref{equ:EquLapHeat}, along with the inequalities \eqref{FirstEstHeatKerbis}, \eqref{SecEstHeatKer}, and \eqref{ThirdEstHeatKer}.
\begin{lemma}
\label{EstFundSol}
Let $\Gamma(x)$ be the function on $\R^2\times \T^1$ provided by Lemma \ref{PerFundSol}. Then, we have the following estimates
\[
|\Gamma(x)| \leq  \dfrac{C}{|x|},\quad
 |\nabla_x\Gamma(x)| \leq \dfrac{C}{|x|^2},\quad |D^2_x \Gamma(x)|\leq \dfrac{C}{|x|^3},
\]
where $\nabla_x$ and $D^2_x$ denote the first and second order derivative, respectively. Here, $C$ is a positive constant, which can vary in each estimate.
\end{lemma}
\begin{proof}
We will first estimate $\Gamma(x)$. Using \eqref{equ:EquLapHeat} and \eqref{FirstEstHeatKerbis}, we deduce that
\begin{align*}
|\Gamma(x)| &\leq  \dfrac{C_1}{4\pi} \int_{0}^{\infty}t^{-3/2} e^{-C_2 t} e^{-C_3 '|x|^2/{t}}\mathrm{d}t,\quad C_3 ' = \min{(1, C_3)}/4\\
&=  \dfrac{C_1}{4\pi} e^{ -2 \sqrt{C_2 C' _3 }|x|} \int_{0}^{\infty} e^{-\left(\frac{\sqrt{C'_3}|x| - \sqrt{C_2 }t}{\sqrt{t}} \right)^2} 2\mathrm{d} (- t^{-1/2})\\
&=  \dfrac{C_1}{2\pi} e^{ -2 \sqrt{C_2 C' _3 }|x|} \int_{0}^{\infty} e^{-\left(\sqrt{C' _3}|x| u - \sqrt{C_2}u^{-1} \right)^2} \mathrm{d} u,\quad u = t^{-1/2}\\
&=  \dfrac{C_1}{2\pi}\dfrac{1}{\sqrt{C' _3}|x|} e^{ -2 \sqrt{C_2 C' _3 }|x|} \int_{0}^{\infty} e^{- \left(\theta - \sqrt{C_2 C' _3 }|x|\theta^{-1} \right)^2} \mathrm{d} \theta,\quad \theta = \sqrt{C' _3}|x| u\\
& \leq \dfrac{C}{|x|},
\end{align*}
for some positive constant $C$, where we have used that 
$
\int_{0}^{\infty} e^{- \left(\theta - \sqrt{C_2 C' _3 }|x|\theta^{-1} \right)^2} \mathrm{d} \theta = \frac{\sqrt{\pi}}{2}.
$

Next, we estimate  $\nabla_x\Gamma(x)$. By taking the derivative with respect to $x$ in the formula \eqref{equ:EquLapHeat}, we deduce that
\begin{align*}
|\nabla_x \Gamma (x)| \leq \int_{0}^{\infty}|\nabla_{\bar{x}} \Gamma_{1,2}(t,\bar{x})||\Gamma_3(t,x_3)-\frac{1}{2\pi}|\mathrm{d} t + \int_{0}^{\infty}|\Gamma_{1,2}(t,\bar{x})||\partial_{x_3}\Gamma_3(t,x_3)|\mathrm{d} t.
\end{align*}
A simple computation shows that
$
\nabla_{\bar{x}} \Gamma_{1,2}(t,\bar{x}) = \frac{-|\bar{x}|}{8\pi t^2}e^{-|\bar{x}|^2/{4t}},
$
and, thanks to the estimates \eqref{FirstEstHeatKerbis} and \eqref{SecEstHeatKer}, we obtain 
\begin{align*}
|\nabla_x \Gamma (x)|  \leq \dfrac{C_1|\bar{x}|}{8\pi}\int_{0}^{\infty} t^{-5/2}e^{-C_2 t}e^{-C'_3|x|^2/t}\mathrm{d} t + \dfrac{C_1}{4\pi} \int_{0}^{\infty}t^{-2}e^{-C_2 t}e^{-C'_3|x|^2/t}\mathrm{d} t ,
\end{align*}
where $ C'_3 = \min(1,C_3)/4 $. Using $\sup_{\R_+ ^\star}q(t) = q(C'_3|x|^2)$, where $q(t) =t^{-1/2}e^{-C'_3 |x|^2 /{2t}}$, for the first integral on the last line of the previous inequality, we deduce that
\begin{align*}
|\nabla_x \Gamma (x)|  \leq \left(\dfrac{C_1}{8\pi}\dfrac{1}{\sqrt{C'_3 e}} + \dfrac{C_1}{4\pi}\right)\int_{0}^{\infty} t^{-2}e^{-C'_3|x^2|/2t} \mathrm{d}t \leq \dfrac{C}{|x|^2},
\end{align*}
for some positive constant $C$.

Finally, we estimate $D^2_x \Gamma(x)$. By direct computation in \eqref{FirstEstHeatKerbis}, we have
\begin{align*}
|D^2_x \Gamma(x)| &\leq \int_{0}^{\infty} |D^2_{\bar{x}} \Gamma_{1,2}(t,\bar{x})||\Gamma_3(t,x_3)-\frac{1}{2\pi}|\mathrm{d} t + \int_{0}^{\infty} |\nabla_{\bar{x}} \Gamma_{1,2}(t,\bar{x})||\partial_{x_3}\Gamma_3(t,x_3)|\mathrm{d} t \\
&\quad + \int_{0}^{\infty} |\Gamma_{1,2}(t,\bar{x})||\partial^2 _{x_3}\Gamma_3(t,x_3)|\mathrm{d} t.
\end{align*}
Since
$
D^2_{\bar{x}} \Gamma_{1,2}(t,\bar{x}) = \frac{1}{8\pi t^2}\left[-I_2 + \frac{\bar{x}\otimes\bar{x}}{2t} \right]e^{-|\bar{x}|^2/{4t}},
$
 and using the inequalities \eqref{FirstEstHeatKerbis}, \eqref{SecEstHeatKer}, and \eqref{ThirdEstHeatKer}, we deduce that
\begin{align*}
|D^2_x \Gamma(x)| &\leq \dfrac{C_1}{8\pi}\int_{0}^{\infty}t^{-5/2}e^{-C_2 t}e^{-|\bar{x}|^2/{4t}}e^{- C_3 x _3 ^2 /{4t}}\mathrm{d}t \\&\quad +  \dfrac{C_1|\bar{x}|^2}{16\pi}\int_{0}^{\infty}t^{-7/2}e^{-C_2 t}e^{-|\bar{x}|^2/{4t}}e^{- C_3 x _3 ^2 /{4t}}\mathrm{d} t \\
&\quad + \dfrac{C_1|\bar{x}|}{8\pi}\int_{0}^{\infty}t^{-3}e^{-C_2 t}e^{-|\bar{x}|^2/{4t}}e^{- C_3 x _3 ^2 /{4t}}\mathrm{d} t\\
&\quad + \dfrac{C_1}{4\pi}\int_{0}^{\infty}t^{-5/2}e^{-C_2 t}e^{-|\bar{x}|^2/{4t}}e^{- C_3 x _3 ^2 /{4t}}\mathrm{d} t\\
& =: I_1 + I_2 + I_3 + I_4.
\end{align*}
The estimates for $I_1$ and $I_4$ are performed as above. Thus, we obtain 
$
I_1 \leq \frac{C}{|x|^3}$ and $ I_4 \leq \frac{C}{|x|^3}$, for some positive constant $C$.
For the integral $I_3$, we have that
\[
I_3 \leq \dfrac{C_1|x|}{8\pi^2}\int_{0}^{\infty}t^{-3}e^{-C_2 t}e^{-C'_3|x|^2/{t}}\marmd t,\quad C'_3 =\min(1,C_3)/4.
\]
Using again $\sup_{\R_+ ^\star}h(t) = h(C'_3|x|^2)$, where $h(t) =t^{-1/2}e^{-C'_3 |x|^2 /{2t}}$, we obtain that
$
I_3 \leq \frac{C}{|x|^3}
$, for some positive constant $C$. 
Similarly for integral $I_2$, we also have
\begin{align*}
I_2 &\leq  \dfrac{C_1|x|^2}{32\pi^2}\int_{0}^{\infty}t^{-7/2}e^{-C_2 t}e^{-C'_3|x|^2/{t}}\mathrm{d} t,\quad C'_3=\min(1,C_3)/4\\
&\leq \dfrac{C_1|x|^2}{32\pi^2}\dfrac{1}{\sqrt{C'_3 e}|x|} \int_{0}^{\infty}t^{-3}e^{-C_2 t}e^{-C'_3|x|^2 /2t} \mathrm{d} t\\
&\leq \dfrac{C}{|x|^3}.
\end{align*}
Combining the estimates of $I_i$ for $ i=1,...,4$, we obtain the estimate of $D^2_{x}\Gamma(x)$.
\end{proof}

Using  Lemma \ref{PerFundSol}, the $L^\infty$ estimate for the function $\Gamma$ from Lemma \ref{EstFundSol}, and following the same reasoning as in the proof for Poisson's equation in $\R^3$, we can show that the solution of the Poisson equation \eqref{equ:PoiCylinLim} is given by
\begin{equation}
\label{equ:PoissonPoten}
\Phi[n](x) = \Xi \star n(x) = \int_{-\pi}^{\pi}\int_{\R^2}{\Xi(x-y)n(y)}\mathrm{d}\bar y\mathrm{d} y_3 . 
\end{equation} 

\subsection{Estimations for the electric field and its gradient on $\R^2 \times \T^1$}
We now provide some estimates of the electric field $E[n] = -\nabla_x\Phi[n]$, which can be derived by handling the singular term in the fundamental solution $\Xi$, following the approach in \cite{Batt} for the domain $x\in\R^3$ and \cite{ReiRen1993} for $x\in\T^3$.
\begin{lemma}
\label{FirstEstiEle}
Let $n$ be a positive concentration and belongs to $ L^1 \cap L^\infty(\R^2\times \T^1)$. Then, there exists a constant $C>0$ such that the electric field $E[n]$ satisfies the following estimate:
\[
\| E[n] \|_{L^\infty} \leq C (\| n\|_{L^\infty} + \| n\|_{L^1}).
\]
\end{lemma}
\begin{proof}
For any $x = (\bar{x},x_3)\in \R^2 \times [-\pi,\pi]$, from the formula \eqref{equ:PoissonPoten}, we have
\begin{align*}
\nabla_x \Phi[n](x) = -\dfrac{1}{4\pi^2}\int_{-\pi}^{\pi}\int_{\R^2}{\dfrac{\bar{x} -\bar{y}}{|\bar{x}-\bar{y}|^2} n(y)}\mathrm{d}\bar{y}\mathrm{d}y_3 +  \int_{x_3 -\pi}^{x_3 + \pi}\int_{\R^2}{\nabla_x\Gamma(x-y)\,n(y)}\mathrm{d}\bar{y}\mathrm{d} y_3.
\end{align*}
The first integral in the previous expression can be bounded by
\begin{align*}
&\dfrac{1}{4\pi^2 } \left( \int_{-\pi}^{\pi}\int_{\R^2}{\nabla_{\bar{x}}\ln |\bar{x} - \bar{y}| \mathds{1}_{\left\{ |\bar{x} - \bar{y}|\leq 1 \right\}} n(y)}\mathrm{d}\bar{y}\mathrm{d}y_3 + \int_{-\pi}^{\pi}\int_{\R^2}{\nabla_{\bar{x}}\ln |\bar{x} - \bar{y}| \mathds{1}_{ \left\{ |\bar{x} - \bar{y}|\geq 1 \right\}} n(y)}\mathrm{d}\bar{y}\mathrm{d}y_3 \right)\\
& \leq C (\|n \|_{L^\infty} + \|n \|_{L^1}).
\end{align*}
For the second integral, we make a decomposition of $\R^2\times\T^1$ in the following way
\[
\R^2\times [x_3 -\pi, x_3 + \pi]: = I \cup J ,
\]
where 
$
I= \left\{ y\in \R^3: |x-y| \geq 1 \right\} \cap \R^2\times [x_3 -\pi, x_3 + \pi]   ,
$ and
$
J= \left\{ y\in\R^3: |x-y| \leq 1 \right\}.
$
It is obviously that $J \subseteq \R^2\times [x_3 -\pi, x_3 + \pi] $. Thus, the last integral in the previous equality can be written as:
\[
\int_{x_3 -\pi}^{x_3 + \pi}\int_{\R^2}{\nabla_x\Gamma(x-y)\,n(y)}\mathrm{d}\bar{y}\mathrm{d} y_3 = \int_{I} \nabla_x\Gamma(x-y)\,n(y) \mathrm{d} y + \int_{J}\nabla_x\Gamma(x-y)\,n(y) \mathrm{d} y.
\]
Using Lemma \ref{EstFundSol}, we deduce that
\begin{equation*}
\begin{split}
\int_{x_3 -\pi}^{x_3 + \pi}\int_{\R^2}{\nabla_x\Gamma(x-y)\,n(y)}\mathrm{d}\bar y\mathrm{d} y_3 
&\leq C\left[ \int_{x_3 -\pi}^{x_3 + \pi}\int_{\R^2}{n(y)}\mathrm{d}\bar y\mathrm{d} y_3 + \int_{|x-y|\leq 1}  \dfrac{1}{|x-y|^2} n(y) \mathrm{d} y \right]\\
&\leq C \left[ \int_{-\pi}^{\pi}\int_{\R^2}{n(y)}\mathrm{d}\bar y\mathrm{d} y_3  + 4\pi \| n\|_{L^\infty}\right]\\
&\leq C (\|n\|_{L^1} + \| n\|_{L^\infty}),
\end{split}
\end{equation*}
where we have used that  $\int_{|x-y|\leq 1} \frac{1}{|x-y|^2} \mathrm{d}y =4\pi$. Combining these estimates, we obtain the desired result in Lemma.
\end{proof}
\begin{lemma}
\label{SecondEstEle}
Let $n \in L^1\cap W^{1,\infty}(\R^2\times\T^1)$ and $n\geq 0$. There exists a constant $C>0$ such that the gradient of the electric field $E[n]$ satisfies the following estimates
\[
\| \nabla_x E[n] \|_{L^\infty} \leq C \left( 1 + \| n\|_{L^\infty}(1+\ln ^{+}( \|\nabla_x n \|_{L^\infty})) + \|n\|_{L^1} \right),
\]
where the notation $\ln^+$ stands for the positive part of $\ln$.
\end{lemma}
\begin{proof}
Observe that the potential $\Phi[n]$ can be expressed as follows:
\begin{align*}
\Phi[n](x) = -\dfrac{1}{4\pi^2}\int_{-\pi}^{\pi} \int_{\R^2}{\ln |\bar{y}| n(\bar{x}-\bar{y},x_3 -y_3)}\mathrm{d}\bar y\mathrm{d} y_3  + \int_{-\pi}^{\pi}\int_{\R^2}{\Gamma(y)\,n(x-y)}\mathrm{d}\bar y\mathrm{d} y_3 ,
\end{align*}
since the functions $\Gamma$ and $n$ are periodic with respect to $x_3$ with a period of $2\pi$. We will now estimate $\partial_{x_1}^2 \Phi[n](x)$. The same approach can be applied in orther cases. Taking the derivative with respect to $x_1$ in the above equality, we obtain
\begin{align*}
\partial_{x_1} \Phi[n](x) = \dfrac{1}{4\pi^2 }\int_{-\pi}^{\pi} \int_{\R^2}{\ln |\bar{x}-\bar{y}| \partial_{y_1} n(\bar{y},y_3)}\mathrm{d}\bar y\mathrm{d} y_3 -\int_{-\pi}^{\pi} \int_{\R^2}{\Gamma(x-y)\partial_{y_1}n(y)}\mathrm{d}\bar y\mathrm{d} y_3,
\end{align*}
which implies that
\begin{align*}
\partial_{x_1}^2 \Phi[n](x) 
&= \dfrac{1}{4\pi^2 }\int_{-\pi}^{\pi} \int_{\R^2}{ \dfrac{x_1 -y_1}{|\bar{x}-\bar{y}|^2} \partial_{y_1} n(\bar{y},y_3)}\mathrm{d}\bar y\mathrm{d} y_3 - \int_{x_3 -\pi}^{x_3 +\pi} \int_{\R^2}{\partial_{x_1}\Gamma(x-y)\partial_{y_1}n(y)}\mathrm{d}\bar y\mathrm{d} y_3\\
&=: K_1 + K_2. 
\end{align*}
The estimation of $K_1$ can be found in \cite{Batt}. We will now estimate $K_2$.
Let $r, R >0$ such that $0< r < R < \infty$ verify the condition
\[
\left\{ y\in\R^3 : |x-y| < R\right\} \subset \R^2 \times [x_3 - \pi, x_3 + \pi].
\]
We then decompose  $\R^2 \times [x_3 - \pi, x_3 + \pi]$ as follows
 \[\R^2 \times [x_3 - \pi, x_3 + \pi] := J_1 \cup J_2 \cup J_3 ,\] where
\[
J_1 =\left\{ y\in\R^3 : |x-y| > R\right\} \cap \R^2 \times [x_3 - \pi, x_3 + \pi] ,
\]
\[
J_2 = \left\{ y\in\R^3 : r < |x-y| < R\right\},\,\,\,\,
J_3 = \left\{ y\in\R^3 : |x-y| < r\right\}.
\]
For the integral over $J_1$, we can use the integration by parts with respect to $y_1$. Noticing that the boundary of $J_1$ is $\partial J_1 = \left\{ y\in\R^3: |x-y| = R\right\} \cup \R^2\times \left\{x_3 -\pi, x_3 +\pi \right\} $, we obtain:
\begin{align}
\label{equ:EquJ1}
&- \int_{J_1}\partial_{x_1}\Gamma(x-y)\partial_{y_1}n(y)\marmd\bar{y}\mathrm{d} y_3 \nonumber \\ &=  \int_{J_1}\partial_{y_1}\partial_{x_1}\Gamma(x-y) n(y) \mathrm{d}\bar{y}\mathrm{d} y_3 - \int_{|x-y|=R} \partial_{x_1}\Gamma(x-y) n(y) \dfrac{-(x_1 - y_1)}{|x-y|}\mathrm{d} \sigma(y)\\
&\quad - \int_{\R^2}{\underbrace{\left[ \partial_{x_1} \Gamma(x-(\bar{y},x_3 +\pi))n(\bar{y},x_3 + \pi) - \partial_{x_1} \Gamma(x-(\bar{y},x_3 -\pi))n(\bar{y},x_3 -\pi) \right]}_{=0}}\mathrm{d}\bar{y}\nonumber.
\end{align}
Similarly, the integral over $J_2$ can be expressed as:
\begin{align}
\label{equ:EquJ2}
&- \int_{J_2}\partial_{x_1}\Gamma(x-y)\partial_{y_1}n(y)\marmd\bar{y}\mathrm{d} y_3 \nonumber\\
&= \int_{J_2}\partial_{y_1}\partial_{x_1}\Gamma(x-y) n(y) \mathrm{d}\bar{y}\mathrm{d} y_3 - \int_{|x-y|= R} \partial_{x_1}\Gamma(x-y) n(y) \dfrac{(x_1 - y_1)}{|x-y|}\mathrm{d} \sigma(y)\\
&\quad - \int_{|x-y|= r} \partial_{x_1}\Gamma(x-y) n(y) \dfrac{-(x_1 - y_1)}{|x-y|}\mathrm{d} \sigma(y)\nonumber.
\end{align}
For the integral over $J_3$, since $\partial_{y_1}n(y) = \partial_{y_1}[n(y)-n(x)]$, we can apply the integration by parts to obtain:
\begin{align}
\label{equ:EquJ3}
&- \int_{J_3}\partial_{x_1}\Gamma(x-y)\partial_{y_1}n(y)\mathrm{d}\bar{y}\mathrm{d} y_3\\
&= \int_{J_3}\partial_{y_1}\partial_{x_1}\Gamma(x-y) [n(y)-n(x)] \mathrm{d}\bar{y}\mathrm{d} y_3 - \int_{|x-y|=r}\partial_{x_1}\Gamma(x-y) [n(y)-n(x)]\dfrac{x_1 - y_1}{|x-y|}\mathrm{d} \sigma(y) \nonumber.
\end{align}
Combining the equalities \eqref{equ:EquJ1}, \eqref{equ:EquJ2}, and \eqref{equ:EquJ3}, we deduce that
\begin{align*}
\partial_{x_1}^2 \Phi[n](x)&=\int_{J_1}\partial_{y_1}\partial_{x_1}\Gamma(x-y) n(y) \mathrm{d}\bar{y}\mathrm{d} y_3 + \int_{J_2}\partial_{y_1}\partial_{x_1}\Gamma(x-y) n(y) \mathrm{d}\bar{y}\mathrm{d} y_3\\
& + \int_{J_3}\partial_{y_1}\partial_{x_1}\Gamma(x-y) [n(y)-n(x)] \mathrm{d}\bar{y}\mathrm{d} y_3  + \int_{|x-y|=r}\partial_{x_1}\Gamma(x-y) n(x)\dfrac{(x_1 - y_1)}{|x-y|}\mathrm{d} \sigma(y)\\
&:= I_1 + I_2 + I_3 + I_4.
\end{align*}
Using Lemma \ref{EstFundSol}, we will estimate the integrals $I_i$ for $i=1,...,4$.
For the integral $I_4$, we apply the $L^\infty$ estimate of $\partial_x \Gamma$ to obtain:
\[
I_4 \leq C \int_{|x-y| =r} \dfrac{1}{|x-y|^2}\marmd \sigma(y) \| n \|_{L^\infty}= 4\pi C \|n\|_{L^\infty}.
\]
For the integral $I_3$, we also obtain an estimate by using the $L^\infty$ estimate of $\partial_x ^2 \Gamma$
\[
I_3 \leq C \int_{|x-y| < r} \dfrac{1}{|x-y|^3}|x-y| \mathrm{d} y \|\nabla_x n\|_{L^\infty} = 2\pi^2 C r\|\nabla_x n\|_{L^\infty}.
\]
Similarly for the integrals $I_2$ and $I_1$, we obtain
\[
I_2 \leq C \int_{r< |x-y| < R} \dfrac{1}{|x-y|^3}\mathrm{d} y \|n\|_{L^\infty} = 2\pi^2 C\ln(R/r)  \|n\|_{L^\infty},
\]
\[
I_1 \leq C \int_{|x-y|>R} \dfrac{1}{|x-y|^3}n(y)\mathrm{d} y \leq \dfrac{C}{R^3}\|n\|_{L^1}.
\]
Finally, combining these estimates of $I_i$ for $i=1,...,4$, we deduce that
\[
K_2 \leq C \left(  \dfrac{1}{R^3}\|n\|_{L^1}+  \ln(R/r)  \|n\|_{L^\infty}+ r\|\nabla_x n\|_{L^\infty}+  \|n\|_{L^\infty}  \right).
\]
Taking $r= \frac{1}{1 + \|\nabla_x n\|_{L^\infty}}$ and $R=1$ gives us the result of the lemma.
\end{proof}
\subsection{Local-in-time existence of smooth solutions}
We will begin by establishing strong solutions for the system \eqref{equ:EquCylinModLim}-\eqref{equ:PoiCylinLim}. It is sufficient to  construct a solution on some time interval $[0,T]$, $T>0$. We present only the main arguments, with the other details left to the reader. We assume that the initial condition $n_{\mathrm{in}}$ satisfies the hypotheses:
\begin{enumerate}
\item[H1)] $n_{\mathrm{in}} \geq 0$,\quad $Be\cdot\nabla_x n_\mathrm{in} =0$,
\item[H2)] $n_{\mathrm{in}}\in W^{1,\infty}(\R^2\times\T^1)\cap W^{1,1}(\R^2\times\T^1)$.
\end{enumerate}

\textbf{Solution integrated along the characteristics}.
By a standard computation, we can rewrite the equation \eqref{equ:EquCylinModLim} as follows:
\begin{equation}
\label{equ:EquCylinModLimBis}
\partial_t n + \left( E \wedge \dfrac{e}{B} \right)\cdot\nabla_x n +  \rot_x \left(\dfrac{e}{B} \right)\cdot \nabla_x n - \rot_x\left( \dfrac{e}{B} \right)\cdot E n =0.
\end{equation}
For any smooth field $E\in L^\infty(0,T;W^{1,\infty}(\R^2\times\T^1))$, we consider the associated characteristics flow of this equation
\begin{align}
\label{equ:EquChaLim}
\left\{
    \begin{array}{ll}
\dfrac{\marmd}{\marmd t}\Pi(t;s,x) = E(t,\Pi(t;s,x))\wedge \dfrac{e(\Pi(t;s,x))}{B(\Pi(t;s,x))} + \rot_x\left( \dfrac{e}{B} \right)(\Pi(t;s,x)),\\
\Pi(s;s,x) = x\in \R^2\times\T^1,
\end{array}
\right.
\end{align}
where $\Pi(t;s,x)$ is the solution of the ODE, $t$ represents the time variable, $s$ is the initial time and $x$ is the initial position. $\Pi(s;s,x) =x$ is our initial condition. Notice that the vector field $\frac{e}{B}$ is also smooth and belongs to $W^{2,\infty}(\R^2\times\T^1)$. Therefore, the characteristics in \eqref{equ:EquChaLim} are well defined for any $(s,x)\in [0,T]\times\R^2\times\T^1$ and there are smooth with respect to $x$. From \eqref{equ:EquChaLim}, the equation \eqref{equ:EquCylinModLimBis} can be written as
\[
\dfrac{\marmd }{\marmd t}n(t,\Pi(t;s,x)) - \rot_x\left(\dfrac{e}{B}\right)(\Pi(t;s,x))\cdot E(t,\Pi(t;s,x))n(t,\Pi(t;s,x)) = 0.
\]
The solution of the transport equation \eqref{equ:EquCylinModLimBis} is given by
\begin{equation}
\label{SolChaPi}
n(t,x) = n_{\mathrm{in}}(\Pi(0;t,x)) \exp\left( \int_{0}^{t}\rot_x\left(\dfrac{e}{B} \right)(\Pi(s;t,x))\cdot E(s,\Pi(s;t,x))\mathrm{d} s \right).
\end{equation}

\textbf{Conservation law on a volume}. 
We have the following conservation law
\begin{align}
\label{ConserVolu}
\int_{\R^2\times\T^1}{n(t,x)} = \int_{\R^2\times\T^1}{n_{\mathrm{in}}(x)},\,\, 0\leq t\leq T.
\end{align}
Indeed, we denote $J(t;s,x)$ is the Jacobian matrix  of $\Pi(t;s,x)$ with respect to $x$ at $(t;s,x)$. The determinant of the Jacobian matrix $J(t;s,x)$ is given by
\begin{align*}
\left\{
\begin{array}{ll}
\dfrac{\mathrm{d}}{\mathrm{d} t} \mathrm{det} \left( J(t;s,x) \right)  = \Divx \left( E(t) \wedge \dfrac{e}{B} + \rot_x \left(\dfrac{e}{B}\right)\right)(\Pi(t;s,x)) \mathrm{det} \left( J(t;s,x) \right),\\
\mathrm{det}(J(t;t,x)) =1.
\end{array}
\right.
\end{align*} 
Hence, we obtain
\[
 \mathrm{det} \left( J(t;s,x) \right) = \exp\left( -\int_{s}^{t}\rot_x\left( \dfrac{e}{B}\right)(\Pi(\theta;s,x))\cdot E(\theta,\Pi(\theta;s,x))\mathrm{d} \theta\right).
\]
Integrating  \eqref{SolChaPi} with respect to $x$ and changing the variable $x$ to $\Pi(t;0,x)$, we obtain \eqref{ConserVolu}.

\textbf{A priori estimates}. We establish here a priori estimates for the solution $n(t, x)$ provided by \eqref{SolChaPi} and its derivative.\\
\textbf{The bound in $L^\infty(0,T;W^{1,\infty}(\R^2\times\T^1))$ of the solutions.}
We have the following bounds for the concentration $n$:
\begin{equation}
\label{NormInfty}
\sup_{t\in[0,T]}\|n(t)\|_{L^\infty(\R^2\times\T^1)} \leq  \|n_{\mathrm{in}}\|_{L^\infty(\R^2\times\T^1)}  \exp(C_B T\sup_{t\in [0,T]}\|E(t)\|_{L^\infty}),
\end{equation}
\begin{equation}
\label{GradNormInfty}
\begin{split}
\|\nabla_x n(t)\|_{L^\infty(\R^2\times\T^1)} \leq &( \|n_{\mathrm{in}}\|_{L^\infty} +  \exp( C_B T(1 + \sup_{t\in[0,T]}\|E(t)\|_{W^{1,\infty}}) )  \|\nabla_x n_{\mathrm{in}}\|_{L^\infty} )\\
&\exp(C_B T(1+\sup_{t\in[0,T]}\|E(t)\|_{W^{1,\infty}})),
\end{split}
\end{equation}
where $C_B$ is a positive constant depending on $B$.

The estimate \eqref{NormInfty} is standard and follows directly from the formula \eqref{SolChaPi}.
We will now prove \eqref{GradNormInfty}. To do this, we estimate $\sup_{t\in [0,T]}\|\nabla_x \Pi(0;t,\cdot)\|_{L^\infty}$. By differentiating with respect to $x$ in \eqref{equ:EquChaLim}, and applying Grönwall's inequality, we obtain
\begin{equation}
\label{GradCharPi}
\|\nabla_x \Pi(0;t,\cdot)\|_{L^\infty} \leq \exp( (1+ C_B) t(1 + \sup_{t\in[0,T]}\|E(t)\|_{W^{1,\infty}}) ),\,\, t\in[0,T].
\end{equation}
Next, by taking the derivative with respect to $x$ in \eqref{SolChaPi}, and using the inequality \eqref{GradCharPi}, we obtain \eqref{GradNormInfty} through straightforward computations.\\
\textbf{The bound in $L^\infty(0,T;W^{1,1}(\R^2\times\T^1))$ of the solutions}. We have
\begin{equation}
\label{NormL1}
\|n(t)\|_{L^1} = \| n_{\mathrm{in}}\|_{L^1},\,\, t\in[0,T],
\end{equation}
\begin{equation}
\label{GradNormL1}
\|\nabla_x n(t)\|_{L^1} \leq \exp( C_B t(1 + \sup_{t\in[0,T]}\|E\|_{W^{1,\infty}}))( \|\nabla n_{\mathrm{in}}\|_{L^1} +  t C_B \sup_{t\in[0,T]}\|E(t)\|_{W^{1,\infty}} \|n_{\mathrm{in}}\|_{L^1} ),
\end{equation}
where $C_B$ is a positive constant depending on $B$. The estimate \eqref{GradNormL1} is derived by first taking the derivative with respect to $x$ in \eqref{SolChaPi}, then integrating with respect to $x$ and changing the variable $x$ to $\Pi(t;0,x)$. Finally, we apply the inequality \eqref{GradCharPi}.


\textbf{Local existence of regular solutions}.
We define the following set of the electric fields:
\[
\Sigma : = \left\{ E\in L^\infty(0,T;W^{1,\infty}(\R^2\times\T^1)): \sup_{[0,T]}\|E(t)\|_{L^\infty}\leq M_1,\quad \sup_{[0,T]}\|\nabla_x E(t) \|_{L^\infty} \leq M_2  \right\},
\]
where $M_i, i=1, 2$ are two constants to be fixed later. Assume an electric field $E$ in $\Sigma$. We consider the solution by characteristic of the equation \eqref{equ:EquCylinModLimBis} on $\R^2\times\T^1$, corresponding to the electric field $E$ and denote by $n^{E}$ which is given by the formula \eqref{SolChaPi}. First, we will show that $n^E$ satisfies \eqref{equ:constraint_n} when it holds at initial time.
\begin{pro} Assume that $n_\mathrm{in}$ satisfies $H1$ and $H2$, and that an electric field $E\in \Sigma$ is given. Then the solution $n^E$ of \eqref{equ:EquCylinModLimBis}, as given by \eqref{SolChaPi}, satisfies the condition \eqref{equ:constraint_n}.
\end{pro}
\begin{proof}
For any $\nu>0$, there is a unique solution of
\begin{align}
\label{equ:modifies_n}
\left\{
\begin{array}{ll}
\partial_t n^\nu + \left( E \wedge \dfrac{e}{B} \right)\cdot\nabla_x n^\nu +  \rot_x \left(\dfrac{e}{B} \right)\cdot \nabla_x n^\nu - \rot_x\left( \dfrac{e}{B} \right)\cdot E n^\nu + \dfrac{Be}{\nu}\cdot\nabla_x n^\nu =0,\\
n^\nu (0,x) = n_\mathrm{in}(x).
\end{array}
\right.
\end{align}
The solution is given by
\[
n^\nu(t,x) = n_\mathrm{in}(Z^\nu(0;t,x))\exp\left( \int_{0}^{t}\rot_x\left(\dfrac{e}{B} \right)(Z^\nu(s;t,x))\cdot E(s,Z^\nu(s;t,x))\mathrm{d} s \right),
\]
where $Z^\nu$ are the characteristics corresponding to the field $E\wedge \frac{e}{B} + \rot_x\left( \frac{e}{B}\right)+ \frac{Be}{\nu}$. Using \eqref{NormInfty} and \eqref{NormL1}, we obtain that
$
\|n^\nu\|_{L^\infty(0,T;L^q(\R^2\times\T^1))}
$
is uniformly bounded with respect to $\nu$, for $q\in [1,\infty]$. Therefore, we can extract a sequence $\nu_k$ converging towards $0$ such that $n^{\nu_k} \rightharpoonup \bar n$ weakly $\star$ in $L^\infty(0,T;L^q(\R^2\times\T^1))$ for some function $\bar n\in L^\infty(0,T;L^q(\R^2\times\T^1))$. Multiplying \eqref{equ:modifies_n} by $\nu_k$ and passing to the limit as $k\to\infty$ in the sense of distributions with the test function $\psi\in C^\infty_c([0,T)\times\R^2\times\T^1)$, one easily gets that $\bar n(t)\in\mathrm{ker}(Be\cdot\nabla_x)$, $t\in[0,T]$. If $\psi\in \mathrm{ker}(Be\cdot\nabla_x)$, the singular term in \eqref{equ:modifies_n} vanishes in the distributional sense, and by passing to the limit for $k\to +\infty$, we deduce that $\bar n$ satisfies \eqref{equ:EquCylinModLimBis} in the distributional sense with the test fuction in $C^\infty_c\cap \mathrm{ker}(Be\cdot\nabla_x)$. Since $n^E$ is the unique solution of \eqref{equ:EquCylinModLimBis}, thus, we have $n^E =\bar n$ and $n^E$ verifies \eqref{equ:constraint_n}.
\end{proof}

Now, we construct the following map $\calF$ on $\Sigma$, whose fixed point gives the solution of the system \eqref{equ:EquCylinModLim}-\eqref{equ:PoiCylinLim}, at least locally in time such solutions exist
\begin{equation*}
\label{MapFixPoint}
E \to \calF(E) = \int_{\R^2\times\T^1}{\nabla_x\Xi(x-y)n^E(t,y)}\mathrm{d}y.
\end{equation*}
We will prove that the map $\calF$ is left invariant on the set $\Sigma$ for a convenient choice of the constants $M_1$ and $M_2$. Then, we aim to establish an estimate such as:
\begin{equation}
\label{MapContract}
\|\calF E(t) - \calF\tilde{E}(t)\|_{L^{\infty}} \leq C_T \int_{0}^{t} \|E(s) -\tilde{E}(s)\|_{L^\infty}\mathrm{d} s,\quad E, \tilde{E}\in\Sigma,\, t\in[0,T],
\end{equation}
for some constant $C_T$, not depending on $E$ and $\tilde{E}$. After that, the existence of the system \eqref{equ:EquCylinModLim}-\eqref{equ:PoiCylinLim} immediately, based on the construction of an iterative method for $\calF$.
\begin{lemma}
\label{ClosedSet}
There exist positive constants $M_1$,  $M_2$ and $T=T(M_1, M_2)$ such that $\calF(\Sigma) \subset \Sigma$.
\end{lemma}
\begin{proof}
Let $E\in \Sigma$. Thanks to Lemma \ref{FirstEstiEle} and the formulas \eqref{NormInfty} and \eqref{NormL1}, we have
 \begin{align*}
 \|\calF(E)(t,\cdot) \|_{L^\infty} &\leq C( \|n_{\mathrm{in}}\|_{L^\infty}  \exp(C_B T\sup_{t\in [0,T]}\|E(t)\|_{L^\infty})  + \|n_{\mathrm{in}}\|_{L^1} )\\
 &\leq C(\|n_{\mathrm{in}}\|_{L^\infty} + \|n_{\mathrm{in}}\|_{L^1}) \exp(C_B T\sup_{t\in [0,T]}\|E(t)\|_{L^\infty} +1).
 \end{align*}
Here, we fix $M_1$ as a constant such that $C \exp(2){(\|n_{\mathrm{in}}\|_{L^\infty} + \|n_{\mathrm{in}}\|_{L^1})} \leq M_1$, and we choose $T = \frac{1}{C_B( M_1 +M_2)}$. Hence, we obtain
$
\sup_{t\in[0,T]}\|\calF(E)(t,\cdot) \|_{L^\infty} \leq  M_1.  
$
Consequently, the bound of $L^\infty$ norm for the density $n(t)$ in \eqref{NormInfty} becomes
\begin{equation}
\label{NormInftyBis}
\| n(t)\|_{L^\infty} \leq \exp(1) \|n_{\mathrm{in}}\|_{L^\infty}.
\end{equation}
It remains to estimate $\| \nabla_x \calF(E)(t,\cdot)\|_{L^\infty}$. Thanks to Lemma \ref{SecondEstEle}, we need to estimate $\ln^+ (\|\nabla_x n(t)\|)$. Using the formula \eqref{GradNormInfty}, we have
\begin{align*}
\ln^+(\| \nabla_x n(t) \|_{L^\infty}) 
&\leq \ln^+ (\|n_{\mathrm{in}}\|_{W^{1,\infty}}(1+ \exp( C_B T(1 + \sup_{t\in[0,T]}\|E(t)\|_{W^{1,\infty}}) )) )\\
&\quad + C_B T (1+\sup_{t\in[0,T]}\|E(t)\|_{W^{1,\infty}})\\
&\leq \ln^+(\|n_{\mathrm{in}}\|_{W^{1,\infty}})+ 1 + 2C_B T(1+\sup_{t\in[0,T]}\|E(t)\|_{W^{1,\infty}}).
\end{align*}
Thus, together with \eqref{NormInftyBis}, we deduce that
\begin{align*}
\| \nabla_x \calF(E)(t,\cdot)\|_{L^\infty} &\leq 2C(1+ \exp(1) \|n_{\mathrm{in}}\|_{L^\infty}(2+ \ln^+(\|n_{\mathrm{in}}\|_{W^{1,\infty}}))+ \|n_{\mathrm{in}}\|_{L^1})\\ &\quad (1+ C_B T \sup_{t\in[0,T]}\|E(t)\|_{W^{1,\infty}}).
\end{align*}
Here, we fix $M_2$ as a constant such that $2C(1+ \exp(1)\|n_{\mathrm{in}}\|_{L^\infty}(2+ \ln^+(\|n_{\mathrm{in}}\|_{W^{1,\infty}}))+ \|n_{\mathrm{in}}\|_{L^1}) \leq \frac{M_2}{2}$ and we take $T = \frac{1}{C_B( M_1 +M_2)}$. Hence, we get
$
\| \nabla_x \calF(E)(t,\cdot)\|_{L^\infty} \leq \dfrac{M_2}{2}2 = M_2.
$
\end{proof}

We establish the inequality \eqref{MapContract}. Let us consider $E$,  $\tilde{E}\in\Sigma$ and denote by $n^{E}$ and $\tilde{n}^{\tilde{E}}$ the characteristic solutions of \eqref{equ:EquCylinModLimBis} with the same initial data $n_\mathrm{in}$, respectively corresponding to the electric fields $E$ and $\tilde{E}$. It can be easily seen from Lemma \ref{FirstEstiEle} that
\[
\| \calF(E)(t) -\calF(\tilde{E})(t)\|_{L^\infty} \leq C (\|n^E(t)-\tilde{n}^{\tilde{E}}(t)\|_{L^\infty} + \|n^E(t) - \tilde{n}^{\tilde{E}}(t)\|_{L^1}),
\]
where the constant $C$ is not depend on $E$ and $\tilde{E}$. Thus, the inequality \eqref{MapContract} follows from the application of thefollowing lemmas, whose proofs are similar to Lemmas 5.4 and 5.5 in \cite{BosTuan}, and are left to the reader.
\begin{lemma}
\label{DiffNormInfty}
We have 
\[
\|n^E(t)-\tilde{n}^{\tilde{E}}(t)\|_{L^\infty} \leq C \int_{0}^{t}\| E(s) - \tilde{E}(s)\|_{L^\infty}\mathrm{d} s ,
\]
for some positive constant $C$, not depending on $E,\tilde{E}$.
\end{lemma}
\begin{lemma}
\label{DiffNormL1}
We have 
\[
\|n^E(t)-\tilde{n}^{\tilde{E}}(t)\|_{L^1} \leq C \int_{0}^{t}\| E(s) - \tilde{E}(s)\|_{L^\infty}\mathrm{d} s ,
\]
for some positive constant $C$, not depending on $E,\tilde{E}$.
\end{lemma}

\textbf{Uniqueness of regular solutions}. The uniqueness of regular solution $n(t, x)$, which belongs to $L^\infty(0,T; W^{1,1}(\R^2\times\T^1 ) \cap W^{1,\infty}(\R^2\times\T^1 ))$, is immediately derived from the
inequality \eqref{MapContract} and Gronwall’s inequality.

Based on the preious arguments, we establish the following result:
\begin{pro}
Assume that the initial condition $n_{\mathrm{in}}$ satisfies the hypotheses $H1$ and $H2$. There exists $T>0$ and a local time strong solution $(n,E)$ on $[0,T]$ for the limit model \eqref{equ:EquCylinModLim}-\eqref{equ:PoiCylinLim} with the initial data $n_\mathrm{in}$. The solution is unique and satisfies
\begin{align*}
n\geq 0,\,\, n\in L^{\infty}(0,T;W^{1,\infty}(\R^2\times\T^1))\cap L^\infty(0,T;W^{1,1}(\R^2\times\T^1)),\\
E\in L^\infty(0,T;W^{1,\infty}(\R^2\times\T^1)).
\end{align*} 
\end{pro}

\end{document}